\documentclass[11pt, a4paper, reqno]{amsart}

\usepackage{graphicx}
\usepackage{amsmath}
\usepackage{amssymb}
\usepackage{amsthm}
\usepackage[utf8]{inputenc}
\usepackage[T1]{fontenc}
\usepackage{amsfonts}
\usepackage{mathtools}
\usepackage{mathrsfs}
\usepackage{enumitem}
\usepackage{xfrac}
\usepackage{bbm}
\usepackage{xcolor}
\usepackage{lmodern}
\usepackage{xparse}
\usepackage[
backend=biber,
style=alphabetic,
sorting=nyt
]{biblatex}
\usepackage{geometry}
\usepackage{blindtext}
\usepackage{stmaryrd} 
\usepackage{url}
\usepackage{bbold}
\usepackage{tikz}\usetikzlibrary{arrows.meta, cd, decorations.markings, 3d, shapes.geometric}
\usepackage{pgfplots}\pgfplotsset{compat=1.18}
\usepackage{standalone}
\usepackage[hidelinks]{hyperref}
\usepackage[title]{appendix}
\usepackage{cleveref}
\usepackage{comment}

\title{An Abstract Perturbation Theorem for Compact Moduli Spaces}
\author[Peter Albers, Irene Seifert, Tom Stalljohann]{Peter Albers$^\ast$, Irene Seifert, Tom Stalljohann$^\ast{}^\ast$}
\thanks{$\ast$ Universit\"at Heidelberg, \url{palbers@mathi.uni-heidelberg.de}}
\thanks{$\ast\ast$ Universit\"at Heidelberg, \url{tstalljohann@mathi.uni-heidelberg.de}}
\date{2026}

\newtheorem{theorem}{Theorem}[section]
\newtheorem*{theorem*}{Theorem}
\newtheorem{MainThm}{Theorem}

\newtheorem{lemma}[theorem]{Lemma}
\newtheorem*{lemma*}{Lemma}
\newtheorem{corollary}[theorem]{Corollary}
\newtheorem*{corollary*}{Corollary}
\newtheorem{MainCor}[MainThm]{Corollary}
\newtheorem{proposition}[theorem]{Proposition}
\newtheorem*{proposition*}{Proposition}

\newtheorem*{claim}{Claim}

\theoremstyle{definition}
\newtheorem{definition}[theorem]{Definition}
\newtheorem*{definition*}{Definition}
\newtheorem*{convention}{Convention}

\theoremstyle{remark}
\newtheorem{remark}[theorem]{Remark}

\newtheorem{example}[theorem]{Example}

\newenvironment{proofClaim}{ \proof[Proof of the claim]}{\endproof}

\newcommand{\Forall}[0]{\forall\,}

\newcommand{\R}[0]{\mathbb{R}}
\newcommand{\C}[0]{\mathbb{C}}
\newcommand{\Z}[0]{\mathbb{Z}}
\newcommand{\N}[0]{\mathbb{N}}

\newcommand{\bbS}[0]{\mathbb{S}}
\renewcommand{\i}[0]{\mathrm{i}}
\newcommand{\ball}[0]{\mathbb{B}}
\newcommand{\ballclosed}{\Bar{\mathbb{B}}}
\newcommand{\eps}[0]{\varepsilon}
\newcommand{\comma}[0]{\, , \, }
\newcommand{\Cinfty}[0]{C^{\infty}}
\newcommand{\Cinftyloc}[0]{\Cinfty_{\mathrm{loc}}}

\newcommand{\ie}[0]{i.e.\ }

\newcommand{\cf}[0]{cf.\ }

\newcommand{\set}[2]{\left\{#1 \,\big| \, #2 \right\}}
\newcommand{\Bigset}[2]{\left\{#1 \,\Big| \, #2 \right\}}

\newcommand{\image}[0]{\mathrm{im}}
\newcommand{\supp}[0]{\mathrm{supp}}
\newcommand{\id}[0]{\mathrm{id}}

\newcommand{\Crit}[0]{\mathrm{Crit}}
\newcommand{\action}[1]{\mathcal{A}_{#1}}
\newcommand{\actionH}[0]{\action{H}}
\newcommand{\loopspace}[0]{\mathscr{L}M}
\newcommand{\Hparam}[0]{H^{\lambda}_{s,t}}

\newcommand{\bbmE}[0]{\mathbbm{E}}
\newcommand{\bbmF}[0]{\mathbbm{F}}

\newcommand{\moduli}[0]{\mathcal{M}}
\newcommand{\Bcal}[0]{\mathcal{B}}
\newcommand{\Ecal}[0]{\mathcal{E}}
\newcommand{\Ucal}[0]{\mathcal{U}}
\newcommand{\Vcal}[0]{\mathcal{V}}
\newcommand{\Wcal}[0]{\mathcal{W}}
\newcommand{\Dcal}[0]{\mathcal{D}}
\newcommand{\OcalE}[0]{\mathcal{O}_{\Ecal}}
\newcommand{\BcalHat}[0]{\widehat{\Bcal}}
\newcommand{\EcalHat}[0]{\widehat{\Ecal}}

\newcommand{\Fsection}[0]{\mathcal{F}}
\newcommand{\ZeroFsection}[0]{S_{\Fsection}}
\newcommand{\Funiv}[0]{\Fsection_{\mathrm{un}}}
\newcommand{\ev}[0]{\mathrm{ev}}
\newcommand{\Bplus}[0]{B\textsuperscript{+}}
\newcommand{\Bres}[0]{B_{\mathrm{res}}}
\newcommand{\Fsectionparam}[1]{\Fsection_{#1}}
\newcommand{\Dvert}[1]{D^{\mathrm{v}}_{#1}}

\begin{document}

\begin{abstract}
    Given a compact zero set of a Fredholm section, our theorem guarantees the existence of a perturbed compact smooth manifold nearby, leaving the original zero set unaltered wherever transversality is already achieved. Such abstract perturbations allow for typical cobordism arguments. We illustrate this by re-proving a well-known theorem of Schwarz asserting the existence of critical points of the Hamiltonian action functional of different action values on symplectically aspherical manifolds.
\end{abstract}

\maketitle

\section{Introduction}
\label{sec: Introduction}

\noindent The study of moduli spaces usually encompasses a rather intricate transversality analysis, in which one establishes transversality of a given section to the zero section, whose zero set constitutes the moduli space, for generic enough auxiliary data. Transversality and the implicit function theorem then yield a smooth manifold structure on the moduli space for generic data. The latter is vital for various cobordism arguments: from simply counting boundary points to defining chain maps or homotopies via counting elements of some moduli space. For instance, this kind of transversality analysis is ubiquitous in the study of pseudoholomorphic curves in symplectic geometry and Floer theory.
Although transversality arguments typically follow familiar patterns, they nonetheless involve some detailed analysis depending on the explicit setting.

The purpose of this paper is to provide a perturbation result in a general framework, which one can simply use as a black box, under the crucial additional assumption that the moduli space is compact. Instead of generic auxiliary data, we perturb the bundle section abstractly into a section that is transverse to the zero section. If the original section is already transverse somewhere, we can keep it unaltered there. Our result is strongly inspired by similar results in polyfold theory \cite{Polyfolds_HWZ}.

More precisely, in Section \ref{subsec: B-Pairs}, we introduce the notion of a \textit{B-bundle pair} $(\Ecal,\Ecal^1)$. Essentially, this is just a pair of Banach bundles $\Ecal \rightarrow \Bcal$ and $\Ecal^1 \rightarrow \Bcal$ so that each fiber $\Ecal_x^1$ is a dense linear subspace of $\Ecal_x$ with compact inclusion $\Ecal_x^1\hookrightarrow \Ecal_x$ and so that trivializations for $\Ecal$ restrict to trivializations for $\Ecal^1 \,$. The prototypical example of a B-bundle pair is a bundle $\Ecal$ of sections over a Banach manifold of maps $\Bcal$, with the fibers of $\Ecal^1$ consisting of sections of higher regularity.
For such a B-bundle pair $(\Ecal,\Ecal^1)$ and a Fredholm section $\Fsection : \Bcal \rightarrow \Ecal$ whose zero set $\ZeroFsection =  \{\Fsection = 0\}$ is compact, we find finitely many sections $s_1,\ldots , s_n : \Bcal \rightarrow \Ecal$ and a residual subset $\Bres$ of a sufficiently small ball $\ball_\eps \subseteq \R^n$ so that for every parameter $\lambda = (\lambda_1,\ldots ,\lambda_n) \in \Bres$ the perturbed section 
\[
\Fsection_\lambda = \Fsection + \sum_{i=1}^n \lambda_i \, s_i \, : \, \Bcal \rightarrow \Ecal
\]
is transverse to the zero section and, crucially, has compact zero set $S_{\Fsection_\lambda} = \{\Fsection_\lambda = 0\}$ as well. Moreover, if $\Fsection$ is  already transverse to the zero section on some closed subset $\Dcal \subseteq \Bcal \,$, we can arrange that the sections $s_i , i=1,\ldots , n \,$, vanish on $\Dcal \,$, so that $S_{\Fsection}$ and $S_{\Fsection_\lambda}$ coincide on $\Dcal \,$. The precise result is stated in Theorem \ref{mthm: perturbation thm}.
\\

\noindent Abstract perturbation results of similar flavor are well-known and constitute a cornerstone of the part of polyfold theory dealing with sc-Fredholm sections. Polyfold theory was developed by Hofer--Wysocki--Zehnder in an attempt to unify various recurring concepts in Floer theory into an abstract framework. More precisely, our theorem is analogous to \cite[Thm.~5.3.5 and 5.3.10]{Polyfolds_HWZ}. In fact, some proofs below are inspired by the ones in \cite{Polyfolds_HWZ}. That said, we point out that in practice it is quite cumbersome to verify the rather long list of assumptions for the theorems in polyfold theory (in particular the existence of a contraction germ). However, in a sufficiently tame setting many of these underlying assumptions are automatically satisfied, as also pointed out by Wehrheim \cite{Wehrheim_Fredholm_Notions}. It is the goal of this paper to provide a self-contained reference for abstract perturbations of classically smooth sections, where the full power of polyfold theory is not necessary.
\\

\noindent We will illustrate how to apply our abstract perturbation result by re-proving a well-known theorem by Schwarz, first stated in \cite{Schwarz-Action_Spectrum}: Suppose $(M,\omega)$ is a closed (compact and without boundary) symplectic manifold that is symplectically aspherical, that is $\int_{\bbS^2} f^*\omega = 0$ for all $f \in \Cinfty(\bbS^2, M) \,$. Given a non-autonomous Hamiltonian $H : M \times \bbS^1 \rightarrow \R \,$, its action functional is defined on the space of contractible loops $\loopspace = C^{\infty}_{\mathrm{contr}}(\bbS^1,M)$ by
\begin{align*}
    \actionH : \loopspace \rightarrow \R \comma \quad \actionH(x) := -\int_{\mathbbm{D}^2} \overline{x}^* \omega + \int_0^1 H_t(x(t)) \, dt \,\, ,
\end{align*}
where $\overline{x} : \mathbbm{D}^2 \rightarrow M$ is any choice of filling disk for $x$, \ie $\overline{x}|_{\bbS^1} = x \,$. Defining the Hamiltonian vector field $X_{H_t}$ by $\omega(X_{H_t}, \,\cdot\,) = -dH_t \,$, the critical points of $\actionH$ are precisely the contractible one-periodic solutions of $\dot x(t) =X_{H_t}(x(t)) \,$, the set of all we also denote by $\mathcal{P}_{\mathrm{contr}}(H) \,$. The Hamiltonian flow $\Phi_H^t : M \rightarrow M \comma t \in \R \,$, is defined via
\begin{align*}
 \begin{cases}
    \frac{d}{dt} \Phi_H^t = X_{H_t} \circ \Phi_H^t  \\
     \Phi_H^0= \id_{M} 
 \end{cases}   
\end{align*}
and the fixed points of $\Phi_H^1$ are precisely the initial points of one-periodic Hamiltonian orbits. Schwarz's theorem now asserts the following.

\begin{MainThm}[\cite{Schwarz-Action_Spectrum}]
\label{mthm: Schwarz theorem}
Given a closed symplectic manifold $(M,\omega)$ which is symplectically aspherical and a smooth Hamiltonian $H : M \times \bbS^1 \rightarrow \R$ for which the time-$1$ map $\Phi_H^1 : M \rightarrow M$ is different from the identity. Then there exist $x,y \in \mathcal{P}_{\mathrm{contr}}(H) = \Crit(\actionH) $ with
\[
\actionH(x) \not= \actionH(y) \,\, .
\]
\end{MainThm}

\noindent Schwarz's original proof relied on the construction of a certain spectral selector in Floer homology. Our proof is based on an argument by McDuff--Salamon \cite[Thm.~9.1.6]{J_holomorphic_curves}: The idea is, for a fixed sequence $(R_n)_n \subseteq \R_{\geq 0}$ tending to infinity, to find finite-energy cylinders $u_n : \R \times \bbS^1 \rightarrow M$ that are solutions of the Floer equation for $H$ on $[-R_n,R_n] \times \bbS^1$ and $J$-holomorphic curves outside $(-R_n - 1,R_n+1) \times \bbS^1 \,$. Taking the $\Cinftyloc$-limit as $n \rightarrow \infty$ of a converging subsequence of the $(u_n)_n$ gives a Floer cylinder $u$ for $H$, whose endpoints $x = u(-\infty,\,\cdot\,)$ and $y = u(+\infty,\,\cdot\,)$ are the desired periodic orbits. The crucial step is showing that such solutions $u_n$ exist. To this end, McDuff--Salamon make use of the Gromov--Witten invariants (which in particular again require elaborate transversality arguments). In contrast, we invoke our abstract perturbation result. This approach has already been carried out in the second author's unpublished master's thesis.

\subsection{Structure of the paper}
\label{subsec: Structure of paper}
In Section \ref{sec: Abstract Perturbation Thm} we present the abstract perturbation result. After introducing the necessary definitions, we state the main results (Theorem \ref{mthm: perturbation thm} and Corollary \ref{mcor: perturbation thm}) in Subsection \ref{subsec: Abstract Perturbation Theorem}. The perturbation theorem (Theorem \ref{mthm: perturbation thm}) itself will be proven in Section \ref{sec: Proof Abstract Perturbation Thm}.

As an application, in Section \ref{sec: Application Schwarz thm} we prove Schwarz's theorem (Theorem \ref{mthm: Schwarz theorem}); the necessary preliminaries are recalled in Subsection \ref{subsec: Preliminaries}. 

Finally, in Section \ref{sec: Perturbing Evaluation Maps}, we extend the abstract perturbation result by also perturbing evaluation maps.
We illustrate the main result of the section (Theorem \ref{mthm: evaluation maps - cobordism}), regarding evaluation maps for cobordisms, by giving a slightly modified proof of the central lemma used to prove Schwarz's theorem.

Two appendices are added: The first contains some rather elementary lemmata from nonlinear functional analysis with the aim of making this paper more self-contained. The second deals with the explicit B-bundle pair construction in the proof of Schwarz's theorem.

\subsection{Acknowledgments}
\label{subsec: Acknowledgments}

We thank Alberto Abbondandolo for inspiring discussions. We acknowledge funding by the Deutsche Forschungsgemeinschaft (DFG, German Research Foundation) through Germany’s Excellence Strategy EXC-2181/1 - 390900948 (the Heidelberg STRUCTURES Excellence Cluster), the Transregional Collaborative Research Center CRC/TRR 191 (281071066) and the Research Training Group RTG 2229 (281869850).

\section{An Abstract Perturbation Theorem}
\label{sec: Abstract Perturbation Thm}

\subsection{Conventions}
\label{subsec: Conventions}

\noindent We mostly follow Lang \cite{Lang_Differential_Geometry} regarding the definitions of a Banach manifold and Banach bundle. For the sake of precision, we nevertheless spell them out below.

By a \textit{Banach manifold} we always mean a Hausdorff topological space equipped with a smooth structure, \ie an equivalence class of a smooth atlas. Here, an atlas is smooth if the transition maps of the charts it contains are smooth ($C^\infty$) maps between subsets of Banach spaces and two smooth atlases are equivalent if their union is a smooth atlas. Without further qualification, all Banach manifolds are allowed to have boundary. A \textit{(smooth) chart} $(\Ucal,\varphi)$ of a Banach manifold $\Bcal$ is an element of an atlas contained in the given smooth structure. It is a homeomorphism $\varphi : \Ucal \rightarrow \Vcal$ from an open subset $\Ucal \subseteq \Bcal$ to a relatively open subset $\Vcal \subseteq \lambda^{-1}(\R_{\geq 0}) \,$, where $\lambda : \bbmE \rightarrow \R$ is a continuous linear functional on some Banach space $\bbmE \,$. The set $\lambda^{-1}(\R_{\geq 0}) \subseteq \bbmE$ is called \textit{generalized half-space in $\bbmE$}
and \textit{(proper) half-space} if $\lambda$ is a non-zero functional. 

A Banach manifold $\Bcal$ is \textit{modeled on the Banach space $\bbmE$} if its smooth structure contains an atlas all whose charts map to relatively open subsets of generalized half-spaces in $\bbmE \,$. A connected Banach manifold is always modeled on some Banach space, which moreover is unique up to toplinear isomorphism.  The smooth structure on $\Bcal$ allows one to speak of $C^k$-differentiable maps for every $k \in \N_{\geq 0} \cup \{\infty\} \,$. Without further qualification, all maps between Banach manifolds are assumed to be smooth. 

A Banach manifold $\Bcal$ is called \textit{$n$-dimensional manifold} if it is modeled on the Banach space $\R^{n} $ and additionally is second-countable. It is called \textit{finite-dimensional manifold} if it is an $n$-dimensional manifold for some $n \in \N_{0} \,$.

\begin{definition}
\label{def: bump functions}
A Banach manifold $\Bcal$ \textit{admits $C^k$-bump functions} (for $k \in \N_{\geq 0} \cup \{\infty\} $) if for every $x \in \Bcal$ and every open neighborhood $\Ucal \subseteq \Bcal$ of $x$ there exists some $f \in C^k(\Bcal,[0,1])$ with $f(x)=1$ and $\supp(f) \subseteq \Ucal \,$.
\end{definition}

\begin{remark}
\label{rmk: Banach mfds - topological properties}
The following topological properties of Banach manifolds will be relevant subsequently.
\begin{enumerate}[label=\roman*)]
    \item\label{it: Banach mfds topological - first countable} Banach manifolds are first-countable. Hence every compact subset of a Banach manifold is sequentially compact.
    \item\label{it: Banach mfds topological - locally R^n and compact} A Banach manifold modeled on the Banach space $\R^n$ for which there exists a countable atlas is second-countable and thus an $n$-dimensional manifold. In particular, compact Banach manifolds modeled on $\R^n$ are $n$-dimensional manifolds.
    \item\label{it: Banach mfds topological - second-countable} If every connected component of a Banach manifold $\Bcal$ is second-countable or (more generally) admits a countable atlas, then $\Bcal$ is metrizable, see \cite[Cor.~1 and 2]{Palais_ParacompactBanachMfds}.
    \item\label{it: Banach mfds topological - regular} Recall that a topological space is called \textit{regular} if every point has a neighborhood basis consisting of closed sets. Every metrizable space is a regular space. A Banach manifold is a regular space if and only if it admits $C^0$-bump functions.
\end{enumerate}
\end{remark}

\begin{remark}
\label{rmk: Banach mfds - bump functions}
The following facts about bump functions are immediate from the definition or are asserted in \cite{Lp_smooth_bump_fcts} and \cite{Bonic_Frampton_Bump_Fcts}. Let $k \in \N_{\geq 0} \cup \{\infty\}$ be fixed.
\begin{enumerate}[label=\roman*)]
\item\label{it: Banach mfds bump fcts - local property} Let $\Bcal$ be a regular topological space and a Banach manifold, modeled on the Banach space $\bbmE \,$. Then  $\Bcal$ admits $C^k$-bump functions if and only if $\bbmE$ admits $C^k$-bump functions.
    \item\label{it: Banach mfds bump fcts - Banach space} A Banach space $\bbmE$ admits $C^k$-bump functions if and only if there exists a non-vanishing function $f \in C^k(\bbmE,\R)$ with bounded support. If for some $p > 0$ and some bounded closed subset $B \subseteq \bbmE$ the $p$-th power of the norm $\lVert \cdot \rVert_{\bbmE}^p $ restricts to a $C^k$-differentiable map $ \bbmE \backslash B \rightarrow \R_{\geq 0} \,$, then $\bbmE$ admits $C^k$-bump functions.
\end{enumerate}
\end{remark}

A \textit{(Banach) submanifold} of a Banach manifold $\Bcal$ is a subset $\mathcal{S} \subseteq \Bcal$ for which there exists a smooth structure on $\mathcal{S}$, where $\mathcal{S}$ is endowed with the subspace topology, so that the inclusion $\mathcal{S} \hookrightarrow \Bcal$ is a smooth immersion. If such a smooth structure exists, it is unique and we regard $\mathcal{S}$ with this smooth structure as a Banach manifold in its own right.

Recall that a \textit{Banachable space} is a topological vector space for which there exists a complete norm inducing the given topology.

A \textit{Banach bundle} consists of the following data:
\begin{itemize}
    \item A smooth surjective map $\pi : \Ecal \rightarrow \Bcal$ between Banach manifolds $\Ecal$ (the \textit{total space}) and $\Bcal$ (the \textit{base}).
    \item On each fiber $\Ecal_x := \pi^{-1}(x)$ the structure of a Banachable space.
    \item The equivalence class of a smooth bundle atlas. 
\end{itemize}
We briefly explain the last point. A \textit{local trivialization} $\Phi: \Ecal|_\Ucal := \pi^{-1}(\Ucal) \overset{\sim}{\rightarrow} \Ucal \times \bbmF $ over the open subset $\Ucal \subseteq \Bcal$ with Banach space $\bbmF$ as model fiber is a diffeomorphism with $\mathrm{pr}_1 \circ \Phi = \pi |_{\Ecal|_\Ucal}$ and so that the fiber-restriction $\Phi_x : \Ecal_x \rightarrow \{x\} \times \bbmF \cong \bbmF$ is a toplinear isomorphism for every $x \in \Ucal \,$. A \textit{smooth bundle atlas} is a collection $\{ \Phi^{(i)} : \Ecal|_{\Ucal_i} \rightarrow \Ucal_i \times \bbmF^{(i)} \}_{i \in I}$ of local trivializations so that $\{\Ucal_i\}_{i \in I}$ is an open covering of $\Bcal$ and the induced maps
\[
\Ucal_i \cap \Ucal_j \rightarrow \mathcal{L}(\bbmF^{(i)}, \bbmF^{(j)}) \comma \quad x \mapsto (\Phi_x^{(j)})^{-1} \circ \Phi_x^{(i)}
\]
are smooth whenever $\Ucal_i \cap \Ucal_j$ is non-empty. Two smooth bundle atlases are equivalent if their union is again a smooth bundle atlas.

For a Banach bundle $\Ecal$, notice that the topology of the Banachable fiber $\Ecal_x$ agrees with the subspace topology of $\Ecal_x$ in $\Ecal$. 

We will often just write $\pi: \Ecal \rightarrow \Bcal$ or $\Ecal$ for a Banach bundle. Let us stipulate that a \textit{(smooth local) trivialization} of a Banach bundle is, by default, a local trivialization which is an element of some smooth bundle atlas in the given bundle atlas equivalence class. Without further qualification, sections of Banach bundles are assumed to be smooth. The \textit{zero section} $\mathcal{O}_{\Ecal} \subseteq \Ecal$ of a Banach bundle $\Ecal$ is a closed subset of $\Ecal$ (and not a map). The \textit{zero set} of a section $\Fsection: \Bcal \rightarrow \Ecal$ will be designated by 
\[
\ZeroFsection := \Fsection \,^{-1}(\mathcal{O}_{\Ecal}) = \set{x \in \Bcal}{\Fsection(x) = 0}  \,\, .
\]
Given a section $\Fsection : \Bcal \rightarrow \Ecal $ and a Banach submanifold $\mathcal{S} \subseteq \Bcal \,$, the restriction $\Fsection|_{\mathcal{S}} : \mathcal{S} \rightarrow \Ecal|_{\mathcal{S}}$ gives a section of the restricted bundle $\Ecal|_{\mathcal{S}} \rightarrow \mathcal{S} \,$. The latter is nothing else than the pullback of $\Ecal$ along the inclusion $\mathcal{S} \hookrightarrow \Bcal \,$.

\subsection{Transversality and Fredholm Sections}
\label{subsec: Transversality and Fredholm Sections}

For our main perturbation theorem, we have to define transversality to the zero section and the notion of a Fredholm section.

Let $\pi : \Ecal \rightarrow \Bcal $ be a Banach bundle with zero section $\OcalE \subseteq \Ecal \,$. At $0_x \in \Ecal_x$ we have the canonical splitting
\[
T_{0_x} \Ecal = T_{0_x} \OcalE \oplus T_{0_x} \Ecal_x \cong T_x \Bcal \oplus \Ecal_x 
\]
and we designate by $\mathrm{pr}^{\mathrm{v}}_x : T_{0_x} \Ecal \rightarrow \Ecal_x$ the projection to $\Ecal_x$ with respect to this splitting.

\begin{definition}
\label{def: vertical differential and transversality}
Let $\Fsection : \Bcal \rightarrow \Ecal$ be a section with zero set $\ZeroFsection = \Fsection^{-1}(\OcalE) \,$.
\begin{enumerate}[label=\alph*)]
    \item For every $x \in \ZeroFsection$ the \textit{vertical differential at $x$} is 
    \[
    \Dvert{x} \Fsection := \mathrm{pr}^{\mathrm{v}}_x \circ d \Fsection(x)  \, : \,\, T_x \Bcal \rightarrow \Ecal_x \,\, ,
    \]
    where $d \Fsection(x) : T_x \Bcal \rightarrow T_{0_x} \Ecal $ denotes the usual differential of $\Fsection \,$.
    \item $\Fsection$ is \textit{transverse to the zero section} if for every $x \in \ZeroFsection$ the vertical differential $\Dvert{x} \Fsection $ is surjective and its kernel has a closed complement in $T_x \Bcal \,$.
    \item $\Fsection$ is a \textit{Fredholm section} if $\Dvert{x} \Fsection : T_x \Bcal \rightarrow \Ecal_x$ is a Fredholm operator for every $x \in \ZeroFsection \,$. If the Fredholm indices additionally satisfy
    \[
    \mathrm{ind}(\Dvert{x} \Fsection ) = d \in \Z \qquad \Forall x \in \ZeroFsection  \,\, ,
    \]
    then $\Fsection$ is called \textit{Fredholm section of index $d$}.
\end{enumerate}
\end{definition}

\begin{remark}
\label{rmk: vertical differential}
Let $\Fsection: \Bcal \rightarrow \Ecal$ be a section of a Banach bundle $\Ecal \,$.
\begin{enumerate}[label=\roman*)]
    \item If $\Fsection$ has empty zero set $\ZeroFsection \,$, then it is transverse to the zero section and a Fredholm section of index $d$ for every $d \in \Z \,$.
    \item Given $x \in \ZeroFsection $ and a trivialization $\Phi:\Ecal|_{\Ucal} \rightarrow \Ucal \times \bbmF$ over an open neighborhood $\Ucal$ of $x \,$. Then 
    \[
    \Phi_x \circ \Dvert{x} \Fsection =  d (\mathrm{pr}_2 \circ \Phi \circ \Fsection)(x)  \, : \, T_x \Bcal \rightarrow \bbmF \,\, ,
    \]
    where $\mathrm{pr}_2 : \Ucal \times \bbmF \rightarrow \bbmF$ denotes the projection and the right-hand side maps to $\bbmF$ via the canonical identification $T_0 \bbmF \cong \bbmF \,$.
    In particular, $\Fsection|_\Ucal : \Ucal \rightarrow \Ecal|_\Ucal$ is transverse to the zero section if and only if $0 \in \bbmF$ is a regular value of $\mathrm{pr}_2 \circ \Phi \circ \Fsection|_\Ucal \,$.
    \item For a Fredholm section the kernel of the vertical differential is finite-dimensional and thus always has a closed complement.
\end{enumerate}
\end{remark}

\begin{remark}
\label{rmk: restricting sections to submfds}
Let $\Fsection : \Bcal \rightarrow \Ecal$ be a section of the Banach bundle $\Ecal \,$.
\begin{enumerate}[label=\roman*)]
    \item Suppose $\mathcal{S} \subseteq \Bcal$ is a Banach submanifold and $x_0 \in \mathcal{S}$ is a zero of $\Fsection \,$. Then
    \[
    (\Dvert{x_0} \Fsection ) |_{T_{x_0} \mathcal{S}} = \Dvert{x_0} (\Fsection|_{\mathcal{S}} ) \,\, ,
    \]
    that is the restriction of the vertical differential of $\Fsection$ to $T_{x_0} \mathcal{S}$ is the vertical differential of the restricted section $\Fsection|_{\mathcal{S}} : \mathcal{S} \rightarrow \Ecal|_{\mathcal{S}} \,$.
    \item If $\Fsection$ is a Fredholm section, then also $\Fsection|_{\partial \Bcal} : \partial \Bcal \rightarrow \Ecal|_{\partial \Bcal}$ is a Fredholm section. Moreover $\mathrm{ind}(\Dvert{x} \Fsection) = 1 + \mathrm{ind}(\Dvert{x} (\Fsection|_{\partial \Bcal}))$ for every $x \in \ZeroFsection \cap \partial \Bcal \,$. This follows immediately from the splitting $T_x \Bcal \cong \R \oplus T_x (\partial \Bcal)$ and Lemma \ref{app lemma: Fredholm finite dim summands} in the appendix.
\end{enumerate}
\end{remark}

\subsection{B-Pairs and B-Bundle Pairs}
\label{subsec: B-Pairs}

We introduce the notion of a B-pair and of a B-bundle pair (the `B' stands for `Banach').

\begin{definition}[B-pair]
A \textit{B-pair} is a pair $(\bbmF, \bbmF_1)$ of Banachable spaces $\bbmF$ and $\bbmF_1 \,$, each with its own topology, so that $\bbmF_1 \subseteq \bbmF$ is a dense (in $\bbmF$) linear subspace and for which the inclusion $\bbmF_1 \hookrightarrow \bbmF $ is a compact linear operator.
\end{definition}

\noindent We will usually tacitly work with Banach spaces (that is, with a distinguished norm), but keep in mind that the notion of a B-pair is a topological one. 

\begin{definition}[B-bundle pair]
A \textit{B-bundle pair $(\Ecal, \Ecal^1)$ over $\Bcal$} consists of two Banach bundles $\pi : \Ecal \rightarrow \Bcal$ and $\pi^1:\Ecal^1 \rightarrow \Bcal$ satisfying 
 \begin{enumerate}[label=\arabic*)]
     \item $\Ecal^1 \subseteq \Ecal$ and $\pi^1 = \pi|_{\Ecal^1}$ (equivalently $\Ecal_x^1 \subseteq \Ecal_x$ for every $x \in \Bcal$),
     \item for every $x \in \Bcal$ there exists an open neighborhood $\Ucal \subseteq \Bcal$ of $x \,$, a B-pair $(\bbmF,\bbmF_1)$ and a smooth local trivialization $\Phi : \Ecal|_{\Ucal} \overset{\sim}{\rightarrow} \Ucal \times \bbmF$ of $\Ecal$ that restricts to a smooth local trivialization $\Phi : \Ecal^1|_{\Ucal} \overset{\sim}{\rightarrow} \Ucal \times \bbmF_1$ of $\Ecal^1 \,$.
 \end{enumerate}
\end{definition}

\begin{remark}
\label{rmk: B-bundle pair}
Given a B-bundle pair $(\Ecal,\Ecal^1) $ over $\Bcal \,$.
\begin{enumerate}[label=(\roman*)]
    \item For every $x \in \Bcal$ the pair of fibers $(\Ecal_x,\Ecal_x^1)$ is a B-pair.
    \item The inclusion $\Ecal^1 \hookrightarrow \Ecal$ is a smooth bundle homomorphism.
\end{enumerate}
\end{remark}

\begin{definition}[\Bplus-section]
\label{B^+ -section}
A smooth section $s : \Bcal \rightarrow \Ecal$ is called \textit{\Bplus-section} of a B-bundle pair $(\Ecal,\Ecal^1)$ if $s(\Bcal) \subseteq \Ecal^1$ and $s : \Bcal \rightarrow \Ecal^1$ is a smooth section.
\end{definition}

\subsection{The Perturbation Theorem}
\label{subsec: Abstract Perturbation Theorem}

The previously introduced notation suffices to now state our main theorem. Suppose we are given
\begin{itemize}
    \item a Banach manifold $\Bcal$ which admits smooth bump functions,
    \item a B-bundle pair $(\Ecal,\Ecal^1) $ over $\Bcal \,$.
\end{itemize}

\noindent We remind the reader that $\Bcal$ is allowed to have non-empty boundary.
We denote by $\ball^n_r \subseteq \R^n$ the open ball of radius $r$ centered at the origin. Recall that a \textit{residual subset} of a topological space is a subset containing a countable intersection of open, dense subsets. Every residual subset of a locally compact Hausdorff space (for example the ball $\ball^n_r$) is dense by the Baire category theorem.

\begin{MainThm}[Perturbation Theorem]
\label{mthm: perturbation thm}
Let $\Fsection : \Bcal \rightarrow \Ecal$ be a Fredholm section of index $d \in \Z$ with compact zero set $\ZeroFsection\,$. Let $\Dcal \subseteq \Bcal$ be a (possibly empty) closed subset of $\Bcal $ so that
\begin{enumerate}[label=(B.\arabic*)]\setcounter{enumi}{-1}
    \item\label{it: perturbation thm assumption}
    for every $x \in \Dcal \cap \ZeroFsection$ the vertical differential $\Dvert{x} \Fsection : T_x \Bcal \rightarrow \Ecal_x$ is surjective and for every $x \in \Dcal \cap \ZeroFsection \cap \partial \Bcal$ the vertical differential $\Dvert{x} (\Fsection|_{\partial \Bcal}) : T_x(\partial \Bcal) \rightarrow \Ecal_x$ is surjective.
\end{enumerate}
Then there exist finitely many \Bplus-sections $s_1,\ldots , s_n : \Bcal \rightarrow \Ecal $ with support contained in $\Bcal \backslash \Dcal $ respectively, some $\eps > 0$ and a residual subset $B_{\mathrm{res}} \subseteq \ball^n_\eps$ of the $\eps$-ball with the following significance: For every $\lambda = (\lambda_1,\ldots , \lambda_n) \in \Bres$ it holds
\begin{enumerate}[label=(B.\arabic*),resume]
    \item\label{it: perturbation thm conclusion - Fredholm} $\Fsection_{\lambda} := \Fsection + \displaystyle{\sum_{i=1}^n \lambda_i \, s_i } : \Bcal \rightarrow \Ecal$ is a Fredholm section of index $d$.
    \item\label{it: perturbation thm conclusion - transverse} $\Fsection_{\lambda}$
    is transverse to the zero section $\OcalE$ and the restriction $\Fsection_{\lambda}|_{\partial \Bcal} : \partial \Bcal \rightarrow \Ecal|_{\partial \Bcal}$ is transverse to the zero section $\mathcal{O}_{\Ecal|_{\partial \Bcal}} \,$.
    \item\label{it: perturbation thm conclusion - zero set compact mfd} The zero set $S_{\Fsection_{\lambda}}$ is a (possibly empty) compact $d$-dimensional smooth submanifold of $\Bcal$ with boundary $\partial S_{\Fsection_{\lambda}} = S_{\Fsection_{\lambda}} \cap \partial \Bcal $ and with $S_{\Fsectionparam{\lambda}} \cap \Dcal  = \ZeroFsection \cap \Dcal \,$.
\end{enumerate}
\end{MainThm}

\noindent We point out the trivial but crucial observation that a submanifold $S_{\Fsection_\lambda}$ as in \ref{it: perturbation thm conclusion - zero set compact mfd} is necessarily non-empty if $\ZeroFsection \cap \Dcal$ is non-empty. The situation in the theorem is depicted schematically in Figure \ref{fig: zero set perturbation thm}.

\begin{remark}
\label{rmk: perturbation thm}
In fact, we will prove slightly more. Namely that we can additionally arrange the following:
\begin{enumerate}[label=(B.\arabic*)]\setcounter{enumi}{3}
    \item\label{it: perturbation rmk - Fredholm}
    $\Fsection_\lambda$ is a Fredholm section of index $d$ for every $\lambda \in \ball^n_\eps \,$.
    \item\label{it: perturbation rmk - compactness}
   The set
    \begin{equation*}
        \Bigset{(x,\lambda) \in \Bcal \times \R^n}{|\lambda| \leq \eps \text{ and } \Fsection_\lambda(x) = 0}
    \end{equation*}
    is a compact subset of $\Bcal \times \R^n \,$.
\end{enumerate}
\end{remark}

\noindent An immediate consequence of Theorem \ref{mthm: perturbation thm} is the next corollary, still in the setting described above.

\begin{MainCor}
\label{mcor: perturbation thm}
Let $\Fsection : \Bcal \rightarrow \Ecal$ be a Fredholm section of index $d \in \Z$ with compact zero set. Suppose that $\Fsection|_{\partial \Bcal} : \partial \Bcal \rightarrow \Ecal|_{\partial \Bcal}$ is transverse to the zero section. Then there exists a compact $d$-dimensional smooth submanifold $\widetilde{S}$ of $\Bcal$ with boundary
\[
\partial \widetilde{S} = \widetilde{S} \cap \partial \Bcal = S_{\Fsection} \cap \partial \Bcal \,\, .
\]
\end{MainCor}

\begin{proof}
    This follows from \Cref{mthm: perturbation thm} with $\Dcal := \partial \Bcal \,$. Define $\widetilde{S} := S_{\Fsection_{\lambda}} \,$, for any choice of $\lambda$ from the residual set $B_{\mathrm{res}} \,$.
\end{proof}

\begin{figure}
    \centering
    \begin{tikzpicture}[scale=0.5, x=10cm, y=10cm, ]

\begin{scope}[fill opacity=0.05, draw opacity=1]

\filldraw[clip] plot [smooth, tension=0.6] coordinates {(-3.5,0.5) (-3,2.5) (-1,3) (1.5,3) (3.5,3.5) (5,2.5) (5,0)  (2.5,-2) (0,-2) (0,0) (-3,-1.5) (-3.5,0.5)};

\filldraw[color=black, ultra thin]  plot[smooth, tension=0.9] coordinates {(0,0) (1.5,3) (1,3) (-4,3) (-4,-3) (0,0)};


\draw[ultra thick] plot [smooth, tension = 0.9] coordinates { (-1,3) (-0.5,2.2) (-0.4,2.5) (0,2) (0.5,2.5) (1.3, 2.2)} ;

\draw[ultra thick, color=red] plot [smooth, tension=0.9] coordinates {(1.25,2.23) (1.3,2.2) (1.35,2.19) (1.4,2.5) (1.55,2) (1.7, 2.2) (2,2.4) (2.5,2.7) (3,3.1) (3.3, 3.6) (3.5, 4)};

\draw[thick, decorate, decoration={coil,segment length=10pt}] (1.3,2.2) -- (1.35,2.1) -- (1.45,2) -- (1.5,1.85) .. controls (1.55,1.85) .. (1.6,1.8)  ;
\filldraw[thick, color=black, fill opacity=1, decorate, decoration={zigzag, segment length=10pt}] (1.6,1.8) -- (1.7,1.82) -- (1.71,1.95) -- (1.62,1.92)  -- cycle ;
\draw[thick, decorate, decoration={coil,segment length=10pt}] (1.46,2.01) -- (1.44,2.09) -- (1.41,2.25)  ;
\fill[fill=black, fill opacity=1] (1.41,2.25) circle (5pt)  ;
\draw[thick, decorate, decoration={coil,segment length=10pt}] (1.69,1.9) .. controls (1.7,1.9) .. (1.8,2.1) -- (2,2.3) .. controls (2.2,2.4) .. (2.4,2.5) -- (2.6,2.6) .. controls (2.7,2.72) .. (2.8,2.82) -- (3,2.9) .. controls (3.1,3) .. (3.2,3.2) .. controls (3.4,3.5) .. (3.5,3.7) ;
\filldraw[thick, color=black, fill opacity=1, rounded corners = 5pt, decorate, decoration={zigzag, segment length=10pt}] (3.35,3.5) -- (3.55,3.5) -- (3.5,3.4) -- (3.3,3.4)  -- cycle ;


\draw[ultra thick] plot [smooth, tension = 0.9] coordinates { (-4,0) (-3,0.5) (-2,0) (-1.5, 0.5) (-1,0) (0.5,0.5)} ;

\draw[ultra thick, color=red] plot [smooth, tension = 0.9] coordinates { (0.3,0.4) (0.4,0.45) (0.45,0.455)  (0.5, 0.47) (0.7,0.5) (1,0.5) (1.25,0.45) (1.5,0.36) (2,0.3) (2.25,0.23) (2.5,0.25) (3,0.15) (3.5,0.25) (3.75,0.1) (4,0.2) (4.25,0.15) (4.5,0) (5,-0.4)};

\draw[thick, decorate, decoration={coil,segment length=10pt}] (0.5,0.5) -- (0.75,0.58) -- (0.8,0.57) -- (1,0.6) .. controls (1.2,0.5) and (1.4,0.5) .. (1.5,0.45) -- (1.7,0.43) -- (2,0.42) .. controls (2.1,0.45) and (2.2,0.43) .. (2.35,0.48) -- (2.5,0.35) -- (2.8,0.33) .. controls (3,0.3) and (3.1,0.25) .. (3.3,0.35) -- (3.5,0.3) -- (3.7,0.3) -- (4,0.29) -- (4.1,0.35) -- (4.2,0.3) -- (4.5,0.1) -- (4.7,0) .. controls (4.8,-0.2) and (4.9,0) .. (5,-0.2);
\draw[thick, decorate, decoration={coil,segment length=10pt}] (1.7,0.43)  -- (1.7,0.5) -- (1.85,0.65);
\draw[thick, decorate, decoration={coil,segment length=10pt}] (1.7,0.43)  -- (1.6,0.55) -- (1.55,0.6);
\draw[thick, decorate, decoration={coil,segment length=10pt}] (1.7,0.43)  -- (1.7,0.5) -- (1.7,0.75) -- (1.75,0.77);
\fill[fill=black, fill opacity=1] (1.7,0.43) circle (5pt)  ;
\fill[fill=black, fill opacity=1] (1.85,0.65) circle (5pt)  ;
\fill[fill=black, fill opacity=1] (1.55,0.6) circle (5pt)  ;
\fill[fill=black, fill opacity=1] (1.75,0.77) circle (5pt)  ;
\filldraw[thick, color=black, fill opacity=1, decorate, decoration={zigzag, segment length=10pt}] (2.35,0.51) -- (2.2,0.45) -- (2.2,0.38) -- (2.36,0.39) -- (2.4,0.45) -- cycle ;


\draw[ultra thick, color=red] plot [smooth, tension = 0.9] coordinates { (3,2) (3.3,2.4) (3.5,2.1) (4,1.7) (3.9,1.4) (3.7,1.1) (3.4,1.1) (3,1.3) (2.7,1.7) (2.9,1.8) (3,2)};

\draw[thick, decorate, decoration={coil,segment length=10pt}] (3,2.2) .. controls (3.2,2.5) .. (3.3,2.53) .. controls (3.4,2.45) and (3.45,2.4) .. (3.5,2.27) -- (3.7,2.1) --  (4,1.85) -- (4.2,1.7) -- (4.1,1.5) .. controls (3.9,1.1) .. (3.8,1.1) -- (3.5,0.9) -- (3.3,1) -- (3.3,0.8)  -- (3.2,1) .. controls (3,1) and (3.1,1.1) .. (2.8,1.3) -- (2.6,1.7) -- (2.6,1.8) -- (2.5,1.75) -- (2.7,1.9) -- (2.9,2) -- (3,2.3) -- cycle;
\filldraw[thick, color=black, fill opacity=1, decorate, decoration={zigzag, segment length=10pt}] (4.2,1.7) -- (4.15,1.7) -- (4.1,1.6) -- (4.3,1.6) -- cycle ;
\filldraw[thick, color=black, fill opacity=1, rounded corners = 6pt, decorate, decoration={zigzag, segment length=10pt}] (2.4,1.9) -- (2.6,1.8) -- (2.7,1.5) -- (2.6,1.45) -- (2.2,1.5) -- cycle ;
\draw[thick, decorate, decoration={coil, segment length=10pt}] (2.3,1.7) ellipse [x radius=0.2, y radius = 0.1, rotate = 10];

\draw[thick, decorate, decoration={coil,segment length=15pt}] (1.3,1.1) .. controls (1.4,1.1) .. (1.5,1.3) -- (1.7,1.45)  ;
\draw[thick, decorate, decoration={coil,segment length=15pt}] (1.3,1.1) .. controls (1.2,1.25) .. (1.1,1.3)  ;
\draw[thick, decorate, decoration={coil,segment length=15pt}] (1.3,1.1) .. controls (1.2,0.9) .. (1.3,0.8)  ;
\fill[fill=black, fill opacity=1] (1.3,1.1) circle (5pt)  ;
\fill[fill=black, fill opacity=1] (1.7,1.45) circle (5pt)  ;
\fill[fill=black, fill opacity=1] (1.78,1.52) circle (5pt)  ;
\fill[fill=black, fill opacity=1] (1.75,1.41) circle (5pt)  ;
\fill[fill=black, fill opacity=1] (1.1,1.3) circle (5pt)  ;
\fill[fill=black, fill opacity=1] (1.3,0.8) circle (5pt)  ;

\end{scope}

\node[font=\fontsize{60}{0}\selectfont] at (4,-1.5) {$\mathcal{B}$};
\node[font=\fontsize{60}{0}\selectfont] at (-2.5,2.3) {$\mathcal{D}$};
\node[font=\fontsize{50}{0}\selectfont] at (-2.2,-0.25) {$S_{\mathcal{F}} \cap \mathcal{D}$};
\node[font=\fontsize{50}{0}\selectfont] at (-2.3,-0.7) {$=\,${\color{red}$ S_{\mathcal{F}_{\lambda}}$} $\cap\, \mathcal{D}$};
\node[font=\fontsize{50}{0}\selectfont] at (2.5,1.2) {$S_{\mathcal{F}}$};
\node[font=\fontsize{50}{0}\selectfont, color=red] at (2.5,-0.05) {$S_{\mathcal{F}_{\lambda}}$};

\end{tikzpicture}
    \caption{The closed subset $\Dcal$ of the manifold with boundary $\Bcal$ is shaded in a darker gray. The compact zero set $\ZeroFsection$ of the section $\Fsection : \Bcal \rightarrow \Ecal$ is colored black. Since $\Fsection$ is already transverse on $\Dcal \,$, the black lines within $\Dcal$ are smooth. Outside $\Dcal$ smoothness does not hold in general (indicated by the zigzag lines and irregular shapes). The zero set $S_{\Fsection_{\lambda}}$ of the perturbed section $\Fsection_{\lambda}$ is a smooth submanifold. It is colored red outside $\Dcal$ and its intersection with $\Dcal$ agrees with $\ZeroFsection \cap \Dcal \,$.} 
    \label{fig: zero set perturbation thm}
\end{figure}
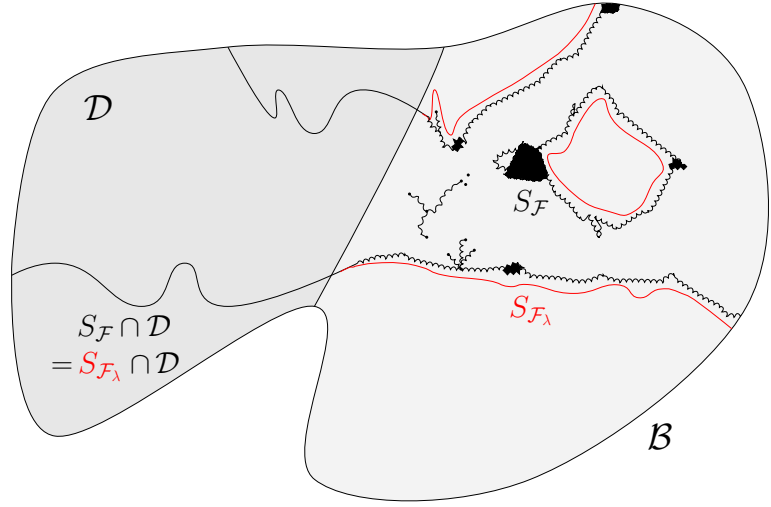

\subsection{Discussion}
\label{subsec: Discussion Abstract Perturbation Thm}

\noindent The notion of a B-bundle pair is inspired by strong bundles appearing in polyfold theory, \cf \cite[Ch.~2.6]{Polyfolds_HWZ}, and Theorem \ref{mthm: perturbation thm} has polyfold analogues as well, see \cite[Thm.~5.3.5 and 5.3.10]{Polyfolds_HWZ}. In fact, our proof of Theorem \ref{mthm: perturbation thm} follows the arguments in \cite{Polyfolds_HWZ} closely. However, since we work with ordinary Fredholm sections in contrast to sc-Fredholm sections, establishing compactness for the zero set of perturbed sections (see Subsection \ref{subsubsec: proof perturbation thm - compactness results}) is easier than in the sc-case: We can resort to local properness of Fredholm maps (see Lemma \ref{app lemma: non-linear Fredholm map}) and avoid the sc-contraction germ property, as in \cite[Thm.~5.2.2]{Polyfolds_HWZ}. Nonetheless, following the polyfold proof in this regard, we still require a distinguished compact dense Banach space in each fiber, and this is where the B-bundle pair comes into play.

Perturbation theorems in the flavor of Theorem \ref{mthm: perturbation thm} are common in nonlinear functional analysis. We refer to \cite[Thm.~A.1]{Action_Selector_SimpleConstruction} for a similar proof in the case of a trivial vector bundle and point out that, for the trivial vector bundle case, no B-bundle pair is required in their proof.

\section{Proof of Theorem \ref{mthm: perturbation thm}}
\label{sec: Proof Abstract Perturbation Thm}

\noindent In this section, we prove the perturbation theorem, that is Theorem \ref{mthm: perturbation thm}. Subsections \ref{subsec: Implicit Function Thm and Parametric Transversality}\,-\,\ref{subsec: Compactness Results for Fredholm Sections} are of preparatory nature: We first recall the implicit function theorem (Subsection \ref{subsec: Implicit Function Thm and Parametric Transversality}), introduce auxiliary norms on bundles (Subsection \ref{subsec: Auxiliary Norms}) and then discuss how to fill up the cokernel of a Fredholm section (Subsection \ref{subsec: Filling up Cokernel}). In Subsection \ref{subsec: Compactness Results for Fredholm Sections} we show compactness for the zero set of a small \Bplus-perturbation of a Fredholm section. It is here where the definition of a B-bundle pair enters crucially. This subsection in particular is inspired by \cite[Ch.~5.2]{Polyfolds_HWZ}. We conclude the proof of Theorem \ref{mthm: perturbation thm} in Subsection \ref{subsec: Proof Perturbation Thm}.

\subsection{Implicit Function Theorem and Parametric Transversality}
\label{subsec: Implicit Function Thm and Parametric Transversality}

We start by reviewing well-known versions of the implicit function theorem and parametric transversality.
The first proposition in this regard is an immediate consequence of the classical implicit function theorem between Banach spaces and the local description of the vertical differential in a trivialization (Remark \ref{rmk: vertical differential}). 

\begin{proposition}[Implicit Function Theorem]
\label{prop: implicit function thm}
Let $\Fsection : \Bcal \rightarrow \Ecal $ be a smooth section of a Banach bundle $\Ecal \rightarrow \Bcal \,$.
Suppose $\Fsection$ is transverse to the zero section $\mathcal{O}_{\Ecal}$ and $\Fsection|_{\partial \Bcal}$ is transverse to the zero section $\mathcal{O}_{\Ecal|_{\partial \Bcal}} \,$. Then the zero set $\ZeroFsection$ is a smooth Banach submanifold of $\Bcal$ with boundary
    \[
    \partial \ZeroFsection = \ZeroFsection \cap \partial \Bcal \,\, .
    \]
    For $x \in \ZeroFsection$ its tangent space at $x$ is $T_x \ZeroFsection = \ker(\Dvert{x} \Fsection)$ and for $x \in \partial \ZeroFsection$ the tangent space of the boundary at $x$ is 
    \[
    T_x (\partial \ZeroFsection) = \ker(\Dvert{x} (\Fsection|_{\partial \Bcal})) = \ker(\Dvert{x} \Fsection ) \cap T_x(\partial \Bcal) \,\, .
    \]
\end{proposition}

\noindent In the implicit function theorem, the empty set is also considered to be a Banach manifold and the possibility that $\ZeroFsection$ is empty is not excluded.

\begin{corollary}
\label{cor: implicit function thm Fredholm sections}
Suppose $\Fsection : \Bcal \rightarrow \Ecal$ is a Fredholm section. Suppose moreover that both $\Fsection$ and $\Fsection|_{\partial \Bcal}$ are transverse to the zero section respectively. Then $\ZeroFsection$ is a smooth Banach submanifold of $\Bcal$ with boundary $\partial \ZeroFsection = \ZeroFsection \cap \partial \Bcal$. Locally around $x \in \ZeroFsection$ it is modeled on $\R^{d(x)} \,$, where $d(x) := \mathrm{ind}(\Dvert{x} \Fsection) \geq 0 \,$. 
\end{corollary}

\begin{proof}
    Locally around $x \in \ZeroFsection$ the Banach submanifold $\ZeroFsection$ is modeled on $T_{x} \ZeroFsection = \ker(\Dvert{x}\Fsection) \,$, which has dimension $\mathrm{ind}(\Dvert{x}\Fsection) \,$.
\end{proof}

\noindent Next, we state a version of a parametric transversality result entering in the proof of the perturbation theorem. 
Although in general it is well-known how to prove parametric transversality, for convenience we highlight the subtleties arising when working with Banach manifolds with boundary.

\begin{proposition}[Parametric Transversality]
\label{prop: Parametric Transversality}
Given a Banach bundle $\Ecal$ over a Banach manifold $\Bcal \,$. Let $P$ be an $n$-dimensional manifold without boundary. Denote by $\EcalHat \rightarrow \Bcal \times P$ the pullback of $\Ecal$ along the projection $\Bcal \times P \rightarrow \Bcal $ to the first factor. Let $\Fsection : \Bcal \times P \rightarrow \EcalHat$ be a Fredholm section of index $d+n \,$, with $d \in \Z \,$. Suppose that
\begin{enumerate}[label=\alph*)]
    \item\label{it: param transversality assumption - transverse} both $\Fsection$ and $\Fsection|_{\partial \Bcal \times P}$ are transverse to the zero section respectively,
    \item\label{it: param transversality assumption - second countable} the zero set $S_{\Fsection}$ is second-countable.
\end{enumerate}
Then for every $p \in P$ it holds
\begin{enumerate}[label=\arabic*)]
    \item\label{it: param transversality conclusion - Fredholm} $\Fsection_p := \Fsection(\,\cdot\,, p) : \Bcal \rightarrow \Ecal$ is a Fredholm section of index $d$
\end{enumerate}
and there exists a residual subset $P_{\mathrm{res}} \subseteq P$ so that for every $p \in P_{\mathrm{res}}$
\begin{enumerate}[label=\arabic*),resume]
    \item\label{it: param transversality conclusion - transverse} $\Fsection_p$ and $\Fsection_p|_{\partial \Bcal}$ are transverse to the zero section $\mathcal{O}_{\Ecal}$ respectively $\mathcal{O}_{\Ecal|_{\partial \Bcal}} \,$.
\end{enumerate}
\end{proposition}

\begin{proof}
    That \ref{it: param transversality conclusion - Fredholm} holds for every $p \in P$ follows immediately from $\Fsection$ being a Fredholm section of index $d+n$ and Lemma \ref{app lemma: Fredholm finite dim summands} in the appendix.
    
    Next, we find a residual subset of $P$ as desired.
    By Corollary \ref{cor: implicit function thm Fredholm sections}, the zero set $S_{\Fsection} \subseteq \Bcal \times P$ is a smooth $(d+n)$-dimensional manifold, second-countable due to assumption. Consider the restriction $\pi : \ZeroFsection \rightarrow P$ of the projection $\Bcal \times P \rightarrow P \,$. By Sard's theorem (requires second-countability), the regular values $R(\pi) \subseteq P$ of $\pi$ form a residual subset of $P$. A direct calculation (as in \cite[Lemma 3.3.29]{Weber} for example) shows that for every regular value $p \in R(\pi)$ the Fredholm section $\Fsection_p$ has surjective vertical differential at every zero and hence
    is transverse to the zero section.
    Repeating the same argument with $\pi' := \pi|_{\partial \ZeroFsection} \,$, the restriction of the projection $\partial \Bcal \times P \rightarrow P$ to the zero set of $\Fsection|_{\partial \Bcal \times P} \,$, we see that
    the section $\Fsection_p|_{\partial \Bcal} : \partial \Bcal \rightarrow \Ecal|_{\partial \Bcal}$ is transverse to the zero section for every $p \in R(\pi') \,$, where $R(\pi') \subseteq P$ denotes the residual set of regular values of $\pi' \,$. The intersection $P_{\mathrm{res}} := R(\pi) \cap R(\pi')$ of the two residual sets is the desired residual set.
\end{proof}

\subsection{Auxiliary Norms}
\label{subsec: Auxiliary Norms}

Inspired by \cite[Def.~5.1.1]{Polyfolds_HWZ}, we introduce the notion of an auxiliary norm on a Banach bundle.

\begin{definition}[Auxiliary Norm]
\label{def: auxiliary norm}
Given a Banach bundle $\pi : \Ecal \rightarrow \Bcal \,$. A continuous function $N : \Ecal \rightarrow \R_{\geq 0}$ is called \textit{auxiliary norm on $\Ecal$} if
\begin{enumerate}[label=\alph*)]
    \item\label{it: auxiliary norm (a)} $N|_{\Ecal_x} : \Ecal_x \rightarrow \R_{\geq 0}$ is a norm on the fiber $\Ecal_x $ for every $x \in \Bcal \,$.
    \item\label{it: auxiliary norm (b)} If $(e_n)_{n \in \N} \subseteq \Ecal$ is a sequence with $\pi(e_n) \rightarrow x \in \Bcal$ and $N(e_n) \rightarrow 0 \,$, then $e_n \rightarrow 0_x$ in $\Ecal \,$.
\end{enumerate}
\end{definition}

\begin{lemma}[\cf Lemma 5.1.2 in \cite{Polyfolds_HWZ}]
\label{lemma: auxiliary norm characterization}
Let $\Ecal$ be a Banach bundle and $N :  \Ecal \rightarrow \R_{\geq 0}$ be a continuous function satisfying \ref{it: auxiliary norm (a)} in Definition \ref{def: auxiliary norm}. Then the following are equivalent.
\begin{enumerate}[label=\arabic*)]
    \item\label{it: auxiliary norm characterization 1)} $N$ is an auxiliary norm on $\Ecal$.
    \item\label{it: auxiliary norm characterization 2)} For every $x_0 \in \Bcal$ and every trivialization $\Phi : \Ecal|_{\Ucal} \overset{\sim}{\rightarrow} \Ucal \times \bbmF$ over an open neighborhood $\Ucal \subseteq \Bcal$ of $x_0$ there exists an open neighborhood $\Ucal' \subseteq \Ucal$ of $x_0$ and a constant $C >0$ with
    \[
    \frac{1}{C} \, \lVert  h \rVert_{\bbmF} \leq N \big( \Phi^{-1}(x,h)\big) \leq C \, \lVert h \rVert_{\bbmF} \qquad \Forall (x,h) \in \Ucal' \times \bbmF \,\, .
    \] 
    Here $\lVert \,\cdot\, \rVert_{\bbmF}$ denotes the norm on the fiber $\bbmF \,$.
\end{enumerate}
\end{lemma}

\begin{proof}
    That \ref{it: auxiliary norm characterization 2)} implies \ref{it: auxiliary norm characterization 1)} follows directly from the inequality in \ref{it: auxiliary norm characterization 2)}. Conversely, suppose $N$ is an auxiliary norm. Let $x_0 \in \Bcal$ be given together with a trivialization $\Phi : \Ecal|_{\Ucal} \rightarrow \Ucal \times \bbmF$ over a neighborhood of $x_0$. Because $\widetilde{N}:= N \circ \Phi^{-1}$ is continuous, there exists a product neighborhood $\Ucal_1 \times \Bar{B}_\eps(0) \subseteq \Ucal \times \bbmF$ of $(x_0,0) \,$, where $\Bar{B}_\eps(0) \subseteq \bbmF$ is the closed ball of radius $\eps >0 \,$, so that $\widetilde{N} \leq 1$ on $\Ucal_1 \times B_\eps(0) \,$.
    As norms are homogeneous, this implies $\widetilde{N}(x,h) \leq \eps^{-1} \, \lVert h \rVert_{\bbmF}$ for every $h \in \bbmF$ and $x \in \Ucal_1 \,$. Now suppose by contradiction that there does not exist a constant $C > 0$ and a neighborhood $\Ucal' \subseteq \Ucal_1$ of $x_0$ with
    \[
    \frac{1}{C} \, \lVert  h \rVert_{\bbmF} \leq \widetilde{N}(x,h) \qquad \Forall (x,h) \in \Ucal' \times \bbmF \,\, .
    \]
    Then there exists a sequence $(x_n,h_n)_n \subseteq \Ucal \times \bbmF$ so that $(x_n)_n$ converges to $x_0$ and $\lVert h_n \rVert_{\bbmF} = 1$ for all $n \in \N$ as well as $\widetilde{N}(x_n,h_n) \rightarrow 0$ as $ n \rightarrow \infty \,$. This uses that $\Bcal$, being a Banach manifold, is first-countable.
    Because $\widetilde{N}$ is an auxiliary norm, $(x_n,h_n)_n$ converges to $(x_0,0) \,$, contradicting the fact that $\lVert h_n\rVert_{\bbmF} = 1$ for all $n \in \N \,$.
\end{proof}

\begin{remark}
\label{rmk: auxiliary norms characterization}
The preceding lemma shows that the restriction $N|_{\Ecal_x}$ of an auxiliary norm $N$ to the fiber $\Ecal_x$ is a norm inducing the given topology on $\Ecal_x \,$.
\end{remark}

\begin{lemma}
\label{lem: auxiliary norms exist around compactum}
Let $\pi : \Ecal \rightarrow \Bcal$ be a Banach bundle over a Banach manifold $\Bcal$ which admits continuous bump functions. Then for every compact set $\mathcal{K} \subseteq \Bcal$ there exists an open set $\Ucal \subseteq \Bcal$ containing $\mathcal{K} $ and an auxiliary norm $N$ on $\Ecal|_{\Ucal} \,$.
\end{lemma}

\begin{proof}
Let $\mathcal{K} \subseteq \Bcal$ be a compact subset. We claim the following:
\begin{claim}
There exists a finite collection $(\Ucal_i, f_i, N_i)_{i=1}^m$ so that
\begin{enumerate}[label=\arabic*)]
    \item the collection $\{\Ucal_i\}_{i=1}^m$ is an open cover of $\mathcal{K}$, 
    \item $f_i \in C^0(\Bcal,[0,1])$ has support contained in $\Ucal_i$, for every $1 \leq i \leq m \,$,
    \item $N_i$ is an auxiliary norm on $\Ecal|_{\Ucal_i} \,$, for every $1 \leq i \leq m$,
    \item $\sum_{i=1}^m f_i > 0$ on $\mathcal{K}$.
\end{enumerate}
\end{claim}

\noindent Assuming the claim, choose an open neighborhood $\Ucal \subseteq \bigcup_{i=1}^m \Ucal_i$ of $\mathcal{K}$ so that $\sum_{i=1}^m f_i > 0$ on $\Ucal \,$. Now set $N := \sum_{i=1}^m f_i \, N_i : \, \Ecal|_{\Ucal} \rightarrow \R_{\geq 0} \,$. One readily verifies that $N$ is an auxiliary norm on $\Ecal|_{\Ucal} \,$.

\begin{proofClaim}
    For each $x \in \mathcal{K}$ choose an open neighborhood $\Ucal_x \subseteq \Bcal$ of $x$ over which a trivialization is defined. Let $N_x$ be the auxiliary norm on $\Ecal|_{\Ucal_x}$ given by transporting the fiber norm $\lVert \,\cdot\, \rVert_{\bbmF}$ via such a trivialization $\Ecal|_{\Ucal_x} \cong \Ucal_x \times \bbmF \,$. Since $\Bcal$ admits continuous bump functions, we can choose a continuous function $f_x : \Bcal \rightarrow [0,1]$ with support contained in $\Ucal_x $ and $f_x(x)=1 \,$. Since $\mathcal{K}$ is compact, we can choose finitely many $x_1,\ldots ,x_m \in \mathcal{K}$ so that $\{f_{x_i}^{-1}(\R_{>0})\}_{i=1}^m$ covers $\mathcal{K} \,$. Finally, relabel $\Ucal_i := \Ucal_{x_i} \comma N_i := N_{x_i} $ and $f_i := f_{x_i} \,$.
\end{proofClaim}
\end{proof}

\subsection{Filling up the Cokernel}
\label{subsec: Filling up Cokernel}

The following lemmata are concerned with sections filling up the cokernel of the vertical differential of a section.

\begin{lemma}
\label{lemma: spanning the cokernel open condition}
Let $\Fsection : \Bcal \rightarrow \Ecal$ be a section of the Banach bundle $\Ecal$ with zero $x_0 \in \ZeroFsection \,$. Suppose $s_1, \ldots , s_n : \Bcal \rightarrow \Ecal$ are finitely many sections so that $s_1(x_0), \ldots , s_n(x_0)$ span the cokernel of $\Dvert{x_0} \Fsection \,$. Then there exists an open neighborhood $\Ucal \subseteq \Bcal$ of $x_0$ so that $s_1(x), \ldots , s_n(x)$ span the cokernel of $\Dvert{x} \Fsection$ for every $x \in \Ucal \cap \ZeroFsection \,$.
\end{lemma}

\begin{proof}
This is a local question, so we may assume that $\Ecal|_{\Ucal} = \Ucal \times \bbmF$ is a trivial bundle for a small neighborhood $\Ucal$ of $x_0 $ contained in  the domain of a chart for $\Bcal$. The vertical differential at
a zero $x \in \ZeroFsection \cap \Ucal$ then is just the usual differential of the principal part $\Fsection : \Ucal \rightarrow \bbmF $ at $x$ (Remark \ref{rmk: vertical differential}) and lies in $\mathcal{L}(\bbmE,\bbmF) \,$, 
the space of bounded linear operators from the the Banach space $\bbmE \,$, on which $\Bcal$ is locally modeled, to the model fiber $\bbmF \,$. The differential of the principal part in the fixed trivialization is defined not only for zeroes of $\Fsection$ but at arbitrary points. The lemma follows by employing Lemma \ref{app lemma: spanning cokernel open condition} in the appendix to the differential of the principal part.
\end{proof}

\begin{lemma}[\cf Lemma 5.3.3 in \cite{Polyfolds_HWZ}]
\label{lemma: bump B^+ sections}
Suppose $\Bcal$ admits smooth bump functions.
Let $(\Ecal,\Ecal^1)$ be a B-bundle pair over $\Bcal$ and $N$ be an auxiliary norm on $\Ecal^1 \,$. Given $x \in \Bcal$, an open neighborhood $\Ucal \subseteq \Bcal$ of $x$ and $e \in \Ecal^1_x\,$. Then for every $\eps > 0$ there exists a \Bplus-section $s : \Bcal \rightarrow \Ecal$ with $s(x)= e \comma \supp(s) \subseteq \Ucal $ and $N \circ s \leq N(e) + \eps \,$.
\end{lemma}

\begin{proof}
    Choose a local trivialization $\Phi : \Ecal^1|_{\Ucal'} \rightarrow \Ucal' \times \bbmF_1$ for $\Ecal^1$ over an open neighborhood $\Ucal' \subseteq \Ucal$ of $x$. Let $h := \Phi_x(e) \in \bbmF_1 \,$. Multiply the section $y \mapsto \Phi^{-1}(y,h)$ of $\Ecal^1|_{\Ucal'}$ by a smooth bump function $\beta : \Bcal \rightarrow [0,1]$ with support in $\Ucal'$ and with $\beta(x) = 1 \,$.
    This gives a smooth section $s : \Bcal \rightarrow \Ecal^1$ with $s(x) = e$ and support in $\Ucal \,$. Since $\Ecal^1 \hookrightarrow \Ecal$ is a smooth bundle map, the section $s : \Bcal \rightarrow \Ecal$ is smooth as well. If the support of $\beta$ is sufficiently small, we additionally obtain $N \circ s \leq N(e) + \eps $ by continuity of $N$.
\end{proof}

\begin{lemma}
\label{lemma: bump B^+ sections fill up cokernel}
Given a B-bundle pair $(\Ecal,\Ecal^1)$ over a Banach manifold $\Bcal$ which admits smooth bump functions and a Fredholm section $\Fsection : \Bcal \rightarrow \Ecal \,$. Let $\Wcal \subseteq \Bcal$ be an open subset containing $\ZeroFsection$ and $N$ an auxiliary norm on $\Ecal^1|_{\Wcal} \,$. Given $x_0 \in \ZeroFsection$ and an open neighborhood $\Ucal_{x_0} \subseteq \Bcal$ of $x_0 \,$.
\begin{enumerate}[label=\alph*)]
    \item\label{it: bump B^+ sections fill up cokernel} There exist finitely many \Bplus-sections $s_1,\ldots , s_n : \Bcal \rightarrow \Ecal$ having support in  $\Ucal_{x_0}$ and with $N \circ s_1|_{\Wcal} \,, \ldots , N \circ s_n|_{\Wcal} \leq 1$ so that
\[
\Dvert{x_0} \Fsection(T_{x_0} \Bcal) + \mathrm{span}\{s_1(x_0),\ldots , s_n(x_0)\} = \Ecal_{x_0} \,\, .
\]
\item\label{it: bump B^+ sections fill up cokernel boundary} If $x_0 \in \partial \Bcal \,$, then we can furthermore arrange that
\[
\Dvert{x_0} \Fsection(T_{x_0} \partial \Bcal) + \mathrm{span}\{s_1(x_0),\ldots , s_n(x_0)\} = \Ecal_{x_0}
\]
\end{enumerate}
\end{lemma}

\begin{proof} We show \ref{it: bump B^+ sections fill up cokernel} in two steps.

    \textit{Step 1:} We show \ref{it: bump B^+ sections fill up cokernel} in the case where $\Wcal = \Bcal \,$. As $\Ecal^1_{x_0} \subseteq \Ecal_{x_0}$ is a dense subspace (because $(\Ecal_{x_0}, \Ecal_{x_0}^1)$ is a B-pair), we can find vectors $v_1,\ldots , v_n \in \Ecal^1_{x_0}$ that span a complement of the image of the Fredholm operator $\Dvert{x_0} \Fsection \,$. By the previous Lemma \ref{lemma: bump B^+ sections}, we can choose \Bplus-sections $s_1,\ldots,s_n$ with support in $\Ucal_{x_0} \,$, for which $v_i$ is a multiple of $s_i(x_0)$ and so that $N \circ s_i \leq 1 \,$.
    
    \textit{Step 2:} We show \ref{it: bump B^+ sections fill up cokernel} in the general case. Since $\Bcal$ admits smooth bump functions, it is a regular space. So, making $\Ucal_{x_0}$ smaller if necessary, we can assume that its closure in $\Bcal$ is contained in $\Wcal \,$. Now apply Step 1 to the restricted B-Bundle pair $(\Ecal|_{\Wcal}, \Ecal^1|_{\Wcal}) $ to obtain \Bplus-sections $s_1,\ldots , s_n : \Wcal \rightarrow \Ecal|_{\Wcal} $ with support in $\Ucal_{x_0} \,$. Since the $\Bcal$-closure of $\Ucal_{x_0}$ is contained in $\Wcal \,$, extending $s_1,\ldots , s_n$ by zero outside $\Wcal $ gives smooth sections on $\Bcal$ as desired.
 
    The proof of \ref{it: bump B^+ sections fill up cokernel boundary} works analogously, choosing vectors $v_1,\ldots , v_n \in \Ecal^1_{x_0}$ that span a complement of $\Dvert{x_0}\Fsection(T_{x_0}\partial \Bcal) \,$.
\end{proof}

\subsection{Compactness Results for Fredholm Sections}
\label{subsec: Compactness Results for Fredholm Sections}

Let $(\Ecal, \Ecal^1)$ be a B-bundle pair over a Banach manifold $\Bcal \,$. Consider a Fredholm section $\Fsection : \Bcal \rightarrow \Ecal$, whose zero set $\ZeroFsection $ we additionally assume to be compact. Moreover, let $\Ucal \subseteq \Bcal$ be an open set containing $\ZeroFsection$ and let $N$ be an auxiliary norm on $\Ecal^1|_{\Ucal} \,$. We keep these data fixed for the rest of this Subsection \ref{subsec: Compactness Results for Fredholm Sections}.

\begin{convention}
\label{convention: auxiliary norms on B-bundle pairs}
We extend $N$ to a function $N : \Ecal|_{\Ucal} \rightarrow [0,\infty]$ by setting 
\[
N(e) := \infty \quad \text{ for } e \in (\Ecal|_\Ucal) \backslash \Ecal^1 \,\, .
\]
\end{convention}

\noindent 

\begin{lemma}
\label{lemma: local precompact regularizing nbhds around zeroes of Fsection}
For every $x_0 \in \ZeroFsection$ there exists an open neighborhood $\Wcal \subseteq \Ucal$ of $x_0$ so that for every $C \in \R$ the set
\[
\set{x \in \Wcal}{N(\Fsection(x)) \leq C}
\]
is precompact in $\Bcal \,$.
\end{lemma}

\begin{proof}
First assume that $x_0$ is an interior point of $\Bcal$.
By choosing a chart around $x_0$ and a common trivialization for $\Ecal$ and $\Ecal^1$ around $x_0$ (which exists by definition of a B-bundle pair) we may localize the setting to one in which
$\Fsection : \Vcal \rightarrow \Vcal \times \bbmF$ is a Fredholm section of a trivial bundle with fiber $\bbmF$ over an open neighborhood $\Vcal \subseteq \bbmE$ of the origin $x_0 = 0$. Here the Banach space $\bbmE$ is the local model of $\Bcal$. We identify the section $\Fsection$ with its principal part, so that we regard $\Fsection$ as a function from $\Vcal$ to $\bbmF \,$.
In the local setting the bundle $\Ecal^1$ corresponds to the trivial bundle $\Vcal \times \bbmF_1 \,$. The fibers $(\bbmF, \bbmF_1)$ constitute a B-pair. In particular the inclusion $\bbmF_1 \hookrightarrow \bbmF$ is compact. By Lemma \ref{lemma: auxiliary norm characterization}, we may replace the given auxiliary norm $N$ with the fiber norm $\lVert \cdot\rVert_{\bbmF_1} \,$. Following our convention, we set $\lVert h \rVert_{\bbmF_1}=\infty$ for every $h \in \bbmF\backslash\bbmF_1 \,$.
We have to find an open neighborhood $\Wcal \subseteq \overline{\Wcal} \subseteq \Vcal$ of the origin so that
\begin{equation}
\label{eq: local precompact regularizing nbhds - U goal}
\text{for every } C \in \R_{\geq 0} \text{ the set } \, \set{x \in \Wcal}{\lVert \Fsection(x) \rVert_{\bbmF_1} \leq C} \,\text{ is precompact in } \bbmE \, .
\end{equation}
Because the differential of $\Fsection : \Vcal \rightarrow \bbmF$ at its zero $x_0 = 0 $ is a Fredholm operator, by Lemma \ref{app lemma: non-linear Fredholm map} in the appendix we can choose an open neighborhood $\Wcal$ of $x_0=0$ with the property that for every sequence $(x_n)_n \subseteq \Wcal$ it holds
\begin{align}
\label{eq: local precompact regularizing nbhds - U choice}
    (\Fsection(x_n))_n \text{ converges in } \bbmF \,\,\, \Longrightarrow \,\,\, \text{a subsequence of } (x_n)_n \text{ converges in } \bbmE \,\, .
\end{align}
Making $\Wcal$ smaller if necessary, we can moreover arrange that its closure in $\bbmE$ is contained in $\Vcal \,$.
Let us verify that $\Wcal$ satisfies \eqref{eq: local precompact regularizing nbhds - U goal}.
To this end, given $C \in \R_{\geq 0}$ and a sequence $(x_n)_n \subseteq \Wcal$ with $\lVert \Fsection(x_n ) \rVert_{\bbmF_1} \leq C \,$, we have to show that a subsequence of $(x_n)_n $ converges in $\bbmE \,$. Since $(\Fsection(x_n))_n$ is bounded in $\bbmF_1 $ and the inclusion $\bbmF_1 \hookrightarrow \bbmF$ is compact, a subsequence of $(\Fsection(x_n))_n$ converges in $\bbmF \,$. By \eqref{eq: local precompact regularizing nbhds - U choice}, a subsequence of $(x_n)_n$ converges in $\bbmE \,$. This finishes the proof in the case where $x_0$ is an interior point.

If $x_0$ is in the boundary of $\Bcal$, in the local setting we may still assume that $\mathcal{V}$ is open in $\bbmE$ by considering a smooth local extension of the localized section $\Fsection$. Choosing $\mathcal{W}$ as above, the intersection of $\mathcal{W}$ with the given half-space then is the desired neighborhood.
\end{proof}

\begin{corollary}
\label{cor: zero set Fsection+s compact}
There exists an open subset $\mathcal{W} \subseteq \Ucal$ containing $\ZeroFsection$ with the following significance. For every $C \in \R$ the set
\begin{equation}
\label{eq: N circ Fsection leq C}
\set{x \in \Wcal}{N(\Fsection(x)) \leq C}
\end{equation}
is precompact in $\Bcal \,$. In particular, for every $C \in \R$ and every \Bplus-section $s : \Bcal \rightarrow \Ecal$ with support contained in $\Wcal$ and $N(s(x)) \leq C$ for all $x \in \Wcal \,$, the zero set $S_{\Fsection+ s}$ is a compact subset of $\Bcal$ and contained in \eqref{eq: N circ Fsection leq C}. 
\end{corollary}

\begin{proof}
    For each $x \in \ZeroFsection$ choose an open neighborhood $\Wcal_x \subseteq \Ucal$ of $x$ as in Lemma \ref{lemma: local precompact regularizing nbhds around zeroes of Fsection}. Since $\ZeroFsection$ is compact, we can choose finitely many $x_1,\ldots , x_m$ so that $\Wcal := \bigcup_{i=1}^m \Wcal_{x_i} $ contains $\ZeroFsection \,$. For every $C \in \R$ the set
    \[
    \set{x\in \Wcal}{N(\Fsection(x)) \leq C} = \bigcup_{i=1}^m  \set{x \in \Wcal_{x_i}}{N(\Fsection(x)) \leq C} \,\, ,
    \]
     is precompact as finite union of precompact sets.
    Finally, suppose $s : \Bcal \rightarrow \Ecal$ is a \Bplus-section with support contained in $\Wcal$ and with $N \circ s|_{\Wcal} \leq C \,$. For an arbitrary $x \in S_{\Fsection + s} \,$, notice that $\Fsection(x) = -s(x) \in \Ecal^1_x \,$. If $s(x) = 0 \,$, then $x \in \ZeroFsection \subseteq \Wcal$ and $N(\Fsection(x))=0 \,$. Otherwise $s(x) \not= 0$ and $x \in \supp(s) \subseteq \Wcal \,$, so that $N(\Fsection(x)) = N(s(x)) \leq C \,$. This shows
   \[
    S_{\Fsection + s} \subseteq \set{x \in \Wcal}{N(\Fsection(x)) \leq C} \,\, ,
   \]
   hence $S_{\Fsection + s} $ is precompact as well. As zero set of a continuous section it is moreover a closed set and therefore actually compact.
\end{proof}

\subsection{Proof of Theorem \ref{mthm: perturbation thm}}
\label{subsec: Proof Perturbation Thm}

We now prove the perturbation theorem, \ie Theorem \ref{mthm: perturbation thm}.

Suppose we are given a Banach manifold $\Bcal$ which admits smooth bump functions and a B-bundle pair $(\Ecal,\Ecal^1)$ over $\Bcal\,$. Let $\Fsection : \Bcal \rightarrow \Ecal$ be a Fredholm section of index $d$ with compact zero set $\ZeroFsection$ and let $\Dcal \subseteq \Bcal$ be a (possibly empty) closed subset of $\Bcal$ satisfying assumption \ref{it: perturbation thm assumption} in the perturbation theorem.

Observe that the conclusion of the perturbation theorem trivially holds true if $\Fsection$ has empty zero set $\ZeroFsection \,$, for we can take $n=1$ and the zero section $s_1 := 0 \,$. From now on let us assume that $\ZeroFsection$ is non-empty.

By Lemma \ref{lem: auxiliary norms exist around compactum} (applied to the compact set $\ZeroFsection $) and Corollary \ref{cor: zero set Fsection+s compact}, we can fix an open subset $\Wcal \subseteq \Bcal$ containing $\ZeroFsection$ and an auxiliary norm $N$ on $\Ecal^1|_{\Wcal}$ so that 
\begin{equation}
\label{eq: choice Wcal}
\set{x \in \Wcal}{N(\Fsection(x)) \leq 1} \,\, \text{ is precompact in } \Bcal \,\, ,
\end{equation}
where we remind the reader of the convention that $N(e) = \infty$ for $e \in (\Ecal|_{\Wcal} ) \backslash \Ecal^1 \,$.
In particular (still Corollary \ref{cor: zero set Fsection+s compact}), for every \Bplus-section $s$ with $\supp(s) \subseteq \Wcal$ and $N\circ s|_{\Wcal} \leq 1\,$, the zero set $S_{\Fsection + s}$ is contained in the set displayed in \eqref{eq: choice Wcal} and compact.

\subsubsection{Choosing the \Bplus-sections}
\label{subsubsec: proof perturbation thm- choosing B^+ sections}
By Lemma \ref{lemma: bump B^+ sections fill up cokernel}, for each $x \in \ZeroFsection$ we can choose finitely many \Bplus-sections $s_1^x,\ldots , s_{k_x}^x : \Bcal \rightarrow \Ecal$ with $\supp(s_i^x) \subseteq \Wcal$ and $N \circ s_i|_{\Wcal} \leq 1$ and so that $s_1^x(x),\ldots , s_{k_x}^x(x)$ span the cokernel of $\Dvert{x} \Fsection \,$. By Lemma \ref{lemma: bump B^+ sections fill up cokernel} \ref{it: bump B^+ sections fill up cokernel boundary}, for $x \in \partial \Bcal$ we can require that $s_1^x(x),\ldots , s_{k_x}^x(x)$ even span the cokernel of $\Dvert{x}(\Fsection|_{\partial \Bcal}) = (\Dvert{x} \Fsection)|_{T_x \partial \Bcal} \,$. Since $\Bcal \backslash \Dcal \subseteq \Bcal$ is an open subset and due to assumption \ref{it: perturbation thm assumption} in the perturbation theorem, we can furthermore arrange that
\begin{align*}
    \begin{cases}
        k_x = 1 \text{ and } s_1^x \text{ vanishes identically on } \Bcal & \text{ if } x \in \Dcal \\
        \supp(s_1^x), \ldots , \supp(s_{k_x}^x) \subseteq \Bcal \backslash\Dcal & \text{ if } x \in \Bcal\backslash \Dcal \,\, .
    \end{cases}
\end{align*}
So $s_i^x$ has support contained in $\Wcal \cap \Bcal\backslash\Dcal$ for every $x \in \ZeroFsection$ and every $i = 1,\ldots , k_x \,$. Now, by Lemma \ref{lemma: spanning the cokernel open condition}, for each $x \in \ZeroFsection$ we may choose an open neighborhood $\Ucal_x \subseteq \Bcal$ of $x$ so that $s_1^x (y),\ldots , s_{k_x}^x(y)$ span the cokernel of $\Dvert{y}\Fsection$ for every $y \in \Ucal_x \cap \ZeroFsection \,$. For $x \notin \partial \Bcal$ we choose $\Ucal_x$ to be contained in the interior of $\Bcal \,$. For $x \in \partial \Bcal$, applying Lemma \ref{lemma: spanning the cokernel open condition} again to the restriction $\Fsection|_{\partial \Bcal} \,$, we can choose $\Ucal_x \subseteq \Bcal$ so that moreover $s_1^x(y),\ldots , s_{k_x}^x(y)$ span the cokernel of $\Dvert{y} (\Fsection|_{\partial \Bcal})$ for every $y \in \Ucal_x \cap \partial\Bcal \cap \ZeroFsection \,$.

We point out the trivial but crucial observation that adding finitely many \Bplus-sections to the collection $\{s_i^x\}_{i=1}^{k_x}$ gives a new collection which still spans the cokernel of $\Dvert{y} \Fsection$ for every $y \in \Ucal_x \cap \ZeroFsection$ and the cokernel of $\Dvert{y} (\Fsection|_{\partial \Bcal})$ for every $y \in \Ucal_x \cap \ZeroFsection \cap \partial \Bcal $ whenever $x \in \partial \Bcal \,$.

Since $\ZeroFsection$ is compact, we can choose finitely many $x_1,\ldots, x_m \in \ZeroFsection$ for which $\Ucal_{x_1} , \ldots , \Ucal_{x_m}$ cover $\ZeroFsection \,$. Taking the union of the corresponding collections $\{s^{x_j}_i\}_{i=1}^{k_{x_j}} \comma j = 1, \ldots , m \,$, and relabeling the sections, we obtain finitely many \Bplus-sections $s_1,\ldots , s_n : \Bcal \rightarrow \Ecal$ having the following properties:
\begin{itemize}
    \item $\supp(s_i) \subseteq \Wcal \cap \Bcal \backslash \Dcal$ and $N \circ s_i|_{\Wcal} \leq 1$ for $i=1,\ldots , n \,$,
    \item for every $x \in \ZeroFsection$ the cokernel of $\Dvert{x} \Fsection$ is spanned by $s_1(x) , \ldots , s_n(x) \,$, 
    \item for every $x \in \ZeroFsection \cap \partial \Bcal$ the cokernel of $\Dvert{x} (\Fsection|_{\partial \Bcal})$ is spanned by $s_1(x) , \ldots , s_n(x) \,$.
\end{itemize}
These are already the \Bplus-sections appearing in the conclusion of the perturbation theorem and we keep them fixed throughout the rest of the proof.

\subsubsection{The universal section}
\label{subsubsec: proof perturbation thm - universal section}
For every $\lambda = (\lambda_1,\ldots ,\lambda_n) \in \R^n $ we define the section
\[
s_{\lambda} := \sum_{i=1}^n \lambda_i \, s_i : \Bcal \rightarrow \Ecal \,\, .
\]
and observe that $s_{\lambda}$ is a \Bplus-section (as linear combination of \Bplus-sections) with support contained in $\Bcal\backslash\Dcal \cap \Wcal \,$. Using that $N \circ s_i|_{\Wcal} \leq 1$ for all $i \,$, we have the estimate
\begin{equation}
\label{eq: N(s_lambda)}
N(s_{\lambda}(x)) \leq \sum_{i=1}^n |\lambda_i| \, N(s_i(x)) \leq \sum_{i=1}^n |\lambda_i| \leq \sqrt{n} \, |\lambda| \qquad \Forall x \in \Wcal \,\, ,
\end{equation}
where $|\lambda|$ denotes the Euclidean norm of the vector $\lambda \in \R^n \,$.

Let $\BcalHat := \Bcal \times \R^{n}$ and consider the pullback bundle $\EcalHat \rightarrow \BcalHat$ of $\Ecal$ along the projection $\Bcal \times \R^n \rightarrow \Bcal \,$. Define the \textit{universal section} $\Funiv : \BcalHat \rightarrow \EcalHat$ by
\begin{equation*}
\Funiv(x,\lambda) := \Fsection_{\lambda}(x) := \Fsection(x) + s_\lambda(x) \quad \text{for } \, x \in \Bcal\comma \lambda = (\lambda_1,\ldots , \lambda_n) \in \R^n \,\, .
\end{equation*}
Notice that $\Fsection = \Fsection_{0} \,$, with $0 \in \R^n \,$.

\subsubsection{Compactness results}
\label{subsubsec: proof perturbation thm - compactness results}

For the next lemma recall that $\ball^n_r \subseteq \R^n$ (respectively $\ballclosed_r^n \subseteq \R^n$) denotes the open (respectively closed) ball of radius $r$ centered at the origin.

\begin{lemma}
\label{lemma: Funiv compact zero set}
For every $\eps \in (0, \tfrac{1}{\sqrt{n}} ]$ the zero set $S_{\Funiv} \cap (\Bcal \times \ballclosed^n_\eps)$ is compact.
\end{lemma}

\begin{proof}
    For $0 < \eps \leq \tfrac{1}{\sqrt{n}}$ we will show
    \[
    S_{\Funiv} \cap (\Bcal \times \ballclosed^n_\eps) \subseteq \set{x \in \Wcal}{N(\Fsection(x)) \leq 1} \times \ballclosed^n_\eps \,\, ,
    \]
    which proves the claimed compactness since the above left-hand side is a closed subset and the right-hand side is precompact by \eqref{eq: choice Wcal}. To show the inclusion, let $(y,\lambda)$ be a zero of $\Funiv$ with $|\lambda| \leq \eps \,$. Then $y$ lies in the zero set $S_{\Fsection + s_{\lambda}} \,$. Now $s_{\lambda}$ is a \Bplus-section with support contained in $\Wcal$ and whose auxiliary norm on $\Wcal$ is, by the estimate in \eqref{eq: N(s_lambda)}, bounded by $\sqrt{n} \, |\lambda| \leq \sqrt{n} \, \eps \leq 1 \,$. By the property of $\Wcal$ mentioned right after \eqref{eq: choice Wcal}, the zero set $S_{\Fsection + s_{\lambda}}$ therefore is contained in $\set{x \in \Wcal}{N(\Fsection(x)) \leq 1} \,$.
\end{proof}

\noindent Note that the previous lemma shows property \ref{it: perturbation rmk - compactness} in Remark \ref{rmk: perturbation thm}.  It moreover entails the following result, for which we recall from Remark \ref{rmk: Banach mfds - topological properties} that compact subsets of Banach manifolds are sequentially compact.

\begin{lemma}
\label{lemma: proof perturbation thm - zero set in prescribed nbhd of S_Fsection}
For every open subset $\Ucal \subseteq \Bcal $ containing $S_{\Fsection} = S_{\Fsection_{0}}$ there exists $\eps_0 > 0$ so that for every $0 < \eps \leq \eps_0$ it holds
\[
S_{\Funiv} \cap (\Bcal \times \ball^n_\eps) \subseteq \Ucal \times \ball_\eps^n \,\, .
\]
In other words, the zero set of the restricted section $\Funiv : \Bcal \times \ball_\eps^n \rightarrow \EcalHat|_{\Bcal \times \ball^n_\eps}$ is contained in $\Ucal \times \ball^n_\eps \,$.
\end{lemma}

\begin{proof}
    Assume by contradiction that there exists a sequence $(x_m, \lambda_m)_{m \in \N}$ of zeroes of $\Funiv$ with $(\lambda_m)_m$ tending to the origin and $x_m \notin \Ucal $ for every $m \in \N \,$. By the compactness result in Lemma \ref{lemma: Funiv compact zero set}, a subsequence $(x_{m_k}, \lambda_{m_k})_k$ converges to some $(x,0) \in S_{\Funiv} \,$. Hence $x \in S_{\Fsection} \subseteq \Ucal \,$. Now the convergence of $(x_{m_k})_k $ to $x$ implies that $x_{m_k} \in \Ucal$ for almost every $k \in \N$, a contradiction.
\end{proof}

\subsubsection{Fredholm property and transversality for the universal section}
\label{subsubsec: proof perturbation thm - Fredholm + Transversality of Funiv}
 \, \\
Next we aim to show the following.

\begin{lemma}
\label{lemma: proof perturbation thm - Fredholm and transversality of Funiv}
For $\eps > 0$ sufficiently small the restriction $\Funiv : \Bcal \times \ball^n_\eps \rightarrow \EcalHat|_{\Bcal \times \ball^n_\eps}$ is a Fredholm section of index $d + n$ that is transverse to the zero section.
\end{lemma}

\begin{proof} We show the lemma in two steps.

\textit{Step 1:}
We first reduce the proof to a local problem. Suppose for each $x \in \ZeroFsection$ we have found an open neighborhood $\Ucal_x \subseteq \Bcal$ of $x$ and some $\eps(x) > 0$ so that the restriction of $\Funiv$ to $\Ucal_x \times \ball^n_{\eps(x)}$ is a Fredholm section of index $d + n$ that is transverse to the zero section. Using the compactness of $\ZeroFsection \,$, we select finitely many $x_1,\ldots , x_m \in \ZeroFsection$ so that
$\Ucal := \bigcup_{i=1}^m \Ucal_{x_i}$ is an open set containing $\ZeroFsection\,$. By Lemma \ref{lemma: proof perturbation thm - zero set in prescribed nbhd of S_Fsection}, for some $0 < \eps < \min_{i=1,\ldots, m} \eps(x_i) \,$, the zero set of the restriction of $\Funiv$ to $\Bcal \times \ball^n_\eps$ is contained in $\Ucal \times \ball^n_\eps \,$. Then $\Funiv$ restricted to $\Bcal \times \ball^n_\eps$ is a Fredholm section of index $d+n$ that is transverse to the zero section. We have successfully reduced the problem to a local one.

\textit{Step 2:} Now let $x_0 \in \ZeroFsection$ be arbitrary but fixed. We have to find an open neighborhood $\Ucal \subseteq \Bcal$ of $x_0$ and $\eps > 0$ so that $\Funiv$ restricted to $\Ucal \times \ball^n_\eps$ is a Fredholm section of index $d + n$ that is transverse to the zero section. Fix a trivialization $\Ecal|_{\Ucal_0} \cong \Ucal_0 \times \bbmF$ with fiber $\bbmF$ over a neighborhood $\Ucal_0 \subseteq \Bcal$ of $x_0$ that is contained in the domain of a chart for $\Bcal \,$. Such a trivialization pulls back to a trivialization of $\EcalHat$ over $\Ucal_0 \times \R^n\,$. Precomposing the principal part in this fixed trivialization with the chart and abusing notation, we therefore consider $\Funiv$ as a map $\Funiv : \Ucal_0 \times \R^n \rightarrow \bbmF \,$, where $\Ucal_0$ is either open in the Banach space $\bbmE$ on which $\Bcal$ is modeled locally or (if $x_0$ is a boundary point of $\Bcal$) relatively open in some half-space of $\bbmE$. We proceed similarly for $\Fsection$ and $s_1, \ldots , s_n \,$.  
The ordinary differential of these maps, as maps between Banach spaces, lies in $\mathcal{L}(\bbmE \oplus \R^n,\bbmF)$ respectively $\mathcal{L}(\bbmE,\bbmF)$ and is defined everywhere on $\Ucal_0 \times \R^n$ respectively $\Ucal_0 \,$, not just at zeroes. This ordinary differential of the principal part will temporarily be denoted by $D \,$. At a zero of a section, its ordinary differential $D$ coincides with its vertical differential $D^{\mathrm{v}} \,$, see  Remark \ref{rmk: vertical differential}.

Since $\Fsection = \Fsection_{0}$ is a Fredholm operator of index $d$ by assumption, the differential $D \Fsection (x_0) \in \mathcal{L}(\bbmE,\bbmF)$ is a Fredholm operator of index $d \,$. Hence, by Lemma \ref{app lemma: Fredholm finite dim summands} in the appendix, the differential $D \Funiv (x_0, 0) : \bbmE \oplus \R^n \rightarrow \bbmF$ is Fredholm of index $d + n \,$. It is a fact that the set of Fredholm operators is open in $\mathcal{L}(\bbmE \oplus \R^n,\bbmF)$ and the Fredholm index is locally constant, see \cite{Lang_Real_Analysis}. By continuity of the differential, therefore $D \Funiv$ is a Fredholm operator of index $d + n$ in a sufficiently small neighborhood of $(x_0, 0) \,$. In particular, $\Funiv$ restricts to a Fredholm section on a small neighborhood $\Ucal \times \ball^n_\eps$ of $(x_0,0) \,$, with $\Ucal \subseteq \Ucal_0$ relatively open. 

For transversality we similarly only need to verify transversality at $(x_0,0) \,$. In more detail, it suffices to show that 
\begin{equation}
\label{eq: proof perturbation thm - D Funiv(x_0,0) surjective}
    \text{the differential } D \Funiv(x_0,0) : \, \bbmE \oplus \R^n \rightarrow \bbmF \text{ is surjective.}
\end{equation}
Let us assume \eqref{eq: proof perturbation thm - D Funiv(x_0,0) surjective} for the moment. Recall that the subset of surjective bounded linear maps is open in $\mathcal{L}(\bbmE\oplus\R^n,\bbmF)$, see \cite[Ch.~XV, Thm.~3.4]{Lang_Real_Analysis}. Thus, after possibly making $\Ucal$ and $\eps $ smaller, $D \Funiv(x,\lambda) $ is surjective for every $(x,\lambda) \in \Ucal \times \ball^n_\eps \,$. In particular, the differential is surjective at every zero $(x,\lambda) \in \Ucal \times \ball^n_\eps $ of $\Funiv \,$. Since the differential is a Fredholm operator, its kernel automatically has a closed complement. This shows that the restriction of $\Funiv$ to $\Ucal \times \ball^n_\eps$ is transverse to the zero section.

It remains to prove \eqref{eq: proof perturbation thm - D Funiv(x_0,0) surjective}.
In the fixed trivialization and chart, the differential of the universal section at $(x, \lambda) = (x,(\lambda_1,\ldots, \lambda_n))$ in direction $(\xi,\mu) = (\xi, (\mu_1,\ldots , \mu_n))$  is given by
    \begin{align*}
     D \Funiv (x, \lambda) \, [\xi, \mu] = D \Fsection(x) [\xi] + \sum_{i=1}^n \lambda_i \, Ds_i(x) [\xi] +  \sum_{i=1}^n  \mu_i\, s_i(x) \,\, .
    \end{align*}
    At $(x_0,0)$ this simplifies to
    \begin{align*}
        D \Funiv(x_0,0) \, [\xi, \mu] = D \Fsection(x_0) [\xi] + \sum_{i=1}^n \mu_i \, s_i(x_0) = \Dvert{x_0} \Fsection [\xi] + \sum_{i=1}^n \mu_i \, s_i(x_0) \,\, .
    \end{align*}
    Since we have chosen the sections $s_1,\ldots , s_n$ to span the cokernel of the vertical differential of $\Fsection$ at every of its zeroes, the differential $D \Funiv(x_0, 0)$ is indeed surjective. This finishes Step 2.
\end{proof}

\begin{lemma}
\label{lemma: proof perturbation thm - Funiv transverse boundary}
For $\eps > 0$ sufficiently small the restriction $\Funiv : \partial \Bcal \times \ball^n_\eps \rightarrow \EcalHat|_{\partial \Bcal \times \ball^n_\eps}$ is transverse to the zero section.
\end{lemma}

\noindent The proof of Lemma \ref{lemma: proof perturbation thm - Funiv transverse boundary} is completely analogous to the proof of the previous Lemma \ref{lemma: proof perturbation thm - Fredholm and transversality of Funiv}. We briefly explain how to adapt the argument but skip the full proof. In the first step one uses an analogous version of Lemma \ref{lemma: proof perturbation thm - zero set in prescribed nbhd of S_Fsection}, asserting that, for every relatively open subset $\Ucal \subseteq \partial \Bcal$ containing $\ZeroFsection \cap \partial \Bcal \,$, the zero set of the restriction of $\Funiv$ to $\partial \Bcal \times \ball^n_\eps$ is contained in $\Ucal \times \ball^n_\eps$ for $\eps$ sufficiently small. 
    In the second step one uses that we have chosen the sections $s_1,\ldots , s_n$ to span the cokernel of $\Dvert{x} (\Fsection|_{\partial \Bcal})$ at every $x \in \ZeroFsection \cap \partial \Bcal \,$.

\begin{lemma}
\label{lemma: proof perturbation thm - zero set second countable}
For $\eps > 0$ sufficiently small the zero set $S_{\Funiv} \cap (\Bcal \times \ball^n_\eps)$ is second-countable.
\end{lemma}

\begin{proof}
    By the transversality results in Lemmata \ref{lemma: proof perturbation thm - Fredholm and transversality of Funiv} and \ref{lemma: proof perturbation thm - Funiv transverse boundary} and the implicit function theorem, the zero set $S_{\Funiv} \cap (\Bcal \times \ball^n_\eps)$ is a Banach manifold modeled on the Banach space $\R^{n+d} \,$, at least for $\eps  >0$ small.
   By the compactness result in Lemma \ref{lemma: Funiv compact zero set}, we can cover $S_{\Funiv} \cap (\Bcal \times \ball^n_\eps)$ with countably many charts.
\end{proof}

\subsubsection{Parametric transversality}
\label{subsubsec: proof perturbation thm - parametric transversality}

Lemmata \ref{lemma: proof perturbation thm - Fredholm and transversality of Funiv} and \ref{lemma: proof perturbation thm - Funiv transverse boundary} show that for $\eps > 0$ sufficiently small the restricted universal section $\Funiv : \Bcal \times \ball^n_\eps \rightarrow \EcalHat|_{\Bcal \times \ball^n_\eps}$ is a Fredholm section of index $d + n$ that is transverse to the zero section and that its restriction to the boundary $\partial \Bcal \times \ball^n_\eps $ is transverse to the zero section as well. Its zero set is second-countable by Lemma \ref{lemma: proof perturbation thm - zero set second countable}. 
Consequently $\Funiv|_{\Bcal \times \ball^n_\eps}$ satisfies all assumptions in the parametric transversality theorem (Proposition \ref{prop: Parametric Transversality}). Therefore
\begin{align}
\label{eq: proof perturbation thm - parametric transversality Fredholm}
    \Fsection_\lambda : \Bcal \rightarrow \Ecal \text{ is a Fredholm section of index } d \text{ for every } \lambda \in \ball^n_\eps
\end{align}
and moreover there exists a residual subset $\Bres \subseteq \ball^n_\eps$ so that 
\begin{align}
\label{eq: proof perturbation thm - parametric transversality transverse}
   \Forall \lambda \in \Bres : \,\, \text{Both } \Fsection_\lambda \text{ and } \Fsection_\lambda|_{\partial \Bcal} \text{ are transverse to the zero section respectively.}
\end{align}

\subsubsection{Concluding the proof}
\label{subsubsec: proof perturbation thm - conclusion}

We have now chosen \Bplus-sections $s_1,\ldots , s_n$ with support in $\Bcal \backslash \Dcal$ (in Subsection \ref{subsubsec: proof perturbation thm- choosing B^+ sections}), a sufficiently small $\eps \in (0,\tfrac{1}{\sqrt{n}}]$ (in Subsection \ref{subsubsec: proof perturbation thm - parametric transversality}) and a residual subset $\Bres \subseteq \ball^n_\eps$ (in Subsection \ref{subsubsec: proof perturbation thm - parametric transversality}). We claim that, with these choices, \ref{it: perturbation thm conclusion - Fredholm}\,-\,\ref{it: perturbation thm conclusion - zero set compact mfd} in Theorem \ref{mthm: perturbation thm} are satisfied for every $\lambda \in \Bres$ and moreover \ref{it: perturbation rmk - Fredholm}\,-\,\ref{it: perturbation rmk - compactness} in Remark \ref{rmk: perturbation thm} hold.
That \ref{it: perturbation rmk - Fredholm} holds, and thus a fortiori also \ref{it: perturbation thm conclusion - Fredholm} for every $\lambda \in \Bres \,$, is stated in \eqref{eq: proof perturbation thm - parametric transversality Fredholm}. That \ref{it: perturbation thm conclusion - transverse} holds for every $\lambda \in \Bres $ is asserted in \eqref{eq: proof perturbation thm - parametric transversality transverse}. Property \ref{it: perturbation rmk - compactness} holds by Lemma \ref{lemma: Funiv compact zero set} since $\eps \leq \tfrac{1}{\sqrt{n}} \,$.

We now show that also \ref{it: perturbation thm conclusion - zero set compact mfd} holds for given $\lambda \in \Bres \,$. Using the already established properties \ref{it: perturbation thm conclusion - Fredholm} and \ref{it: perturbation thm conclusion - transverse}, by Corollary \ref{cor: implicit function thm Fredholm sections} the zero set $S_{\Fsection_\lambda}$ is a submanifold of $\Bcal$ with boundary $\partial S_{\Fsection_\lambda} = S_{\Fsection_\lambda} \cap \partial \Bcal$ and modeled on $\R^d$. It is compact due to the already proven \ref{it: perturbation rmk - compactness}. In particular $S_{\Fsection_\lambda}$ is also second-countable, see Remark \ref{rmk: Banach mfds - topological properties} \ref{it: Banach mfds topological - locally R^n and compact}, and therefore a $d$-dimensional submanifold.
Lastly, $S_{\Fsection_\lambda} \cap \Dcal = \ZeroFsection \cap \Dcal$ because the \Bplus-sections $s_1,\ldots , s_n$ vanish on $\Dcal \,$.

This finishes the proof of Theorem \ref{mthm: perturbation thm}.

\section{Application: Re-proving Schwarz's Theorem}
\label{sec: Application Schwarz thm}

\subsection{Preliminaries}
\label{subsec: Preliminaries}

In the following, $(M,\omega)$ will always be a closed (\ie compact and without boundary) symplectic manifold of dimension $\dim(M) = 2m \,$. We moreover assume that $M$ is \textit{symplectically aspherical}, that is
\[
\int_{\bbS^2} f^* \omega = 0 \quad \Forall f \in \Cinfty(\bbS^2,M) \,\, ,
\]
or equivalently the map on homology $\omega : H_2(M) \rightarrow \R $ is trivial when restricted to the image of the Hurewicz homomorphism $\pi_2(M,q) \rightarrow H_2(M)$ for every $q \in M \,$. Note that the image of the Hurewicz homomorphism $\pi_2(M,q) \rightarrow H_2(M)$ depends only on the connected component of $M$ containing $q \,$.

\subsubsection{$J$-holomorphic curves}
\label{subsubsec: J-holomorphic Curves}

Let $J$ be an $\omega$-compatible almost complex structure on $M$. It induces the inner product $g_J := \omega(\,\cdot\, , J \,\cdot\,)$ on $M$.
A \textit{$J$-holomorphic curve} on a Riemann surface $(\Sigma,j_\Sigma)$ is a smooth map $u : \Sigma \rightarrow M$ that satisfies $J(u) \circ du = du \circ j_\Sigma \,$. Its \textit{energy} is defined by
\begin{equation}
\label{eq: energy J-holomorphic curves}
E(u) := \tfrac{1}{2} \int_{\Sigma} \lVert du \rVert_J^2 \, \mathrm{dvol}_{\Sigma} = \int_{\Sigma} u^*\omega \,\, ,
\end{equation}
where $\mathrm{dvol}_\Sigma = g_\Sigma(j_\Sigma\,\cdot\comma\cdot\,)$ is the volume form for a choice of conformal metric $g_\Sigma$ on $(\Sigma,j_\Sigma)$
and the norm appearing in the integral is defined for any real linear map $L : T_z \Sigma \rightarrow T_{u(z)}M$ by
\[
\lVert L \rVert_J := \frac{1}{|v|} \sqrt{|L(v)|_J^2 + |L(j_\Sigma v)|^2_J} \quad \text{ for any choice of } v \in T_z \Sigma , \, v \not= 0 \,\, .
\]
The integrand $ \lVert du \rVert_{J}^2 \, \mathrm{dvol}_\Sigma$ is independent of the choice of the conformal metric $g_\Sigma$ (even if $u$ is not $J$-holomorphic).

\begin{remark}
\label{rmk: symp aspherical -> no J-hol curves}
The energy identity \eqref{eq: energy J-holomorphic curves} and the symplectic asphericity imply that the only $J$-holomorphic spheres $u: \bbS^2 \rightarrow M$ are the constant ones.
\end{remark}

\noindent For convenience let us repeat the well-known Removal-of-singularities-theorem. We denote by $\mathbb{D}^2_r = \set{z \in \C}{|z| \leq r}$ the closed disc or radius $r $ and by $\mathbb{D}^2 := \mathbb{D}_1^2$ the unit disk. 

\begin{proposition}[Removal of singularities, \cite{J_holomorphic_curves}]
\label{prop: removal-of-singularities}
If $u : \mathbb{D}^2_r\backslash\{0\} \rightarrow M$ is a $J$-holomorphic curve of finite energy $E(u) < \infty \,$, then $u$ extends to a $J$-holomorphic curve on $\mathbb{D}^2_r \,$.
\end{proposition}

\noindent Finally, we also recall the definition of the anti-holomorphic part of a one-form.

\begin{definition}
\label{def: anti-holom part one-form}
Given a real linear one-form $\alpha \in \Omega^1(\Sigma, E)$ with values in a complex vector bundle $(E,J)$ over $\Sigma$, its \textit{complex antilinear part} is the section 
\[
\alpha^{(0,1)} = \tfrac{1}{2} \left( \alpha + J \circ \alpha \circ j_\Sigma \right)  \in \Omega^{0,1}(\Sigma,E) = \Gamma(\Lambda^{0,1} \otimes_{\C} E)
\]
of the complex bundle $\Lambda^{0,1} \otimes_{\C} E \,$, where $\Lambda^{0,1} := \Lambda^{0,1} T^* \Sigma $ denotes the complex line bundle of complex antilinear functionals $T_z \Sigma \rightarrow \C \comma z \in \Sigma \,$.
\end{definition}

\begin{remark}
\label{rmk: def anti-holom part}
The one-forms $\alpha$ we consider below will take on values in the pullback bundle $(u^*TM, u^*J) \,$, for a given smooth map $u : \Sigma \rightarrow M$ and almost complex structure $J$ on $M$.
\end{remark}

\noindent For any smooth map $u : \Sigma \rightarrow M$ the Cauchy-Riemann operator $\overline{\partial}_J$ applied to $u$ then is the complex antilinear part of the differential of $u$, that is
\[
\overline{\partial}_J \, u := (du)^{(0,1)} \in \Omega^{0,1}(\Sigma, u^*TM) \,\, .
\]

\subsubsection{Embedding the cylinder into the sphere}
\label{subsubsec: cylinder into sphere}

 We identify the cylinder $\R \times \bbS^1$ biholomorphically with the punctured sphere $\bbS^2\backslash\{z_-,z_+\} = \C \mathrm{P}^1\backslash\{0,\infty\}$ via
\[
\R \times \bbS^1 \overset{\cong}{\longrightarrow} \C\backslash \{0\} = \C \mathrm{P}^1\backslash\{0,\infty\} \comma \quad (s,t) \mapsto e^{2 \pi (s+\i t)} \,\, .
\]
Observe that a $J$-holomorphic curve $u : \R \times \bbS^1 \rightarrow M$ has energy $E(u) = \int_{\R \times \bbS^1} |\partial_s u|_J^2 \, ds \,dt \,$. This motivates the next definition of the energy for arbitrary smooth maps.

\begin{definition}
\label{def: energy}
The energy $E(u)$ of a smooth map $u : \R \times \bbS^1 \rightarrow M$ is defined by
\[
E(u) := \int_{\R \times \bbS^1} |\partial_s u|_J^2 \, ds \, dt \, \in [0,\infty] \,\, .
\]
The energy of a map $u : \bbS^2 \rightarrow M$ is $E(u) := E(u|_{\R \times \bbS^1} ) \,$.
\end{definition}

\subsubsection{The action functional}
\label{subsubsec: Action Functional}

Let $\loopspace := \Cinfty_{\mathrm{contr}}(\bbS^1,M)$, the space of contractible smooth loops. Given a Hamiltonian $H \in \Cinfty(M \times \bbS^1,\R) \,$, its Hamiltonian action functional $\actionH : \loopspace \rightarrow \R$ for a contractible loop $x \in \loopspace$ is defined by
\[
\actionH(x) := -\int_{\mathbb{D}^2} \overline{x}^* \omega + \int_0^1 H_t(x(t)) \, dt \,\, ,
\]
where $\overline{x} \in \Cinfty(\mathbb{D}^2,M)$ is any choice of filling disk for $x \,$, that is $\overline{x}|_{\partial \mathbb{D}^2} = x \,$. (This definition is independent of the filling disk $\overline{x}$ because $M$ is symplectically aspherical.)

\subsubsection{The $L^2$-gradient and the Floer equation}
\label{subsubsec: L^2-gradient}

Given an $\omega$-compatible almost complex structure $J$ on $M$ with induced metric $g_J := \omega(\,\cdot\, , J\,\cdot\,) \,$, consider the $L^2$-metric on $\loopspace$ which at $x \in \loopspace$ is defined by
\begin{equation}
\label{eq: L^2 metric}
\langle v_1 , v_2 \rangle_{x} := \int_{\bbS^1} (g_J)_{x(t)}(v_1(t), v_2(t)) \, dt \qquad \Forall v_1,v_2 \in T_x \loopspace = \Gamma(x^*TM) \,\, .
\end{equation}
The gradient of $\actionH$ with respect to this metric is given by
\[
\nabla \actionH (x) =  J(x)(\partial_t x - X_{H_t}(x)) \in \Gamma(x^*TM) \,\, ,
\]
where $X_{H_t}$ is the Hamiltonian vector field defined by $\omega(\,\cdot\, , X_{H_t}) = dH_t \,$. Thus the critical points of $\actionH$ are precisely the contractible one-periodic solutions to $\partial_t x(t) = X_{H_t}(x(t)) \,$.
The negative gradient flow lines $s \mapsto u_s \in \loopspace$ of $\actionH $ are, when considered as maps $u : \R \times \bbS^1 \rightarrow M \,$, the solutions of the Floer equation
\begin{equation}
\label{eq: Floer eq}
\partial_s u + J(u) \, (\partial_t u - X_{H_t}(u) ) = 0 
\end{equation}
and are also called \textit{Floer cylinders}.

\begin{remark}
\label{rmk: energy action functional}
Suppose $u : \R \times \bbS^1 \rightarrow M$ is a solution of \eqref{eq: Floer eq}.
\begin{enumerate}[label=\roman*)]
    \item If $u_s = u(s,\,\cdot\,)$ converges (in the $H^1(\bbS^1,M)$-topology) to loops $x_\pm \in \loopspace $ as $s \rightarrow \pm \infty \,$, then 
    \[
    E(u) = - \int_{-\infty}^\infty \langle \nabla \actionH(u_s) , \partial_s u_s \rangle \, ds =  \actionH(x_-) - \actionH(x_+) \,\, .
    \]
    \item If $E(u)=0 \,$, then $u$ is $s$-independent. If moreover $ u(s,\,\cdot\,) = u(0,\,\cdot\,)$ is contractible, then $u(s,\,\cdot\,)$ is a critical point of $\actionH \,$.
    \item For every $s \in \R$ the reparametrized cylinder $u(\,\cdot\,+ s , \,\cdot\,)$ also solves \eqref{eq: Floer eq}.
\end{enumerate}
\end{remark}

\subsubsection{Gromov compactness}
\label{subsubsec: Gromov Compactness}

Let $H^\lambda_{s,t} : M \rightarrow \R$ be an $(s,t,\lambda)$-dependent Hamiltonian, where $(s,t) \in \R \times \bbS^1$ and $\lambda \in I$ is a parameter in a compact interval $I \subseteq \R $ and the dependence on the parameters is smooth. We assume that for some constant $s_0 >0$ we have
\[
X_{\Hparam }= 0 \qquad \text{ for } |s| \geq s_0 \,  \text{ and } t \in \bbS^1 \comma \lambda \in I
\]
and are interested in solutions of the \textit{parameter-dependent Floer equation}
\begin{equation}
\label{eq: Floer eq parameter}
 \partial_s u + J(u) (\partial_t u - X_{\Hparam}(u)) = 0 \,\, .
\end{equation}
We say that a pair $(\lambda, u) \in I \times \Cinfty(\bbS^2,M)$ is a \textit{solution of \eqref{eq: Floer eq parameter}} if $u|_{\R \times \bbS^1}$ solves \eqref{eq: Floer eq parameter} with parameter $\lambda \,$. 
The next lemma is standard and proofs in its flavor can be found in \cite{J_holomorphic_curves}.

\begin{lemma}
\label{lemma: uniformly bounded gradients and energy}
Let $(\lambda_n, u_n)_{n \in \N}$ be a sequence of solutions of \eqref{eq: Floer eq parameter}.
\begin{enumerate}[label=\alph*)]
    \item\label{it: uniformly bounded gradients} If the gradients $\lVert du_n \rVert_J$ are uniformly bounded, then a subsequence of $(u_n)_n$ converges in $\Cinfty(\bbS^2,M) \,$.
    \item\label{it: uniformly bounded energy} If the $u_n$ have uniformly bounded energy $E(u_n) \,$, then the $u_n$ have uniformly bounded gradients. 
\end{enumerate}
\end{lemma}

\begin{proof}
    Part \ref{it: uniformly bounded gradients} essentially follows from Arzel\'a-Ascoli and elliptic regularity. We provide a concise argument. Instead of viewing $u_n$ as solution of the inhomogeneous Cauchy-Riemann equation $\overline{\partial}_J \, u_n = (\nu_n(z, u_n))^{(0,1)} $ with $\nu_n((s,t),x) :=  dt \otimes X_{H^{\lambda_n}_{s,t}} (x) \,$, we consider its graph $\widetilde{u}_n(z) := (z, u_n(z)) \,$, which is a pseudoholomorphic curve for the almost complex structure $\widetilde{J}_n$ on $\bbS^2 \times M $ defined by
    \begin{equation*}
    \widetilde{J}_n ( z , x) := \left(\begin{matrix}
        j_{\bbS^2}(z) & 0 \\
       - J(x) \circ \nu_n(z,x) + \nu_n(z,x) \circ j_{\bbS^2}(z) & J(x)  
    \end{matrix} \right)  \,\, ,
    \end{equation*}
    see \cite[Ch.~8.1]{J_holomorphic_curves}. Up to a subsequence we may assume that $(\lambda_n)_n$ converges to some $\lambda \in I \,$, so that $\widetilde{J}_n$ converges in the $\Cinfty$-topology to an almost complex structure $\widetilde{J} \,$. Since the $\widetilde{u}_n$ have uniformly bounded gradients, a subsequence converges in $\Cinfty(\bbS^2,\bbS^2\times M)$ by \cite[Thm.~4.1.1]{J_holomorphic_curves}.
    
    For Part \ref{it: uniformly bounded energy}, we assume by contradiction that there exists a sequence $(z_n)_n \subseteq \bbS^2$ so that $\lVert du_n (z_n) \rVert_J \overset{n}{\rightarrow} \infty $ (up to extracting a subsequence of the $u_n$). Using that the norm of the vector field $X_{\Hparam}$ is uniformly bounded, the usual reparametrization trick (see \cite[Ch.~6.6]{Audin-Damian}) yields a non-constant $J$-holomorphic plane $v : \C \rightarrow M$ of finite energy $E(v) \leq \sup_n E(u_n) < \infty\,$. Removing the singularity at infinity (Proposition \ref{prop: removal-of-singularities}), we can extend $v$ to a non-constant $J$-holomorphic sphere $v:\bbS^2 \rightarrow M \,$. This is contradicting the symplectic asphericity of $M$ (see Remark \ref{rmk: symp aspherical -> no J-hol curves}).
\end{proof}

\begin{corollary}
\label{cor: uniformly bounded energy -> compactness}
Let $(\lambda_n, u_n)_{n} \subseteq I \times \Cinfty(\bbS^2,M) $ be a sequence of solutions of \eqref{eq: Floer eq parameter} with uniformly bounded energy. Then a subsequence converges in $I \times \Cinfty(\bbS^2,M)$ to some solution $(\lambda,u)$ of \eqref{eq: Floer eq parameter}.
\end{corollary}

\noindent Next we consider a parameter-independent Hamiltonian $H : M \times \bbS^1 \rightarrow \R$ and solutions to \eqref{eq: Floer eq}.

\begin{lemma}
\label{lemma: Floer compactness parameter-independent case}
Given a sequence $(u_n)_{n \in \N} \subseteq \Cinfty(\R \times \bbS^1,M)$ with uniformly bounded energy $E(u_n) = \int_{\R \times \bbS^1} |\partial_s u_n|_J^2 \, ds \, dt \,$. Suppose moreover, for some sequence $(R_n)_n \subseteq [0,+\infty]$ tending to infinity, that $u_n$ is a solution of \eqref{eq: Floer eq} on $(-R_n,R_n) \times \bbS^1 \,$. Then a subsequence of $(u_n)_n$ converges in $\Cinftyloc(\R \times \bbS^1,M)$ to a solution $u : \R \times \bbS^1 \rightarrow M$ of \eqref{eq: Floer eq} of energy $E(u) \leq \sup_n E(u_n) < \infty \,$.
\end{lemma}

\begin{proof}
    It suffices to show for every $R \geq 0$ that a subsequence converges in $\Cinfty([-R,R] \times \bbS^1)$. (The lemma then follows by extracting iteratively, for increasing $n \in \N \,$, subsequences converging on $[-n,n] \times \bbS^1$ and finally taking the diagonal subsequence.) Uniformly bounded energy again implies uniformly bounded gradients on $(-R-1,R+1) \times \bbS^1$ and the proof proceeds as above.
\end{proof}

\begin{corollary}
\label{cor: convergence to critical pts at infty}
Let $u : \R \times \bbS^1 \rightarrow M$ be a finite-energy solution of \eqref{eq: Floer eq} and suppose $u(0,\,\cdot\,)$ is a contractible loop. Then for every sequence $(s_n)_{n \in \N} \subseteq \R$ tending to $+ \infty \,$, there exists a subsequence $(s_{n_k})_{k \in \N}$ and $x_- , x_+ \in \Crit(\actionH)$ so that
\[
u(\pm s_{n_k} , \,\cdot\,) \overset{k \rightarrow \infty}{\longrightarrow} x_\pm \quad \text{ in } \Cinfty(\bbS^1,M) \,\, .
\]
\end{corollary}

\begin{proof}
    Because $u$ has finite energy and $s \mapsto \actionH(u(s,\,\cdot\,))$ is decreasing, the limits 
    \[
    \lim_{s \rightarrow \pm \infty} \actionH(u(s,\,\cdot\,)) \in (-\infty,\infty)
    \] 
    exist.
    The reparametrized sequence $u (s_n + \,\cdot\comma \cdot\,)$ of solutions of \eqref{eq: Floer eq} has energy $E(u) \,$, so by the above lemma a subsequence (which we suppress in the notation) converges in $\Cinftyloc (\R \times \bbS^1,M)$ to a solution $u_+ $ of \eqref{eq: Floer eq}. Notice that $u_+(s,\,\cdot\,)$ is a contractible loop for every $s \in \R$ because the $u(s_n +s , \,\cdot\,)$ are. From 
    \[
    \actionH(u_+(\widetilde{s}, \,\cdot\,)) = \lim_{n \rightarrow \infty} \actionH(u(s_n + \widetilde{s} , \,\cdot\,))= \lim_{s \rightarrow +\infty} \actionH(u(s,\,\cdot\,)) \qquad \Forall \widetilde{s} \in \R
    \]
    we infer that $\actionH(u_+)$ is constant, so that $u_+$ has zero energy and thus is $s$-independent and sitting at a critical point $x_+ = u_+ (0,\,\cdot\,)  \in \loopspace$ of $\actionH \,$. Observe that $u(s_n,\cdot\,) \overset{n \rightarrow \infty}{\longrightarrow} x_+$ in $\Cinfty(\bbS^1,M) \,$.
    Repeating the same argument with $s_n$ replaced by $-s_n \,$, we extract a further subsequence and find $x_- \in \Crit(\actionH)$ so that also $u(-s_n, \,\cdot\,) \overset{n \rightarrow \infty}{\longrightarrow}  x_- \,$.
\end{proof}

\subsection{The Perturbed Floer Equation}
\label{subsec: Perturbed Floer Equation}

Let $H : M \times \bbS^1 \rightarrow \R$ be an arbitrary smooth Hamiltonian whose time-$1$ flow $\Phi_{H}^1$ is not the identity. We fix once and for all an $\omega$-compatible almost complex structure $J$ on $M$.
We also fix a family of compactly supported bump functions $(\beta_R)_{R \geq 0} \subseteq \Cinfty(\R,[0,1]) $ such that
\begin{enumerate}[label=($\beta$.\arabic*)]
    \item\label{it: beta_r choice - smooth} the function $\R_{\geq 0} \times \R \rightarrow [0,1] \comma (R,s) \mapsto \beta_R(s) \,$, is smooth,
    \item\label{it: beta_r choice - beta_0} $\beta_0 \equiv 0 \,$,
    \item\label{it: beta_r choice - compact supp} $\beta_R (s) = 0$ for $s \notin (-R-1, R+1) \,$,
    \item\label{it: beta_r choice - one on [-R,R]} $\beta_R(s) = 1$ for $s \in [-R,R]$ and $R \geq 1 \,$,
    \item\label{it: beta_r choice - derivatives in- and decreasing} $\frac{d}{ds}\beta_R (s) \geq 0$ for $s \leq 0$ and $\frac{d}{ds}\beta_R (s) \leq 0$ for $s \geq 0 \,$.
\end{enumerate}
(See also Figure \ref{fig: beta_R bump functions}.) We note that $\lim_{R \rightarrow \infty} \beta_R \equiv 1$ in the pointwise sense.\\
We now interpolate between the $J$-holomorphic curve equation and \eqref{eq: Floer eq} and consider solutions to
\begin{equation}
\label{eq: Floer interpolated}
    \partial_s u(s,t) + J(u(s,t)) \, (\partial_t u (s,t) - \beta_R(s) \, X_{H_t}(u(s,t))) = 0 \,\, .
\end{equation}
For given $R$, solutions $u$ of \eqref{eq: Floer interpolated} are $J$-holomorphic curves outside $(-R-1,R+1) \times \bbS^1$ and, if $R \geq 1$, solutions of the Floer equation \eqref{eq: Floer eq} on $[-R,R] \times \bbS^1$. By removal-of-singularities, any finite-energy solution $u : \R \times \bbS^1 \rightarrow M$ of \eqref{eq: Floer interpolated} extends to a smooth map $u : \bbS^2 \rightarrow M \,$.

Observe that $\beta_R \, dt \otimes X_{H_t}(u)$ extends to a globally defined $u^*TM$-valued one-form on $\bbS^2$ because $\beta_R$ has compact support in $\R \,$.
Given a smooth map $u : \bbS^2 \rightarrow M \,$, its restriction to $\R \times \bbS^1$ is a solution of \eqref{eq: Floer interpolated} if and only if
\begin{equation}
\label{eq: Floer interpolated one-forms}
\overline{\partial}_J \, u - (\beta_R \, dt \otimes X_{H_t}(u))^{(0,1)} = ( du - \beta_R \, dt \otimes X_{H_t}(u))^{(0,1)} = 0 \quad \text{ on } \bbS^2 \,\, .
\end{equation}
Indeed, the complex antilinear one-form $( du - \beta_R \, dt \otimes X_{H_t}(u))^{(0,1)}$ is completely determined on the dense subset $\R \times \bbS^1 \subseteq \bbS^2$ by evaluating it with the vector field $\partial_s \,$.

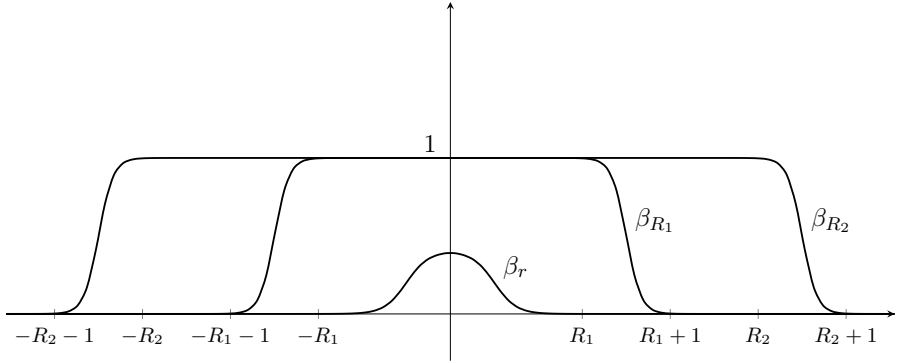
\begin{figure}
    \centering
    \begin{tikzpicture}

\begin{axis}[width = 15cm , height= 7cm,
xmin=-20.2, xmax=20.2, ymin=-0.3, ymax=2,
xtick = {-18,-14,-10,-6,6,10,14,18} , xticklabels = {$-R_2-1$, $-R_2$ ,$-R_1-1$,$-R_1$ ,$R_1$,$R_1+1$, $R_2$, $R_2+1$}, xticklabel style = {font=\footnotesize},
 ytick = {1}, yticklabel style = {yshift=6pt},
 scale=1, restrict y to domain=-1.5:1.2,
 axis x line=center, axis y line= center,
 samples=200,
]

\node   at (axis cs:8,0.6) [anchor=west] {$\beta_{R_1}$};
\node   at (axis cs:16,0.6) [anchor=west] {$\beta_{R_2}$};
\node   at (axis cs:2,0.3) [anchor=west] {$\beta_{r}$};

 \addplot[black, samples=100, smooth, domain=-20.2:20.2, thick]
 {
 ifthenelse( x < 0 , 0.5*(1+tanh(1.5*(x + 8)))  ,  0.5*(1+tanh(1.5*(-x + 8))) )  
 };

\addplot[black, samples=100, smooth, domain=-20.2:20.2, thick]
 {
 ifthenelse( x < 0 , 0.5*(1+tanh(1.5*(x + 16)))  ,  0.5*(1+tanh(1.5*(-x + 16))) )  
 };

 \addplot[black, samples=100, smooth, domain=-20.2:20.2, thick]
 {
 ifthenelse( x < 0 , 0.2*(1+tanh((x + 2)))  ,  0.2*(1+tanh((-x + 2))) )  
 };
 
\end{axis}

\end{tikzpicture}
    \caption{Bump functions $\beta_R$ for parameters $0 < r < 1 < R_1 < R_1+1 < R_2 \,$.}
    \label{fig: beta_R bump functions}
\end{figure}

\subsection{Moduli Spaces}
\label{subsec: Moduli Spaces}

By assumption $\Phi_H^1 \not= \id_M \,$, hence we can fix once and for all a point 
\begin{equation}
\label{eq: q_0 choice}
q_0 \in M \backslash \mathrm{Fix}(\Phi_H^1) \,\, .
\end{equation}
For every $R \geq 0$ we define the moduli spaces
\begin{align*}
    \moduli_R &:= \Bigset{u \in \Cinfty(\bbS^2,M)}{ \begin{matrix}
        \overline{\partial}_J \, u -  (\beta_R \, dt \otimes X_{H_t}(u))^{(0,1)} = 0 \\
        u(0,0) = q_0 \text{ and } [u] = 0 \in \pi_2(M,q_0)
    \end{matrix}} \\
    \moduli_{[0,R]} &= \set{(r,u) \in [0,R] \times \Cinfty(\bbS^2,M)}{u \in \moduli_r} \,\, .
\end{align*}

\noindent The moduli space $\moduli_0$ is particularly easy to understand, as the next lemma explains.

\begin{lemma}
\label{lemma: moduli_0 = pt}
$\moduli_0$ consists of exactly one point, namely the constant sphere $u_0 : \bbS^2 \rightarrow M$ through $q_0 \,$.
\end{lemma}

\begin{proof}
    Because $\beta_0 $ vanishes identically, elements $u \in \moduli_0$ are contractible $J$-holomorphic spheres through $q_0 \,$. We already know (see Remark \ref{rmk: symp aspherical -> no J-hol curves}) that the only $J$-holomorphic spheres to $M$ are the constant ones.
\end{proof}

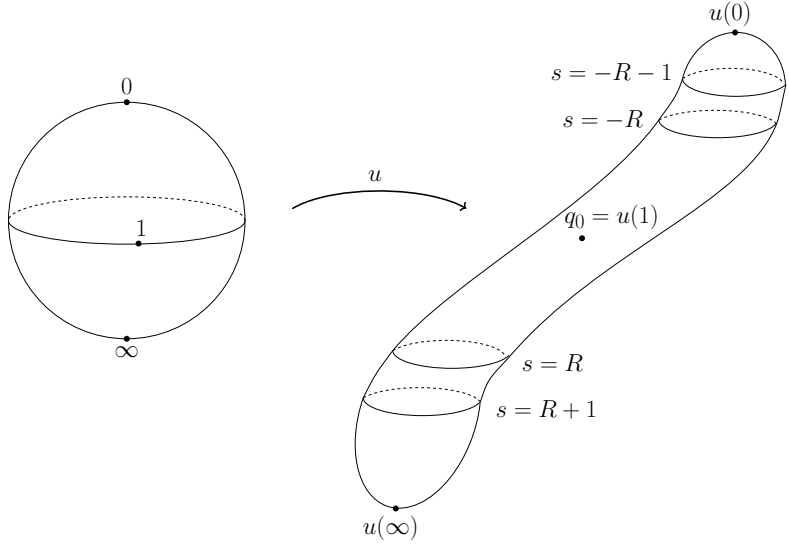
\begin{figure}
    \centering
    \begin{tikzpicture}[scale=0.3, x=7cm, y=7cm]
\tikzstyle{every node}=[font=\Huge]

\draw (-10,0) circle (2);
  \draw (-12,0) arc (180:360:2 and 0.4);
  \draw[dashed] (-8,0) arc (0:180:2 and 0.4);
  \fill[fill=black] (-10,2) circle (10pt) node[above, yshift=4pt]{$0$} ; 
  \fill[fill=black] (-10,-2) circle (10pt) node[below, yshift=-4pt]{$\infty$} ;
  \fill[fill=black] (-9.8,-0.39) circle (10pt) node[above, yshift=4pt, xshift=3pt]{$1$} ;

\draw[ultra thick, ->] (-7.2,0.2) arc (150:30:1.7 and 0.6);
\node at (-5.8,0.75) (u_map) {$u$};


 \draw[thick] (-5.5,-2.2) .. controls (-4.5,-1) and (-2,0.3) .. (-1,1.7);
 \draw[thick] (-3.55,-2.33) .. controls (-2,-0.5) and (0.5,0.3) .. (1,1.66) ;

 \draw[thick] (-5.5,-2.2) arc (180:345:1 and 0.3) ;
 \draw[dashed] (-5.5,-2.2) arc (155:3:1 and 0.3) ;
 \draw[thick] (-1,1.7) arc (180:350:1 and 0.3) ;
 \draw[dashed] (-1,1.7) arc (153:12:1.07 and 0.3) ;

 \draw[thick] (-5.5,-2.2) .. controls (-5.8,-2.6) ..  (-6,-3);
 \draw[thick] (-3.55,-2.33) .. controls (-3.9,-2.7) .. (-4.03,-3.1) ;
 \draw[thick] (-1,1.7) .. controls (-0.9,1.86) and (-0.7,2.05) .. (-0.6,2.4) ;
 \draw[thick] (1,1.66) .. controls (1.05, 1.8) and (1.1, 2.1) .. (1.15,2.3);
 \draw[thick] (-6,-3) arc (180:345:1 and 0.3) ;
 \draw[dashed] (-6,-3) arc (155:10:1.05 and 0.3) ;
 \draw[thick] (-0.6,2.4) arc (180:340:0.9 and 0.3) ;
 \draw[dashed] (-0.6,2.4) arc (155:10:0.9 and 0.3) ;
 
 \draw[thick, rotate around={-19:(-6,-3)}] (-6,-3) arc (180:380:0.98 and 1.6) ;
 \draw[thick, rotate around={-14:(-0.6,2.4)}] (-0.6,2.4) arc (180:18:0.88 and 1);

\fill[fill=black] (-2.3,-0.3) circle (10pt) node[above, yshift=5pt, xshift=30pt] (q_0) {$q_0 = u(1)$} ;
\fill[fill=black] (-5.45,-4.87) circle (10pt) node[below, yshift=-3pt, xshift=-5pt] (q_0) {$u(\infty)$} ;
\fill[fill=black] (0.29,3.18) circle (10pt) node[above, yshift=3pt, xshift=-5pt] (q_0) {$u(0)$} ;

\node at (-2.9,-3.2) {$s=R+1$} ;
\node at (-2.8,-2.4) {$s=R$} ;
\node at (-1.95,1.7) {$s=-R$} ;
\node at (-1.8,2.5) {$s=-R-1$} ;




\end{tikzpicture}
    \caption{Element $u : \bbS^2 = \C \mathrm{P}^1 \rightarrow M$ of the moduli space $\moduli_R$ for $R \geq 1 \,$. Via the embedding $\R \times \bbS^1 \hookrightarrow \C\mathrm{P}^1  \comma (s,t) \mapsto e^{2\pi(s+\i t)} \,$, the point $(s,t)=(0,0)$ corresponds to $1 \in \C\mathrm{P}^1 \,$. For $|s| \geq R+1$ the curve $u$ is $J$-holomorphic. The restriction $u|_{[-R,R] \times \bbS^1}$ is a Floer cylinder for Hamiltonian $H$.}
    \label{fig: moduli_R element}
\end{figure}

\subsection{Analytic Set-Up}
\label{subsec: Analytic Set-Up}

We fix once and for all an even integer $p > 2 \,$. It is well-known that 
\[
W^{1,p}(\bbS^2,M) = \set{u \in C^0(\bbS^2,M)}{u \in W^{1,p}_{\mathrm{loc}} \text{ in any local charts for } \bbS^2 \text{ and } M}
\]
is a second-countable Banach manifold without boundary. By Remark \ref{rmk: Banach mfds - topological properties} \ref{it: Banach mfds topological - second-countable}, consequently $W^{1,p}(\bbS^2,M)$ is also metrizable.

\begin{lemma}
\label{lem: W^1,p(S^2,M) admits smooth bump fcts}
$W^{1,p}(\bbS^2,M)$ admits smooth bump functions.
\end{lemma}

\begin{proof}
    Since $W^{1,p}(\bbS^2,M)$ is metrizable, it is in particular a regular topological space. It thus suffices to show that the local models of $W^{1,p}(\bbS^2,M)$ admit smooth bump functions respectively. Locally around $u \in \Cinfty(\bbS^1,M)$ the Banach manifold $W^{1,p}(\bbS^2,M)$ is modeled on the Banachable space $W^{1,p}(\bbS^2,E) \,$, with smooth vector bundle $E:=u^*TM$ over $\bbS^2 $ of rank $2m$.
    Covering $\bbS^2$ with finitely many charts $\varphi_j : U_j \overset{\sim}{\rightarrow} \Omega_j \subseteq \C \,$, which extend over the closure of $U_j\,$, and choosing smooth extendable trivializations $\Phi_j : E|_{U_j} \overset{\sim}{\rightarrow} U_j \times \R^{2m} \comma j=1,\ldots , \ell \,$, a norm inducing the given topology on $W^{1,p}(\bbS^2,E)$ is given by
    \[
    \left( \lVert \eta \rVert_{W^{1,p}(E)} \right)^p := \sum_{j=1}^\ell \left(\lVert \mathrm{pr}_2 \circ \Phi_j \circ \eta \circ \varphi_j^{-1} \rVert_{W^{1,p}(\Omega_j,\R^{2m})} \right)^p \qquad \Forall \eta \in W^{1,p}(\bbS^2,E) \,\, ,
    \]
    where $\mathrm{pr}_2$ denotes the projection to the second component.
    Since $p$ is an even integer, the $L^p$-norm $\lVert\cdot \rVert_{L^p}^p$ is smooth on $L^p(\Omega_j,\R^{2m}) $, see \cite[Thm.~1]{Lp_smooth_bump_fcts}, so that the $W^{1,p}$-norm $\lVert \cdot \rVert_{W^{1,p}}^p$ is smooth on $W^{1,p}(\Omega_j,\R^{2m})$. Since each map $\eta \mapsto \mathrm{pr}_2 \circ \Phi_j \circ \eta \circ \varphi_j^{-1}$ is bounded linear and thus smooth, the norm $\lVert \cdot \rVert^p_{W^{1,p}(E)} $ is smooth as well. We conclude with Remark \ref{rmk: Banach mfds - bump functions} \ref{it: Banach mfds bump fcts - Banach space} that $W^{1,p}(\bbS^2,E)$ admits smooth bump functions.
\end{proof}

\noindent Every smooth Banach submanifold of $W^{1,p}(\bbS^2,M)$ inherits second-countability, metrizability and existence of smooth bump functions. Now recall that we have fixed the point $q_0 \in M \backslash \mathrm{Fix}(\Phi_H^1)$ in \eqref{eq: q_0 choice} and consider the regular level set $\ev^{-1}(q_0)$ of the surjective submersion $\ev : W^{1,p}(\bbS^2,M) \rightarrow M \comma \ev(u) := u(0,0) \,$. The Banach manifold $\Bcal$ we will work with is the connected component of $\ev^{-1}(q_0)$ containing the constant sphere through $q_0 \,$, that is
\begin{equation}
\label{eq: Bcal def Schwarz thm}
\Bcal = \set{u \in W^{1,p}(\bbS^2,M)}{u(0,0) = q_0 \text{ and } [u] = 0 \in \pi_2(M,q_0)} \,\, .
\end{equation}
By the above considerations, $\Bcal$ is a second-countable, metrizable Banach manifold without boundary which admits smooth bump functions.
\begin{remark}
\label{rmk: tangent space at B}
Denote by $\pi_{TM} : TM \rightarrow M$ the base point projection. The tangent space of $W^{1,p}(\bbS^2,M)$ at $u$ is given by
\[
T_u W^{1,p}(\bbS^2,M) =\set{\xi \in W^{1,p}(\bbS^2,TM)}{\pi_{TM} \circ \xi = u} \,\, .
\]
The differential of the evaluation at $u$ is $d \ev (u) \, [\xi] = \xi(0,0)$ and it is easy to construct a continuous right-inverse of $d \ev (u) : T_u W^{1,p}(\bbS^2,M) \rightarrow T_{u(0,0)} M \,$, showing that $\ev$ indeed is a submersion. It follows that the tangent space of $\Bcal$ at $u \in \Bcal$ is
\begin{equation}
\label{eq: tangent space at B}
T_u \Bcal = \ker d\ev(u) =\set{\xi \in W^{1,p}(\bbS^2,TM)}{\pi_{TM} \circ \xi = u \text{ and } \xi(0,0) = 0 }
\end{equation}
and that it has a closed complement in $T_u W^{1,p}(\bbS^2,M) $ of dimension $\dim(M) = 2m \,$.
\end{remark}

\noindent Over $\Bcal$ we have the smooth Banach bundle $\Ecal \rightarrow \Bcal$ whose fiber over $u \in \Bcal$ is given by
\[
\Ecal_u := L^p (\bbS^2, \Lambda^{0,1}T^*\bbS^2 \otimes_{\C} u^*TM) \,\, .
\]
As before, we also abbreviate $\Lambda^{0,1} := \Lambda^{0,1}T^*\bbS^2 \,$. The fiber of the complex vector bundle $\Lambda^{0,1} \otimes_{\C} u^*TM$ over $z \in \bbS^2$ consists of the complex antilinear maps $(T_z \bbS^2 , j_{\bbS^2}) \rightarrow (T_{u(z)} M , J_{u(z)} ) \,$. The zero section of $\Ecal$ will be denoted by $\mathcal{O}_{\Ecal} \,$. We also define
\[
\BcalHat := \R_{\geq 0} \times \Bcal \quad \text{ and } \BcalHat_{[0,R]} := [0,R] \times \Bcal \quad \text{ for every } R \geq 0 
\]
and denote by $\EcalHat \rightarrow \BcalHat$ respectively $\EcalHat_{[0,R]} \rightarrow \BcalHat_{[0,R]}$ the pullback of $\Ecal$ along the obvious projections $\BcalHat \rightarrow \Bcal$ respectively $\BcalHat_{[0,R]} \rightarrow \Bcal \,$. We define the (smooth) section
\[
\Fsection : \BcalHat \rightarrow \EcalHat \comma \quad \Fsection(R,u) = \Fsection_R(u) := \overline{\partial}_J \, u - (\beta_R \, dt \otimes X_{H_t}(u))^{(0,1)} \,\, .
\]
For fixed $R \,$, the restriction $\Fsection_R : \Bcal \rightarrow \Ecal $ is a section of $\Ecal \,$.

\begin{lemma}
\label{lemma: zeroes of Fsection = moduli}
$\moduli_R$ is the zero set of $\Fsection_R \,$, that is ${\Fsection_R}^{-1}(\mathcal{O}_{\Ecal}) = \moduli_R \,$.
\end{lemma}

\begin{proof}
If $\Fsection_R(u) = 0 \,$, then \eqref{eq: Floer interpolated one-forms} holds almost everywhere. By elliptic regularity $u$ is smooth and \eqref{eq: Floer interpolated one-forms} holds in fact everywhere (see \cite[App.~B.4]{J_holomorphic_curves}). This shows the inclusion ${\Fsection_R}^{-1}(\mathcal{O}_{\Ecal}) \subseteq \moduli_R \,$. The converse inclusion is trivial.
\end{proof}

\noindent By the preceding lemma the zero set of $\Fsection$ restricted to $\BcalHat_{[0,R]}$ is precisely $\moduli_{[0,R]} \,$.

\subsection{Compactness of Moduli Spaces}
\label{subsec: Compactness Moduli Spaces}

Our next concern is the crucial compactness property of the moduli spaces $\moduli_{[0,R]} \,$. We first recall that the \textit{Hofer norm} $\lVert H \rVert_{\mathrm{Hof}}$ of $H$ is defined as
\[
\lVert H \rVert_{\mathrm{Hof}} := \int_0^1  (\max_M H_t - \min_M H_t ) \, dt \,\in \R_{\geq 0 } \,\, .
\]

\begin{lemma}
\label{lemma: Energy bounded by Hofer norm}
For every $R \geq 0$ and every $u \in \moduli_R$ it holds $E(u) \leq \lVert H \rVert_{\mathrm{Hof}} < \infty \,$.
\end{lemma}

\begin{proof}
The proof follows the lines of \cite[Lemma 4.1]{Quantum-Cup-length} and \cite[Lemma 3.1]{Albers_Momin}. In the following, for $u : \bbS^2 \rightarrow M$ and $s \in \R$ we abbreviate $u_s := u(s,\,\cdot\,) \in \loopspace \,$.
Consider the $(R,s)$-dependent action functional
\[
\action{R,s} : \loopspace \rightarrow \R \comma \quad \action{R,s}(x) := - \int_{\mathbbm{D}^2} \overline{x}^* \omega + \beta_R(s) \int_0^1  H_t(x(t)) \, dt \,\, ,
\]
where $\overline{x} : \mathbbm{D}^2 \rightarrow M$ is filling disk for the contractible loop $x \,$.
Its gradient for the $L^2$-metric, defined in \eqref{eq: L^2 metric}, is
\[
\nabla \action{R,s} (x) = J(x) (\partial_t x - \beta_R(s) \, X_{H_t}(x)) \,\, ,
\]
so that solutions $(R, u : \bbS^2 \rightarrow M)$ of \eqref{eq: Floer interpolated} (or equivalently of \eqref{eq: Floer interpolated one-forms}) satisfy
\[
\frac{d}{ds} u_s + \nabla \action{R,s}(u_s) = 0 \qquad \Forall s \in \R \,\, .
\]
and hence
\begin{align*}
E(u) &= \int_{-\infty}^{+\infty} \Big\lVert \frac{du_s}{ds} \Big\rVert_{L^2}^2 \, ds = - \int_{-\infty}^{+\infty} d \action{R,s}(u_s) \, \Big[ \frac{du_s}{ds} \Big] \, ds \\
&= -\int_{-\infty}^{+\infty} \frac{d}{ds} \action{R,s}(u_s) \, ds + \int_{-\infty}^{+\infty} \frac{\partial \action{R,s}}{\partial s}(s, u_s) \, ds \\
&=\action{R,-\infty}(u_{-\infty})- \action{R,+\infty}(u_{+\infty}) + \int_{-\infty}^{+\infty} \beta_R'(s) \int_0^1 H_t(u_s(t)) \, dt \, ds \\
&= \int_{-\infty}^{+\infty} \beta_R'(s) \int_0^1 H_t(u(s,t)) \, dt \, ds \,\, ,
\end{align*}
where for the last equation we used that $\beta_R(\pm \infty) = 0 $ and that $u_{\pm \infty}$ are constant loops. Using Properties \ref{it: beta_r choice - compact supp} and \ref{it: beta_r choice - derivatives in- and decreasing} of the family $(\beta_r)_{r \geq 0}$ of bump functions, this can be further simplified to
\begin{align*}
    E(u) &= \int_{-R-1}^{0} \underbrace{\beta_R'(s)}_{\geq 0} \int_0^1 H_t(u(s,t)) \, dt \, ds + \int_{0}^{R+1} \underbrace{\beta_R'(s)}_{\leq 0} \int_0^1 H_t(u(s,t)) \, dt \, ds \\
    &\leq \int_{-R-1}^{0}\beta_R'(s) \int_0^1 \max_M H_t \, dt \, ds + \int_{0}^{R+1} \beta_R'(s)\int_0^1 \min_M H_t \, dt \, ds \\
    &= \underbrace{\beta_R(0)}_{\in [0,1]} \left( \int_0^1 \max_M H_t \, dt - \int_0^1 \min_M H_t \, dt\right) \leq \lVert H \rVert_{\mathrm{Hof}}  \,\, .
\end{align*}
The penultimate step uses the fundamental theorem of calculus and $\beta_R(-R-1) = 0 = \beta_R(R+1) \,$. This shows that solutions $(R, u: \bbS^2 \rightarrow M)$ of \eqref{eq: Floer interpolated one-forms} satisfy the energy bound $E(u) \leq \lVert H \rVert_{\mathrm{Hof}}$ and in particular this holds whenever $u \in \moduli_R \,$.
\end{proof}

\begin{proposition}
\label{prop: moduli space compact}
For every $R \geq 0$ the moduli space $\moduli_{[0,R]}$ is a compact subset of $\BcalHat_{[0,R]} = [0,R] \times \Bcal \,$.
\end{proposition}

\begin{proof}
    Elements $(r,u) \in \moduli_{[0,R]}$ of the moduli space are solutions of the parameter-dependent Floer equation \eqref{eq: Floer interpolated} with Hamiltonian $(r,s,t,x) \mapsto \beta_r(s) H_t(x) \,$, which vanishes for $|s| \geq R+1 \,$. Now Corollary \ref{cor: uniformly bounded energy -> compactness}, which can be applied due to the uniform energy bound in the preceding Lemma \ref{lemma: Energy bounded by Hofer norm}, implies sequential compactness of $\moduli_{[0,R]}$ in the subspace topology of $\BcalHat_{[0,R]}$ (and even in $[0,R] \times \Cinfty(\bbS^2,M)$). Because $\BcalHat_{[0,R]}$ is metrizable, this shows the asserted compactness of $\moduli_{[0,R]} \,$.
\end{proof}

\noindent Let us emphasize again that Corollary \ref{cor: uniformly bounded energy -> compactness}, and therefore the just proven compactness of $\moduli_{[0,R]} \,$, crucially rely on $(M,\omega)$ being symplectically aspherical.

\subsection{Analyzing the Vertical Differential}
\label{subsec: Vertical Differential}

Recall that $\Dvert{u} \Fsection_R : T_u \Bcal \rightarrow \Ecal_u$ denotes the vertical differential of $\Fsection_R$ at $u \in {\Fsection_R}^{-1}(\mathcal{O}_{\Ecal}) = \moduli_R \,$.

\subsubsection{Fredholm property}
\label{subsubsec: D^vert Fredholm}
 We aim to show the following proposition.
\begin{proposition}
\label{prop: D^vert Fsection_R Fredholm + index}
For every $R \geq 0$ and every $u \in \moduli_R \,$, the vertical differential
$\Dvert{u} \Fsection_R$ is a Fredholm operator of index zero.
\end{proposition}

\begin{corollary}
\label{cor: D^vert Fsection Fredholm + index}
For every $(R,u) \in \BcalHat = \R_{\geq 0} \times \Bcal$ with $\Fsection(R,u) = 0\,$, the vertical differential $\Dvert{(R,u)} \Fsection : T_{(R,u)} \BcalHat \rightarrow \EcalHat_{(R,u)}$ is Fredholm of index $1$.
\end{corollary}

\begin{proof}
    Combine Proposition \ref{prop: D^vert Fsection_R Fredholm + index} with Lemma \ref{app lemma: Fredholm finite dim summands} in the appendix.
\end{proof}

\noindent To prove Proposition \ref{prop: D^vert Fsection_R Fredholm + index}, we first need to introduce some notation, which will also play a role in Subsection \ref{subsubsec: Transversality to Zero Section} below.

The section $\Fsection_R : \Bcal \rightarrow \Ecal$ admits a natural smooth extension to a section of
\[
\bigcup_u L^p (\bbS^2, \Lambda^{0,1} \otimes_{\C} u^*TM) \rightarrow W^{1,p}(\bbS^2,M)
\]
by the same formula $u \mapsto \overline{\partial}_J \, u - (\beta_R \, dt \otimes X_{H_t}(u))^{(0,1)} \,$. Abusing notation, we will also write $\Fsection_R$ for this extension and $\Ecal_u$ for the fibers of the extended vector bundle over $W^{1,p}(\bbS^2,M) \,$. 

In the following, we adopt the notation of \cite[Ch.~3.1]{J_holomorphic_curves}. Let $\nabla$ denote the Levi-Civita connection for $g_J = \omega(\,\cdot \, , J\, \cdot\,)$ and consider the connection
\[
\widetilde{\nabla}_X Y := \nabla_X Y - \tfrac{1}{2} J (\nabla_X J) Y \,\, ,
\]
which preserves $g_J$ and $J$ (but may have torsion). 
The principal part of $\Fsection_R$ in the trivialization induced by $\widetilde{\nabla}$ centered at $u \in \Cinfty(\bbS^2,M)$ then is
\begin{equation}
\label{eq: vertical part Fsection_R extended}
   T_u W^{1,p}(\bbS^2,M) = W^{1,p}(\bbS^2,u^*TM) \rightarrow \Ecal_u \comma \quad \xi \mapsto \Phi_u(\xi)^{-1} \, \Fsection_R (\exp_u (\xi)) \,\, . 
\end{equation}
Here the exponential map is with respect to $g_J$ and 
\begin{equation}
\label{eq: parallel transport along nabla^tilde}
\Phi_u (\xi) :  u^*TM \overset{\sim}{\rightarrow} \exp_u(\xi)^{*}TM
\end{equation}
is the complex bundle isomorphism given between the fibers over $z \in \bbS^2$ by $\widetilde{\nabla}$-parallel transport along the geodesic $\tau \mapsto \exp_{u(z)}(\tau \,\xi(z)) \,$. We denote by
\[
D_u \Fsection_R : T_u W^{1,p}(\bbS^2,M) \rightarrow \Ecal_u
\]
the differential at the origin of the map given in \eqref{eq: vertical part Fsection_R extended} and point out that
\[
D_u \Fsection_R = \Dvert{u} \Fsection_R  \qquad \text{if } \Fsection_R(u)= 0 \,\, .
\]

\begin{proof}[Proof of Proposition \ref{prop: D^vert Fsection_R Fredholm + index}]
Let $u \in \moduli_R$ be arbitrary.
By Lemma \ref{app lemma: Fredholm finite dim summands} in the appendix, keeping in mind that $T_u \Bcal$ is a complemented subspace in $T_u W^{1,p}(\bbS^2,M)$ of codimension $2m$ (see Remark \ref{rmk: tangent space at B}),
in order to prove the proposition it suffices to show that the vertical differential
\[
D_u \Fsection_R = \Dvert{u} \Fsection_R  : \, T_u W^{1,p}(\bbS^2,M) \rightarrow \Ecal_u
\]
of the extended section is Fredholm of index $2m \,$.
Now recall that $[u] = 0 \in \pi_2(M,q_0)$ since $u \in \moduli_R \,$. The proposition thus follows from the next Lemma \ref{lemma: D_u Fredholm + index}.
\end{proof}

\begin{lemma}
\label{lemma: D_u Fredholm + index}

For every $u \in \Cinfty(\bbS^2,M)$ which is null-homotopic, the operator $D_u \Fsection_R : T_u W^{1,p}(\bbS^2,M) \rightarrow \Ecal_u$ is Fredholm of index $2m \,$.
\end{lemma}

\begin{proof}
    Observe that $\Fsection_R$ is the sum of the Cauchy-Riemann section $\overline{\partial}_J$ and a perturbation
    \[
   \sigma: \quad \widetilde{u} \mapsto - (\beta_R \, dt \otimes X_{H_t}(\widetilde{u}))^{(0,1)}
    \]
    which does not differentiate $\widetilde{u} \,$. Hence $D_u \sigma$ factors continuously through the subspace
    \[
    W^{1,p}(\bbS^2, \Lambda^{0,1} \otimes_{\C} u^*TM) \subseteq L^p(\bbS^2, \Lambda^{0,1} \otimes_{\C} u^*TM) \,\, .
    \]
    Because the inclusion $W^{1,p} \subseteq L^p$ is compact, $D_u \sigma$ is a compact operator. Thus $D_u \Fsection_R$ is Fredholm iff $D_u \overline{\partial}_J$ is and, if this is the case, their Fredholm indices coincide. And indeed, by \cite[Prop.~3.1.11]{J_holomorphic_curves}, the operator $D_u \overline{\partial}_J$ is Fredholm and its index computes to
    \[
    \mathrm{ind}(D_u \overline{\partial}_J) = m (2 - 2g) + 2 c_1(u^*TM) \,\, ,
    \]
    where $g$ is the genus of the underlying Riemann surface and $c_1(u^*TM)$ denotes the first Chern number of $(u^*TM,u^*J) \,$.
    The genus of $\bbS^2$ is $g=0$ and $c_1(u^*TM) = 0$ due to the assumption that $u$ is null-homotopic. This proves the lemma.
\end{proof}

\subsubsection{Transversality to the zero section}
\label{subsubsec: Transversality to Zero Section}

Although the zero section $\moduli_{[0,R]}$ of $\Fsection|_{[0,R] \times \Bcal}$ might fail to be a manifold, we at least have automatic transversality to the zero section for the parameter $R = 0 \,$.

\begin{proposition}
\label{prop: Fsection_0 transverse to zero section}
$\Fsection_0 = \overline{\partial}_J : \Bcal \rightarrow \Ecal$ is transverse to the zero section.
\end{proposition}

\begin{proof}
    Observe that $\Fsection_0 = \overline{\partial}_J$ because $\beta_0 $ vanishes identically.
    We already know (Proposition \ref{prop: D^vert Fsection_R Fredholm + index}) that the vertical differential of $\Fsection_0$ at a zero is Fredholm, so its kernel automatically has a closed complement. Thus transversality to the zero section is equivalent to surjectivity of $\Dvert{u_0} \overline{\partial}_J : T_{u_0} \Bcal \rightarrow \Ecal_{u_0}$ at the constant sphere $u_0$ through $q_0$ (which is the unique zero of $\Fsection_0 \,$, see Lemma \ref{lemma: moduli_0 = pt}). Because $\Dvert{u_0} \overline{\partial}_J$ has Fredholm index zero, again by Proposition \ref{prop: D^vert Fsection_R Fredholm + index}, it suffices to show injectivity to prove surjectivity.
    
    The differential of $\overline{\partial}_J$ at an arbitrary $u \in \Cinfty(\bbS^2,M)$ is given explicitly by
    \[
    D_u \overline{\partial}_J : T_u W^{1,p}(\bbS^2,M) \rightarrow \Ecal_u \comma \quad \xi \mapsto (\nabla \xi)^{(0,1)} - \tfrac{1}{2} J(u) \, (\nabla_\xi J)(u) \, \partial_J  (u)  \,\, ,
    \]
    see \cite[Prop.~3.1.1]{J_holomorphic_curves}. Because $u_0$ is constant, its differential simplifies to
    \begin{align*}
    \Dvert{u_0} \overline{\partial}_J = D_{u_0}\overline{\partial}_J : T_{u_0} W^{1,p}(\bbS^2,M) \rightarrow \Ecal_{u_0} \comma \quad \xi \mapsto (\nabla \xi)^{(0,1)} = \tfrac{1}{2} (\nabla \xi + J(q_0) \circ \nabla \xi \circ j_{\bbS^2}) 
    \end{align*}
    and its restriction to $T_{u_0} \Bcal$ is the vertical differential of $\overline{\partial}_J : \Bcal \rightarrow \Ecal$ at $u_0 \,$.
    
     Moreover, still because $u_0$ is constant, tangent vectors $\xi \in T_{u_0} W^{1,p}(\bbS^2,M)$ are just $W^{1,p}$-functions $\xi : \bbS^2 \rightarrow T_{q_0} M$ and $\nabla \xi$ is the usual differential $d \xi$ of $\xi$ as map to the vector space $T_{q_0} M \,$.
    If $\xi \in T_{u_0} W^{1,p}(\bbS^2,M) $ is in the kernel of $D_{u_0} \overline{\partial}_J \,$, then $\xi$ must be constant (see \cite[Lemma 6.7.6]{J_holomorphic_curves}): Indeed, the restriction of $\nabla \xi = d\xi$ to the punctured Riemann sphere $\C \mathrm{P}^1\backslash\{\infty\} = \C$ is a one-form on $\C$ with values in $T_{q_0} M$ for which
    \[
    d\xi + J(q_0) \circ d\xi \circ \i = 0 \,\, .
    \]
    Hence $\xi : (\C,\i) \rightarrow (T_{q_0} M,J(q_0))$ is an entire holomorphic function which extends continuously to $\C \mathrm{P}^1 \,$. Such a function must be constant by Liouville's theorem. 

    A tangent vector $\xi \in T_{u_0} \Bcal$ in the kernel of $\Dvert{u_0} \overline{\partial}_J : T_{u_0} \Bcal \rightarrow \Ecal_{u_0}$ therefore is constant and additionally satisfies $\xi(0,0) = 0$ by \eqref{eq: tangent space at B}, hence is the zero function.
\end{proof}

\begin{remark}
\label{rmk: Fsection_0 transverse - surjectivity}
One can also verify surjectivity of $D_{u_0} \overline{\partial}_J |_{T_{u_0} \Bcal}$ directly by using the fundamental solution of the Cauchy-Riemann operator. Then one can even read of the Fredholm index of $D_u \Fsection_R|_{T_u \Bcal}$ at arbitrary $(R,u)$ without invoking any formula: At $(R,u) = (0,u_0)$ the index is $\dim (\moduli_0) = 0$ due to transversality. Because $\Cinfty(\bbS^2,M) \rightarrow \Z \comma u \mapsto \mathrm{ind}(D_u \overline{\partial}_J) \,$, is a continuous function and $\Bcal \cap \Cinfty(\bbS^2,M)$ is path-connected (both in the $\Cinfty$-topology respectively), the index 
\[
\mathrm{ind}(D_u \Fsection_R|_{T_u \Bcal}) = \mathrm{ind}(D_u \overline{\partial}_J) - 2m 
\]
must be zero for all $R \geq 0$ and $u \in \Bcal \cap \Cinfty(\bbS^2,M)\,$.
\end{remark}

\subsection{Applying the Perturbation Theorem}
\label{subsec: Schwarz thm - perturbation thm application}

For fixed $R > 0$ we will apply Corollary \ref{mcor: perturbation thm} of the perturbation theorem to the Banach manifold $\BcalHat_{[0,R]} = [0,R] \times \Bcal $ and the section $\Fsection : \BcalHat_{[0,R]} \rightarrow \EcalHat_{[0,R]} \,$. Notice that $\BcalHat_{[0,R]}$ admits smooth bump functions because $\Bcal$ does.
To apply the perturbation theorem, we also need a B-bundle pair $(\EcalHat_{[0,R]}, \EcalHat_{[0,R]}^1)$ over $\BcalHat_{[0,R]} \,$. To this end, consider the bundle $\Ecal^1 \rightarrow \Bcal$ whose fiber over $u \in \Bcal$ is given by
\[
\Ecal_u^1 := W^{1,p} (\bbS^2, \Lambda^{0,1} \otimes_{\C} u^*TM) \subseteq \Ecal_u \,\, .
\]

\begin{remark}
\label{rmk: Ecal^1_u def Schwarz theorem}
Apart from the higher regularity of the sections, this is the same definition as the one of $\Ecal_u \,$. However, since $u$ is now of the same regularity $W^{1,p}$ as sections in its fiber $\Ecal^1_u \,$, the above definition is slightly more subtle than before. We define $W^{1,p}(\bbS^2,\Lambda^{0,1} \otimes_{\C} u^*TM)$ to be the space of continuous sections $\eta$ of the complex $C^0$-bundle $\Lambda^{0,1} \otimes_{\C} u^*TM$ over $\bbS^2$ so that $(u, \eta(X) ) \in W^{1,p}(\bbS^2,TM)$ for every smooth vector field $X \in \mathfrak{X}(\bbS^2) \,$.
\end{remark}

\noindent It is not difficult to see that $(\Ecal, \Ecal^1)$ is a B-bundle pair over $\Bcal \,$. Indeed, a vector bundle chart (\ie a trivialization combined with a chart on the base) of $\Ecal$ over a chart domain centered at a smooth $u \in \Cinfty(\bbS^2,M)$ is given via the pointwise parallel transport map $\Phi_{u}$ introduced in \eqref{eq: parallel transport along nabla^tilde}. Such a vector bundle chart maps the fibers of $\Ecal^1$ to
\[
W^{1,p}(\bbS^2,\Lambda^{0,1} \otimes_{\C} u^*TM) \subseteq L^p(\bbS^2,\Lambda^{0,1} \otimes_{\C} u^*TM) \,\, .
\]
The above pair $(L^p,W^{1,p})$ is a B-pair. Details are carried out in Appendix \ref{app subsec: B-bundle pair Schwarz}. The B-bundle pair $(\Ecal,\Ecal^1)$ over $\Bcal$ pulls back to a B-bundle pair $(\EcalHat_{[0,R]} , \EcalHat_{[0,R]}^1)$ over $\BcalHat_{[0,R]} = [0,R] \times \Bcal \,$.
\\

\noindent The following lemma relies on the perturbation theorem and is the crucial ingredient in the proof of Theorem \ref{mthm: Schwarz theorem}. It is here were we deviate the most from \cite{J_holomorphic_curves}.

\begin{lemma}
\label{lemma: moduli_R non-empty}
$\moduli_R \not= \varnothing$ for every $R \geq 0 \,$.
\end{lemma}

\begin{proof}
    By Lemma \ref{lemma: moduli_0 = pt}, $\moduli_0 = \{u_0\}$ consists of a single point. Now let $R > 0$ be arbitrary and assume by contradiction that $\moduli_R$ is empty. 
    The smooth section $\Fsection : \BcalHat_{[0,R]} = [0,R] \times \Bcal \rightarrow \EcalHat_{[0,R]}$ is a Fredholm section of index $d = 1$ (Corollary \ref{cor: D^vert Fsection Fredholm + index}) and has compact zero set $\moduli_{[0,R]}$ (Lemma \ref{lemma: zeroes of Fsection = moduli}, Proposition \ref{prop: moduli space compact}). Moreover, its restriction to the boundary $\{0,R\} \times \Bcal$ is transverse to the zero section as well. Indeed, $\Fsection_0$ is transverse to the zero section by Proposition \ref{prop: Fsection_0 transverse to zero section} and $\Fsection_R$ is transverse to the zero section since its zero set $\moduli_R$ is empty by assumption. By Corollary \ref{mcor: perturbation thm} there exists a smooth compact $1$-dimensional manifold whose boundary is
    \[
    \moduli_{[0,R]} \,\cap \, \partial ([0,R] \times \Bcal) = \moduli_0 \cup \moduli_R = \moduli_0 = \{u_0\} \,\, .
    \]
    But every compact $1$-dimensional manifold has an even number of boundary points. Contradiction.
\end{proof}

\subsection{Finish}
\label{subsec: Finish Schwarz Theorem}

We are now ready to finalize the proof of Schwarz's theorem. Recall from \eqref{eq: q_0 choice} that we have chosen $q_0 \in M \backslash \mathrm{Fix}(\Phi_H^1)$ and that $u(0,0) = q_0$ for every $u \in \moduli_R \,$.

\begin{proof}[Proof of Theorem \ref{mthm: Schwarz theorem}]
    Fix any sequence $(R_n)_n \subseteq \R_{\geq 1}$ tending to infinity. By Lemma \ref{lemma: moduli_R non-empty}, for each $n \in \N$ we can choose $u_n \in \moduli_{R_n} \,$. Then $u_n$ solves \eqref{eq: Floer interpolated} and thus is a solution of \eqref{eq: Floer eq} on $[-R_n,R_n] \times \bbS^1 \,$. Moreover $E(u_n) \leq \lVert H \rVert_{\mathrm{Hof}}$ by Lemma \ref{lemma: Energy bounded by Hofer norm}. By Lemma \ref{lemma: Floer compactness parameter-independent case}, up to a subsequence (which we suppress in the notation), the sequence $(u_n)_n$ converges in $\Cinftyloc(\R \times \bbS^1,M)$ to a finite-energy solution $u : \R \times \bbS^1 \rightarrow M$ of \eqref{eq: Floer eq}. Notice that each $u_n(s,\,\cdot\,)$ is a contractible loop ($u_n$ being defined on the sphere), so that its limit $u(s,\,\cdot\,) $ is contractible as well. Moreover $u(0,0) = \lim_n u_n(0,0) = q_0 \,$. By Corollary \ref{cor: convergence to critical pts at infty}, there exists a sequence $(s_n)_n \subseteq \R$ tending to $+\infty$ and $x_- , x_+ \in \Crit(\actionH)$ with
    \[
    u(\pm s_n, \,\cdot\,) \overset{n \rightarrow \infty}{\longrightarrow} x_\pm \quad \text{ in } \Cinfty(\bbS^1,M) \,\, .
    \]
    We have
    \begin{align*}
    0 \leq E(u) &= - \lim_{n \rightarrow \infty} \int_{-s_n}^{s_n} \frac{d}{ds} \actionH(u(s,\,\cdot\,)) \, ds =
    \lim_{n \rightarrow \infty} \Big( \actionH(u(-s_n,\,\cdot\,)) - \actionH(u(s_n,\,\cdot\,)) \Big) \\
    &= \actionH(x_-) - \actionH(x_+) \,\, .
    \end{align*}
    If $E(u) = 0 \,$, then $u$ is $s$-independent and $u(s,\,\cdot\,) = x_-=x_+ \in \Crit(\actionH) \,$. Consequently $q_0 = u(0,0) \in \mathrm{Fix}(\Phi_H^1) \,$, contradicting the choice of $q_0 \in M \backslash \mathrm{Fix}(\Phi_H^1)$.\\
    Hence $E(u) > 0$ and we have found critical points $x_-, x_+ \in \Crit(\actionH)$ with
    \[
    \actionH(x_-) > \actionH(x_+) \,\, .
    \]
    This concludes the proof.
\end{proof}

\section{Perturbing Evaluation Maps}
\label{sec: Perturbing Evaluation Maps}

\noindent In Section \ref{sec: Abstract Perturbation Thm}, our set-up was comparatively modest, in the sense that we only dealt with the zero set of a Fredholm section $\Fsection$ on a Banach manifold $\Bcal \,$. However in applications, where $\Bcal$ is a Banach manifold of maps having the finite-dimensional manifold $M$ as target, there often appears a natural evaluation map $e : \Bcal \rightarrow M$. For example, in Section \ref{sec: Application Schwarz thm} we worked with the evaluation map $e = \ev : u \mapsto u(0,0) \,$.
This particular evaluation map is smooth but in general, if the point at which one evaluates depends on $u \in \Bcal$, then $e$ is merely continuous.
Having perturbed the zero set $\ZeroFsection$, we now also perturb such an evaluation map. Our intention is to provide a reference for an upcoming paper. The main result of this section is Theorem \ref{mthm: evaluation maps - cobordism}. It very roughly says the following (the precise assumptions are stated in Subsection \ref{subsec: evaluation maps - cobordisms}): Let $\Fsection : [0,1] \times \Bcal \rightarrow \Ecal$ be a Fredholm section with compact zero set, $e$ be an evaluation on $[0,1] \times \Bcal$ and $Z \subseteq M$ be a submanifold and closed subset of $M$. If the number of zeroes $(0,x) \in \ZeroFsection$ with $e(0,x) \in Z $ is odd, then there also must exist a zero $(1,x) \in \ZeroFsection$ with $e(1,x) \in Z $.

\subsection{Generic Transversality}
\label{subsec: evaluatin maps - Generic Transversality}

In this subsection we fix
\begin{itemize}
    \item a compact finite-dimensional manifold $S$,
    \item a compact submanifold $Q \subseteq \partial S$ without boundary,
    \item a finite-dimensional manifold $M$ without boundary,
    \item a distance function $d_M : M \times M \rightarrow \R_{\geq 0}$ (that is, a metric) inducing the given topology on $M$,
    \item a countable collection $\{Z_n\}_{n \in \N}$ of submanifolds of $M$ without boundary.
\end{itemize}

\noindent The submanifolds $Z_n \subseteq M , n \in \N$, do not have to be distinct, that is, we allow for repetitions. Following the conventions described in Subsection \ref{subsec: Conventions}, all of the finite-dimensional manifolds above are smooth, second-countable and pure-dimensional.
In the following, transversality is denoted by $\pitchfork \,$. We write $\ball_r^k \subseteq \R^k$ (respectively $\ballclosed^k_r$) for the open (respectively closed) ball of radius $r$ and $\ball^k = \ball^k_1$ for the open unit ball (respectively $\ballclosed^k$ for the closed unit ball).

\begin{lemma}
\label{lemma: evaluation maps - transversality}
For $f \in \Cinfty(S,M)$ suppose that $f|_Q \pitchfork Z_n$ for every $n \in \N \,$. Then there exists $k \in \N \,$, a residual subset $B_{\mathrm{res}} \subseteq \ball^k$ and a smooth map $F : S \times \ball^k \rightarrow M$ so that, writing $F_a := F(\,\cdot\,,a)$ for $a \in \ball^k \,$, it holds
\begin{enumerate}[label=\arabic*)]
    \item\label{it: evaluation maps - F_0 = f} $F_0  = f \,$.
    \item\label{it: evaluation maps - F_a|_Q = f|_Q} $F_a|_Q = f|_Q$ for all $a \in \ball^k \,$.
    \item\label{it: evaluation maps - F_a transverse} $F_a  \pitchfork Z_n$ and $F_a|_{\partial S} \pitchfork Z_n$ for all $a \in B_{\mathrm{res}}$ and $n \in \N \,$.
\end{enumerate}
\end{lemma}

\begin{proof}
The proof is a minor adaption of \cite[Thm.~6.36]{Lee}. Embed $M$ into Euclidean space $\R^k $ and choose a tubular neighborhood $U_M \subseteq \R^k$ of $M$ and a retraction $r : U_M \rightarrow M$ (that is $r|_M = \id_M$) that is also a smooth submersion. Choose a constant $\eps > 0$ so that $f(x) +a \in U_M $ for all $x \in S$ and $a \in \ball^k_\eps\,$. Because $Q \subseteq S$ is a closed subset, we can moreover choose a function $\rho \in \Cinfty(S,[0,\eps])$ with $\rho^{-1}(0) = Q \,$. We now define
\[
F(x,a) := r(f(x) + \rho(x) \, a) \quad \text{ for } (x,a) \in S \times \ball^k \,\, .
\]
Properties \ref{it: evaluation maps - F_0 = f} and \ref{it: evaluation maps - F_a|_Q = f|_Q} are immediate since $r$ is a retraction. For every $x \in S$ with $\rho(x) >  0 \,$, the map
\[
\ball^k \rightarrow U_M \subseteq \R^k \comma \quad a \mapsto f(x) + \rho(x) \, a
\]
is a local diffeomorphism, so its composition with $r$ is a submersion. Hence $F(x,\,\cdot\,) : \ball^k \rightarrow M$ is a submersion for $x \in S \backslash Q \,$. On the other hand, $F|_{ Q \times \ball^k}$ is transverse to each $Z_n$ since $f|_Q$ is. Remembering that $Q \subseteq \partial S$, altogether this shows that
\begin{align*}
    \begin{cases}
    F \pitchfork Z_n & \Forall n \in \N \\
    F|_{\partial S \times \ball^k} \pitchfork Z_n & \Forall n \in \N \,\, .
    \end{cases}
\end{align*}
By parametric transversality, for every $n \in \N$ there exists a residual subsets $B_{n} \subseteq \ball^k$ so that both $F_a \pitchfork Z_n$ and $F_a|_{\partial S} \pitchfork Z_n$ for every $a \in B_n \,$. Taking the countable intersection $B_{\mathrm{res}} := \bigcap_{n \in \N} B_n$ gives the desired residual set. 
\end{proof}

\begin{corollary}
\label{cor: evaluation maps - transversality}
Let $f \in C^0(S,M)$ and suppose that $f|_Q$ is smooth and transverse to each $Z_n \comma n \in \N \,$. Then for every $\eps > 0$ there exists $h \in \Cinfty(S,M)$ with
\begin{enumerate}[label=\arabic*)]
    \item\label{it: evaluation maps - sim transversality - f = ftilde on Q} $f|_Q = h|_Q \,$.
    \item\label{it: evaluation maps - sim transversality - ftilde transverse} $h \pitchfork Z_n$ and $h|_{\partial S} \pitchfork Z_n$ for every $n \in \N \,$.
    \item\label{it: evaluation maps - sim transversality - ftilde close to f} $\max_S d_M(f, h) < \eps \,$.
\end{enumerate}
\end{corollary}

\begin{proof}
\textit{Step 1:} First, additionally assume that $f$ is smooth on $S$. Choose a smooth map $F : S \times \ball^k \rightarrow M$ and a residual subset $B_{\mathrm{res}} \subseteq \ball^k$ as in Lemma \ref{lemma: evaluation maps - transversality}. Because $S$ is compact, for given $\eps > 0$ there exists $\delta > 0$ so that 
\[ 
\max_S d_M(f , F_a) = \max_S d_M(F_0,F_a) < \eps \qquad \Forall a \in \ball^k_\delta \,\, .
\]
Now choose $a \in \ball^k_\delta \cap B_{\mathrm{res}} \not= \varnothing$ and set $h := F_a \,$.

\textit{Step 2:} We prove the general case. Since $f|_Q$ is smooth and $Q \subseteq S$ is a closed subset, we can find a smooth map $f_1 \in \Cinfty(S,M)$ that equals $f$ on $Q$ and with $\max_S d_M(f,f_1) < \eps/2 \,$. Such a map exists, following the proof of \cite[Thm.~6.26]{Lee}. Now apply Step 1 to $f_1$ and $\eps/2 \,$.
\end{proof}

\subsection{Existence of Perturbed Evaluation Maps}
\label{subsec: evaluation maps - Existence Perturbed Evaluation Maps}

As in the setting of the perturbation theorem (Theorem \ref{mthm: perturbation thm}), suppose we are given
\begin{itemize}
    \item a Banach manifold $\Bcal$ admitting smooth bump functions,
    \item a B-Bundle pair $(\Ecal,\Ecal^1)$ over $\Bcal\,$,
    \item a Fredholm section $\Fsection : \Bcal \rightarrow \Ecal$ of index $d + 1 \in \Z$ with compact zero set $\ZeroFsection \,$.
\end{itemize}
For reasons which will become apparent in the next subsection, it is convenient to let $\Fsection$ have index $d+1$ (instead of index $d$ as in Theorem \ref{mthm: perturbation thm}).

We now add some further structure to this.
\begin{itemize}
    \item Let $\Dcal \subseteq \Bcal$ be a union of boundary components of $\Bcal$ so that $\Fsection|_{\Dcal} : \Dcal \rightarrow \Ecal|_{\Dcal}$ is transverse to the zero section and moreover $\Dcal \cap S_{\Fsection} \not= \varnothing \,$.
\end{itemize}
Notice that $\Dcal$ is a closed subset of $\Bcal \,$, that condition \ref{it: perturbation thm assumption} in Theorem \ref{mthm: perturbation thm} is satisfied and that $\Dcal \cap \ZeroFsection$ is a compact smooth $d$-dimensional submanifold of $\Bcal$ without boundary. Suppose, we have additionally fixed
\begin{itemize}
     \item a finite collection $\Dcal_1,\ldots , \Dcal_\ell \subseteq \Bcal $ of closed subsets,
     \item a finite-dimensional smooth manifold $M$ without boundary and a distance function $d_M$ on $M$ inducing the given topology,
     \item finitely many smooth submanifolds $Z_1, \ldots , Z_k \subseteq M \,$, each without boundary,
     \item a continuous map $e : \Bcal \rightarrow M$ so that $e|_{\Dcal \cap \ZeroFsection}$ is smooth and transverse to $Z_j$ for every $j=1,\ldots ,k\,$.
\end{itemize}

\noindent Lastly, we fix finitely many \Bplus-sections $s_1,\ldots , s_n : \Bcal \rightarrow \Ecal$ with support contained in $\Bcal \backslash \Dcal \,$, a sufficiently small $\eps > 0$ and a residual subset $\Bres \subseteq \ball_{\eps}^n$ so that properties \ref{it: perturbation thm conclusion - Fredholm}\,-\,\ref{it: perturbation thm conclusion - zero set compact mfd} in Theorem \ref{mthm: perturbation thm} are satisfied for every parameter $\lambda \in \Bres $ and moreover the compactness assertion \ref{it: perturbation rmk - compactness} in Remark \ref{rmk: perturbation thm} holds. We know that such choices exist and let them be fixed but arbitrary in the following.
As in the proof of Theorem \ref{mthm: perturbation thm}, we also set
\[
 \Funiv(\,\cdot\, , \lambda) := \Fsection_\lambda :=\Fsection + \sum_{i=1}^n \lambda_i \, s_i \, : \, \Bcal \rightarrow \Ecal \quad \text{ where } \lambda =(\lambda_1,\ldots, \lambda_n) \in \R^n \,\, .
\]
By property \ref{it: perturbation thm conclusion - zero set compact mfd} in the perturbation theorem, for every $\lambda \in \Bres$ the zero set $S_{\Fsection_{\lambda}}$ is a compact $(d+1)$-dimensional smooth submanifold of $\Bcal$ with boundary $ \partial S_{\Fsection_{\lambda}} = S_{\Fsection_{\lambda}} \cap \partial \Bcal$ and with $\varnothing \not=  \ZeroFsection \cap \Dcal = S_{\Fsection_{\lambda}} \cap \Dcal \subseteq \partial S_{\Fsection_{\lambda}} \,$. In particular $S_{\Fsection_\lambda}$ is non-empty.

\begin{proposition}
\label{prop: perturbed evaluations}
For given $\delta > 0$ there exists $\lambda \in \Bres$ and a smooth map $e_{\lambda} : S_{\Fsection_{\lambda}} \rightarrow M$ with the following properties:
\begin{enumerate}[label=\arabic*)]
    \item\label{it: evaluation maps - distance} $d_M(e, e_{\lambda}) < \delta $ on $S_{\Fsection_{\lambda}}$.
    \item\label{it: evaluation maps - equal on Dcal} $e_{\lambda} = e$ on $S_{\Fsection_{\lambda}} \cap \Dcal = \ZeroFsection \cap \Dcal \,$.
    \item\label{it: evaluation maps - transverse} both $e_{\lambda}$ and $e_{\lambda}|_{\partial S_{\Fsection_{\lambda}} } $ are transverse to each $Z_j \comma j=1,\ldots , k\,$.
    \item\label{it: evaluation maps - empty preimage} for every $i=1,\ldots , \ell$ and $j=1,\ldots ,k$ it holds
    \begin{align*}
        \set{x \in \Dcal_i \cap \ZeroFsection}{e(x) \in \overline{Z_j}} = \varnothing \quad \Longrightarrow \quad (e_{\lambda})^{-1}(\,\overline{Z_j}\,) \cap \Dcal_i = \varnothing \,\, .
    \end{align*}
    Here $\overline{Z_j}$ denotes the closure of $Z_j$ in $M$.
\end{enumerate}
\end{proposition}

\noindent The proposition is summarized as diagram in Figure \ref{diagram: perturbing evaluation map}.

\begin{proof}[Proof of Proposition \ref{prop: perturbed evaluations}]
    Let $\delta > 0$ be arbitrary but fixed. Choose a sequence $(\delta_m)_{m \in \N} \subseteq (0,\delta)$ tending to zero and a sequence $(\lambda_m)_{m \in \N} \subseteq \Bres$ tending to the origin in $\R^n \,$. We now apply Corollary \ref{cor: evaluation maps - transversality} to the compact $(d+1)$-dimensional manifold $S_{\Fsection_{\lambda_m}} \,$, its submanifold $  \Dcal \cap \ZeroFsection = \Dcal \cap S_{\Fsection_{\lambda_m}}\subseteq \partial S_{\Fsection_{\lambda_m}}$ and the continuous map $e|_{S_{\Fsection_{\lambda_m}}} $ whose restriction to $\Dcal \cap \ZeroFsection$ is smooth and transverse to each $Z_j \,$. Thus for each $m \in\N$ we can choose a smooth map $e_{\lambda_m} : S_{\Fsection_{\lambda_m}} \rightarrow M$ so that
    \begin{itemize}
        \item $e_{\lambda_m} = e$ on $\Dcal \cap \ZeroFsection$
        \item both $e_{\lambda_m}$ and $e_{\lambda_m}|_{\partial S_{\Fsection_{\lambda_m}}}$ are transverse to each $Z_j \comma j=1,\ldots , k \,$.
        \item $d_M(e, e_{\lambda_m}) <  \delta_m < \delta$ on $S_{\Fsection_{\lambda_m}} \,$.
    \end{itemize}
    So properties \ref{it: evaluation maps - distance}\,-\,\ref{it: evaluation maps - transverse} above hold for each $(\lambda_m, e_{\lambda_m})$ and it remains to find an $m \in \N$ so that also property \ref{it: evaluation maps - empty preimage} holds for $(\lambda_m, e_{\lambda_m})$. We assume by contradiction that no such $m$ exists. Then, up to extracting a subsequence, there exist $i \in \{1,\ldots,\ell\}$ and $j \in \{1,\ldots,k\}$ so that
    \begin{equation}
    \label{eq: evaluation maps - empty preimage proof}
    \set{x \in \Dcal_{i} \cap \ZeroFsection}{e(x) \in \overline{Z_{j}}} = \varnothing
    \end{equation}
    and moreover $(e_{\lambda_m})^{-1}(\,\overline{Z_{j}}) \cap \Dcal_{i} \not= \varnothing $ for all $m \,$. For each $m$ choose 
    \[
    x_m \in (e_{\lambda_m})^{-1}(\,\overline{Z_{j}}) \cap \Dcal_{i} \subseteq S_{\Fsection_{\lambda_m}} \cap \Dcal_i \,\, .
    \]
    Notice that $\Funiv(x_m,\lambda_m) = \Fsection_{\lambda_m}(x_m)  = 0 \,$. Since we suppose that \ref{it: perturbation rmk - compactness} in Remark \ref{rmk: perturbation thm} holds, the restricted zero set $S_{\Funiv} \cap (\Bcal \times \ballclosed_\eps^n) $ is compact. Thus, up to a further subsequence, we may assume that $(x_m,\lambda_m)_m$ converges. Since $\Dcal_i$ is closed and $(\lambda_m)_m$ tends to the origin, the limit of $(x_m,\lambda_m)_m$ must be of the form $(x,0) \in S_{\Funiv}$ with $x \in \Dcal_i \,$. We have $x \in \ZeroFsection$ since $(x,0) \in S_{\Funiv}$ and $\Funiv(\,\cdot\,,0) = \Fsection \,$.
    \begin{claim}
        $e(x)$ lies in the closure of $Z_j \,$.
    \end{claim}
    \noindent Assuming the claim, $x \in \Dcal_i \cap \ZeroFsection$ satisfies $e(x) \in \overline{Z_j} \,$, contradicting \eqref{eq:  evaluation maps - empty preimage proof}. This means there exists an $m \in \N$ so that properties \ref{it: evaluation maps - distance}\,-\,\ref{it: evaluation maps - empty preimage} hold for $(\lambda_m, e_{\lambda_m})\,$, which finishes the proof of the proposition.
    \begin{proofClaim}
        Since $e_{\lambda_m}(x_m) \in \overline{Z_j}$ for every $m \,$, it suffices to show that $e_{\lambda_m}(x_m)$ converges to $e(x) \,$. To this end, recall that we have chosen the functions $e_{\lambda_m} : S_{\Fsection_{\lambda_m}} \rightarrow M$ so that $d_M(e,e_{\lambda_m}) < \delta_m $ on $S_{\Fsection_{\lambda_m}} \,$. We now estimate 
        \begin{align*}
        d_M(e(x), e_{\lambda_m}(x_m)) &\leq d_M(e(x), e(x_m)) + d_M(e(x_m), e_{\lambda_m}(x_m)) \\
        &\leq d_M(e(x),e(x_m)) + \delta_m \,\, .
        \end{align*}
        The first term on the right-hand side tends to zero by continuity of $e$ and $\delta_m$ tends to zero as well,
        hence so does $d_M(e(x),e_{\lambda_m}(x_m)) \,$. 
    \end{proofClaim}
\end{proof}

\begin{figure}
    \centering
    \tikzset{
  symbol/.style={
    draw=none,
    every to/.append style={
      edge node={node [sloped, allow upside down, auto=false]{$#1$}}}
  }
}

\begin{tikzcd}[row sep = small, column sep = tiny]
 |[xshift=1ex]| \mathcal{E}^1 \ar[r, hook] \ar[dddr, "\pi^1"' {yshift=1ex}] & \mathcal{E} \ar[ddd,"\pi"' {yshift=1ex}] &  &  & \\
   &  &  & &   \\ & & & & \\
  \mathcal{D} \ar[r, symbol = \subseteq]  & \mathcal{B} \ar[rrr, "e"] \ar[uuu, bend right, "\mathcal{F}_\lambda"'] &  &  &  M \supseteq Z_j \\
  \mathcal{D} \cap S_{\mathcal{F}_\lambda} \ar[u, symbol=\subseteq] \ar[r, symbol=\subseteq]   &  S_{\mathcal{F}_\lambda} \ar[u, symbol=\subseteq] \ar[urrr, dashed, end anchor = south west, shorten = -2mm , "e_{\lambda}"' {near start, yshift=0.25ex }, "\pitchfork Z_j"  {rotate=25 ,  xshift=2ex ,yshift=-0.25ex}] &  &  & 
\end{tikzcd}
    \caption{Situation in Proposition \ref{prop: perturbed evaluations}. The upper triangle with corners $\Ecal, \Ecal^1, \Bcal$ and the lower trapezoid with corners $\Dcal, \, \Dcal \cap S_{\Fsection_\lambda} , \, S_{\Fsection_\lambda} , \, M$ commute.}
    \label{diagram: perturbing evaluation map} 
\end{figure}

\subsection{Cobordisms and Evaluation Maps}
\label{subsec: evaluation maps - cobordisms}

We will apply Proposition \ref{prop: perturbed evaluations} to derive our last main result, Theorem \ref{mthm: evaluation maps - cobordism}, below.
Suppose we are given
\begin{itemize}
    \item a Banach manifold $\Bcal$ without boundary which admits smooth bump functions,
    \item a B-bundle pair $(\Ecal,\Ecal^1)$ over $\Bcal \,$,
    \item a real number $R > 0 \,$; We write $\BcalHat := [0,R] \times \Bcal$ and $\EcalHat$ for the pullback of $\Ecal$ along the projection $\BcalHat \rightarrow \Bcal $ to the second factor.
    \item a Fredholm section $\Fsection : \BcalHat \rightarrow \EcalHat$ of index $d + 1$ with compact zero set $\ZeroFsection\,$; We also write $\Fsection_r(x) := \Fsection(r,x)$ for $(r,x) \in [0,R] \times \Bcal$ and assume that $\Fsection_0 : \Bcal \rightarrow \Ecal$ is transverse to the zero section.
    \item a finite-dimensional manifold $M$ without boundary and a submanifold $Z \subseteq M$ without boundary that is a closed subset of $M \,$,
    \item a continuous map $e : \BcalHat \rightarrow M$ so that, writing $e_r(x) := e(r,x)$ for $(r,x) \in [0,R] \times \Bcal \,$, the map $e_0 : \Bcal \rightarrow M$ is smooth.
\end{itemize}
Notice that, due to the assumed transversality of $\Fsection_0 \,$, the zero set $S_{\Fsection_0}$ is a (possibly empty) compact $d$-dimensional submanifold of $\Bcal $ without boundary. The map $e_0$ then restricts to a smooth map $e_0 : S_{\Fsection_0} \rightarrow M \,$. If $e_0|_{S_{\Fsection_0}}$ is transverse to $Z \,$, then the preimage $(e_0|_{S_{\Fsection_0}})^{-1}(Z)$ is a compact submanifold of $S_{\Fsection_0}$ without boundary of dimension $d - \mathrm{codim}_M(Z) \,$. So, if $Z$ is additionally of codimension $d$, the preimage consists of finitely many points.

\begin{MainThm}
\label{mthm: evaluation maps - cobordism}
Suppose that $\mathrm{codim}_M(Z) = d$ and that $e_0 |_{S_{\Fsection_0}} : S_{\Fsection_0} \rightarrow M$ is transverse to $Z$ with $\# \, (e_0|_{S_{\Fsection_0}})^{-1}(Z) \in \N_0$ odd. Then there exists $x \in S_{\Fsection_R}$ with $e_R(x) \in Z \,$.
\end{MainThm}

\begin{proof}
    Notice that the assumption that $\#\,(e_0|_{S_{\Fsection_0}})^{-1}(Z)$ is an odd number in particular means that $S_{\Fsection_0}$ is non-empty.
    We apply the perturbation theorem to the Banach manifold with boundary $\BcalHat = [0,R] \times \Bcal \,$. Consider its boundary components $\Dcal := \{0\} \times \Bcal$ and $\Dcal_1 := \{R\} \times \Bcal \,$. By assumption, $\Fsection_0 = \Fsection|_{\Dcal} : \Dcal \rightarrow \EcalHat|_{\Dcal}$ is transverse to the zero section, $\Dcal \cap \ZeroFsection = \{0\} \times S_{\Fsection_0}$ is non-empty and the restriction $e|_{\Dcal \cap \ZeroFsection} = e_0|_{S_{\Fsection_0}}$ is transverse to $Z \,$. Consequently, the setting is as described in the previous Subsection \ref{subsec: evaluation maps - Existence Perturbed Evaluation Maps}.
    From Theorem \ref{mthm: perturbation thm} and Proposition \ref{prop: perturbed evaluations} we deduce the existence of a compact $(d+1)$-dimensional submanifold $\widetilde{S}$ of $[0,R] \times \Bcal$ with boundary
    \begin{align*}
      \partial \widetilde{S} = \widetilde{S} \cap (\{0,R\} \times \Bcal) = (\widetilde{S} \cap \Dcal) \sqcup (\widetilde{S} \cap \Dcal_1) = (\{0\} \times S_{\Fsection_0}) \sqcup (\widetilde{S} \cap \Dcal_1)  
    \end{align*}
    and a smooth map $\widetilde{e} : \widetilde{S} \rightarrow M$ satisfying
\begin{itemize}
    \item $\widetilde{e} = e$ on $\{0\} \times S_{\Fsection_0} \,$.
    \item Both $\widetilde{e}$ and $\widetilde{e} \,|_{\partial \widetilde{S}}$ are transverse to $Z \,$.
    \item If there does not exist $(R,x) \in S_{\Fsection}$ with $e(R,x) \in Z \,$, then $\widetilde{e}^{\,-1}(Z) \cap \Dcal_1 = \varnothing \,$.
\end{itemize}
Due to the transversality, the preimage $\widetilde{e}^{\,-1}(Z) \subseteq \widetilde{S}$ is a submanifold of dimension $\dim(\widetilde{S}) - \mathrm{codim}_M(Z) = 1$ and with boundary
\begin{align*}
   \partial (\widetilde{e}^{\,-1}(Z)) &= \widetilde{e}^{\,-1}(Z) \cap \partial \widetilde{S} = ( e^{-1}(Z) \cap (\{0\} \times S_{\Fsection_0}) ) \sqcup (\widetilde{e}^{\,-1}(Z) \cap \Dcal_1) \\
   &\cong (e_0|_{S_{\Fsection_0}})^{-1}(Z) \sqcup (\widetilde{e}^{\,-1}(Z) \cap \Dcal_1) \,\, ,
\end{align*}
where we used the first bullet point for the second equation. By continuity, $\widetilde{e}^{\,-1}(Z)$ is a closed subset of $\widetilde{S}$ and therefore compact itself. As compact $1$-dimensional manifold, $\widetilde{e}^{\,-1}(Z)$ has an even number of boundary points. So, from the assumption that $(e_0|_{S_{\Fsection_0}})^{-1}(Z)$ consists of an odd number of points, we infer that $\widetilde{e}^{\,-1}(Z) \cap \Dcal_1$ is non-empty. Thus, by the last bullet point, there exists $x \in S_{\Fsection_R}$ with $e(R,x) \in Z \,$. 
\end{proof}

\begin{remark}
\label{rmk: evaluation maps - cobordism}
More generally, instead of requiring that $Z$ has codimension $d$ and that $(e_0|_{S_{\Fsection_0}})^{-1}(Z)$ consists of an odd number of points, we could assume that the submanifold $(e_0|_{S_{\Fsection_0}})^{-1}(Z)$ is not cobordant to the empty set. Then the conclusion of Theorem \ref{mthm: evaluation maps - cobordism} continues to hold.
\end{remark}

\subsubsection{Illustration}
\label{subsubsec: evaluation maps - cobordisms illustration}

We briefly illustrate how to apply Theorem \ref{mthm: evaluation maps - cobordism} by giving a slightly modified proof of Lemma \ref{lemma: moduli_R non-empty} (the key step in the proof of Schwarz's theorem). Subsequently, we will work in the setting of Section \ref{sec: Application Schwarz thm}. We have fixed a point $q_0 \in M \backslash \mathrm{Fix}(\Phi_H^1) \,$.
Consider
\[
\Bcal := \set{u \in W^{1,p}(\bbS^2,M)}{u \text{ null-homotopic}} \,\, ,
\]
which is a union of connected components of $W^{1,p}(\bbS^2,M) \,$, one for each component of $M$. (The above definition of $\Bcal$ only differs from the Banach manifold defined in \eqref{eq: Bcal def Schwarz thm} in that we do not impose a condition on the evaluation of $u$ at $(0,0)$.) Just as before, we consider the bundle $\Ecal$ over $\Bcal$ with fiber $\Ecal_u := L^p(\bbS^2,\Lambda^{0,1} \otimes_{\C} u^*TM) $ 
and note that we can extend $\Ecal$ to a B-bundle pair $(\Ecal,\Ecal^1) \,$. For fixed $R > 0$ we consider the section
\[
\Fsection : [0,R] \times \Bcal \rightarrow \EcalHat_{[0,R]} \comma \quad \Fsection(r,u) = \Fsection_r(u) = \overline{\partial}_J \, u - (\beta_r \, dt \otimes X_{H_t}(u))^{(0,1)} \,\, .
\]
By Lemma \ref{lemma: D_u Fredholm + index}, this is a Fredholm section of index $2m +1 = \dim(M) +1 \,$. Its zero set $\ZeroFsection$ is compact (analogous to Proposition \ref{prop: moduli space compact}) and $\Fsection_0$ is transverse to the zero section (follows from Proposition \ref{prop: Fsection_0 transverse to zero section}). We work with the smooth evaluation map $e : [0,R] \times \Bcal \rightarrow M$ defined by $e(r,u) := \ev(u) = u(0,0) \,$. Due to the symplectic asphericity of $(M,\omega)$, the zero set $S_{\Fsection_0}$ consists precisely of the constant spheres $u : \bbS^2 \rightarrow M $. Hence the restriction $\ev|_{S_{\Fsection_0}}$ is a diffeomorphism onto $M$. Its inverse is the smooth map $M \rightarrow \Bcal \,$, sending $q \in M$ to the constant sphere $u \equiv q \,$, restricted in the codomain to the submanifold $S_{\Fsection_0} \,$. This shows that $\ev|_{S_{\Fsection_0}}$ is transverse to $Z:= \{q_0\}$ and that $(\ev|_{S_{\Fsection_0}})^{-1}(Z) = \{u_0\} $ consists of a single point. By Theorem \ref{mthm: evaluation maps - cobordism}, there exists $u \in S_{\Fsection_R}$ with $u(0,0) = \ev(u) = q_0 \,$. In other words, the moduli space $\moduli_R$ is non-empty.

\appendix

\section{Lemmata from (Non-)Linear Functional Analysis}
\label{app sec: Lemmata from Functional Analysis}

\begin{lemma}
\label{app lemma: Fredholm finite dim summands}
Let $\bbmE$ and $\bbmF$ be Banach spaces and $V$ and $W$ be finite-dimensional vector spaces. Given a bounded linear operator $A = (A_{1},A_2) : \bbmE \oplus V \rightarrow \bbmF \oplus W$ with components $A_1 : \bbmE \oplus V \rightarrow \bbmF$ and $A_2 : \bbmE \oplus V \rightarrow W \,$. Then $A$ is Fredholm if and only if $A_1|_{\bbmE} : \bbmE \rightarrow \bbmF$ is Fredholm. If this is the case, then their Fredholm indices are related by
\[
\mathrm{ind}(A) = \mathrm{ind}(A_1|_{\bbmE}) + \dim(V) - \dim(W) \,\, .
\]
\end{lemma}

\begin{proof}
    Write $A$ in the form
    \[
    A = \left( \begin{matrix}
        A_1|_{\bbmE} & A_1|_{V} \\
        A_2|_{\bbmE} & A_2|_{V}
    \end{matrix} \right) =  \left( \begin{matrix}
        A_1|_{\bbmE} & 0 \\
        0 & 0
    \end{matrix} \right) +  \left( \begin{matrix}
        0 & A_1|_{V} \\
        A_2|_{\bbmE} & A_2|_{V}
    \end{matrix} \right) =: \widetilde{A} + K
    \]
    and observe that $K$ is a bounded linear operator of finite rank and hence compact. Thus $A$ is Fredholm iff $\widetilde{A}$ is and, if this is the case, their Fredholm indices coincide. Moreover $\ker(\widetilde{A}) = \ker(A_1|_{\bbmE}) \oplus V$ and 
    \[
    \mathrm{coker}(\widetilde{A}) = (\bbmF \oplus W) / (\image(A_1|_{\bbmE}) \oplus \{0\}) \cong \mathrm{coker}(A_1|_{\bbmE}) \oplus W \,\, .
    \]
\end{proof}

\begin{lemma}
\label{app lemma: spanning cokernel open condition}
Let $X$ be a topological space and $\bbmE$ and $\bbmF$ Banach spaces. Given a continuous map $F : X \rightarrow \mathcal{L}(\bbmE,\bbmF)$ into the space of bounded linear operators (with the operator norm). Let $s_1,\ldots , s_n : X \rightarrow \bbmF$ be finitely many continuous maps. For given $x_0 \in X$ suppose that $s_1(x_0), \ldots , s_n(x_0)$ span the cokernel of $F(x_0) \,$. Then there exists an open neighborhood $U \subseteq X$ of $x_0$ so that $s_1(x),\ldots , s_n(x)$ span the cokernel of $F(x)$ for every $x \in U \,$.
\end{lemma}

\begin{proof}
    Consider the continuous map $G : X \rightarrow \mathcal{L}(\bbmE \oplus \R^n, \bbmF)$ given by
    \begin{align*}
    G(x)(e,\lambda) := F(x) (e) + \sum_{i=1}^n \lambda_i s_i(x) \quad \text{ for } x \in X \comma e \in \bbmE \comma \lambda = (\lambda_1,\ldots , \lambda_n ) \in \R^n \,\, .
    \end{align*}
    Notice that $s_1(x),\ldots , s_n(x)$ span the cokernel of $F(x)$ if and only if $G(x)$ is a surjective linear operator. It is a well-known fact that the subset in $\mathcal{L}(\bbmE \oplus \R^n,\bbmF)$ consisting of surjective operators is open, see \cite[Ch.~XV Thm.~3.4]{Lang_Real_Analysis}. The preimage of this subset under $G$ is the desired open neighborhood $U \subseteq X$ of $x_0 \,$.
\end{proof}

\noindent The next lemma is a special case of \cite[Thm.~A.4.3]{J_holomorphic_curves}.

\begin{lemma}[\cite{J_holomorphic_curves}, Thm.~A.4.3]
\label{app lemma: linearize C^1 function with Fredholm differential}
   Let $\bbmE $ and $\bbmF$ be Banach spaces and $U \subseteq \bbmE$ an open neighborhood of the origin $0 \in \bbmE \,$. Let $f : U \rightarrow \bbmF$ be a $C^1$-map with $f(0)= 0$ for which the differential $A := Df(0) \in \mathcal{L}(\bbmE,\bbmF)$ is a surjective Fredholm operator. Then there exists an open neighborhood $W \subseteq \bbmE$ of the origin and a $C^1$-diffeomorphism $g : W \rightarrow g(W)$ onto the open subset $g(W) \subseteq U$ with 
   \[
   (f \circ g ) (x) =  A x \quad \Forall x \in W
   \]
   and moreover $g(0) = 0$ and $Dg(0) = \mathbbm{1} \,$.
\end{lemma}

\begin{proof}
    Choose a complement $\bbmE' \subseteq \bbmE$ of $\ker(A)$ and set $T := (A|_{\bbmE'})^{-1} : \bbmF \rightarrow \bbmE \,$, so that $A T = \mathbbm{1}_{\bbmF} \,$. The map $\psi : U \rightarrow \bbmE$ defined by
    \[
    \psi(x) := x + T(f(x) - Ax)
    \]
    satisfies $\psi(0) = 0$ and has differential $D \psi(0) = \mathbbm{1}\,$. By the inverse function theorem, there exist open neighborhoods $V \subseteq U$ and $W \subseteq \bbmE$ of the origin so that $\psi : V \rightarrow W$ is a $C^1$-diffeomorphism. Now set $g := \psi^{-1}$ and notice that
    \[
    (A \circ \psi ) (x) = Ax + AT(f(x) - Ax) = Ax + f(x) - Ax = f(x) \qquad \Forall x \in V \,\, .
    \]
\end{proof}

\noindent It is known that Fredholm maps are locally proper, see \cite[Thm.~1.6]{Sard-Smale}, and the next lemma can be regarded a reformulation of this fact.

\begin{lemma}
\label{app lemma: non-linear Fredholm map}
Let $\bbmE $ and $\bbmF$ be Banach spaces and $U \subseteq \bbmE$ be an open neighborhood of the origin $0 \in \bbmE \,$. Let $f : U \rightarrow \bbmF$ be a $C^1$-map for which the differential $D f(0) \in \mathcal{L}(\bbmE,\bbmF)$ is a Fredholm operator. Then there exists an open neighborhood $V \subseteq U$ of $0$ so that for every sequence $(x_n)_{n \in \N} \subseteq V$ it holds
\[
(f(x_n))_n \text{ converges in } \bbmF \quad \Longrightarrow \quad  \text{a subsequence of } (x_n)_n \text{ converges in } \bbmE \,\, .
\]
\end{lemma}

\begin{proof} We show the lemma in increasing generality.

\textit{Step 1:} First we show the special case in which 
$f = A|_U : U \rightarrow \bbmF$ is the restriction of a surjective Fredholm operator $A$ to an open neighborhood $U$ of the origin. Choose a complement $\bbmE' \subseteq \bbmE$ of $K:= \ker(A) $ and write $\mathrm{pr}_{\bbmE'} : \bbmE \rightarrow \bbmE'$ and $\mathrm{pr}_{K} : \bbmE \rightarrow K$ for the corresponding projections. The restriction $A|_{\bbmE'} : \bbmE' \rightarrow \bbmF$ is invertible. Choose any open neighborhood $V \subseteq U $ that is bounded in $\bbmE \,$. Given an arbitrary sequence $(x_n)_n \subseteq V$ for which $(f(x_n))_n = (Ax_n)_n$ converges in $\bbmF \,$. Write $x_n = x_n' + y_n$ with $x_n' = \mathrm{pr}_{\bbmE'} (x_n) $ and $y_n := \mathrm{pr}_{K}(x_n) \,$.
Then $x'_n = (A|_{\bbmE'})^{-1} \circ A x_n$ and hence $(x_n')_n$ converges.
Because $K = \ker(A)$ is finite-dimensional and $(y_n)_n$ has bounded norm, a subsequence of $(y_n)_n$ will converge. Consequently also $(x_n)_n = (x_n' + y_n)_n$ has a convergent subsequence.  

\textit{Step 2:} We now show the lemma under the additional assumption that $f(0)=0$ and $A := Df(0)$ is a surjective Fredholm operator. By Lemma \ref{app lemma: linearize C^1 function with Fredholm differential} there exist an open neighborhood $W \subseteq \bbmE$ of the origin and a $C^1$-diffeomorphism $g : W \rightarrow g(W) $ so that $g(W) \subseteq U$ is an open neighborhood of the origin with $g(0) = 0 \comma Dg(0) = \mathbbm{1}$ and so that 
\[
(f \circ g) (x) = A x \qquad \Forall x \in W \, .
\]
By Step 1 the conclusion of the lemma holds for $f \circ g$ and thus it also holds for $f$.

\textit{Step 3:} We show the lemma in the general case. Set $A := Df(0) $ and
choose a splitting $\bbmF = \mathrm{im}(A) \oplus \bbmF' \,$. We write $\mathrm{pr}_{\mathrm{im}(A)} : \bbmF \rightarrow \mathrm{im}(A)$ for the projection. Then 
\[
\widetilde{f} : \, U \longrightarrow \mathrm{im}(A) \comma \quad \widetilde{f}(x) := \mathrm{pr}_{\mathrm{im}(A)}(f(x) - f(0)) \,\, ,
\]
is of class $C^1$, having as derivative at the origin the surjective Fredholm operator $D \widetilde{f}(0) = A : \bbmE \rightarrow \mathrm{im}(A) \,$. By Step 2 there exists an open neighborhood $V \subseteq U$ of the origin so that every sequence $(x_n)_n \subseteq V \,$, for which $(\widetilde{f}(x_n))_n$ converges, has a convergent subsequence. Now, if $(x_n)_n \subseteq V$ is a sequence so that $(f(x_n))_n $ converges, then clearly also $(\widetilde{f}(x_n))_n$ converges, hence $(x_n)_n$ has a convergent subsequence.
\end{proof}

\begin{remark}
\label{app rmk: non-linear Fredholm map}
Step 3 in the proof above shows that in the lemma it suffices to assume that $Df(0)$ has finite-dimensional kernel and a closed complemented image.
\end{remark}

\section{Constructing B-Bundle Pairs}
\label{app sec: B-Bundle Pair Construction}

\noindent In this appendix we first describe a way to extend a given Banach bundle $\Ecal$ to a B-bundle pair $(\Ecal,\Ecal^1)$. Afterwards we apply this to the setting of Subsection \ref{subsec: Schwarz thm - perturbation thm application}.

\subsection{B-Bundle Pair Construction Lemma}
\label{app subsec: B-Bundle Pair Construction lemma}

Observe that in Corollary \ref{mcor: perturbation thm} and Theorem \ref{mthm: evaluation maps - cobordism} the bundle $\Ecal^1$ is only mentioned in the assumptions but not in the conclusion. To be able to apply these results to a bundle $\Ecal$, it is therefore of interest to know when one can extend $\Ecal$ to a B-bundle pair $(\Ecal,\Ecal^1)\,$. 

\begin{example}
\label{app example: finite rank B-Bundle pair}
If $\Ecal \rightarrow \Bcal$ is a Banach bundle of finite rank, then $(\Ecal,\Ecal)$ is a B-bundle pair.
\end{example}

\noindent The next lemma occasionally enables one to construct a vector bundle structure on a suitable candidate $\Ecal^1 \,$. Although it is trivial, it is useful to once spell it out explicitly.

\begin{lemma}[B-Bundle Pair Construction]
\label{app lem: B-Bundle Construction}
Given a Banach bundle $\pi : \Ecal \rightarrow \Bcal \,$. Suppose we are given the following data
\begin{itemize}
    \item for every $x \in \Bcal$ a linear subspace $\Ecal^1_x \subseteq \Ecal_x \,$,
    \item a smooth bundle atlas $\mathcal{A} := \{\Phi^{(i)} : \Ecal|_{\Ucal_i} \rightarrow \Ucal_i \times \bbmF^{(i)}\}_{i \in I}$ for $\Ecal$ in the given bundle atlas equivalence class,
    \item for each $i \in I$ a B-pair $(\bbmF^{(i)}, \bbmF^{(i)}_1)$ with $\Phi^{(i)}_x (\Ecal_x^1) = \bbmF^{(i)}_1$ for every $x \in \Ucal_i \,$.
\end{itemize}
Suppose moreover that for every $i,j \in I$ and all $x \in \Ucal_i \cap \Ucal_j$ the transition map $ \Phi_x^{(j)} \circ (\Phi_x^{(i)})^{-1} \in \mathcal{L}(\bbmF^{(i)},\bbmF^{(j)})$ restricts to a bounded linear map $\bbmF^{(i)}_1 \rightarrow \bbmF^{(j)}_1$ and that the induced map
\begin{align*}
 \Ucal_i \cap \Ucal_j \rightarrow \mathcal{L}(\bbmF^{(i)}_1, \bbmF^{(j)}_1) \comma \quad x \mapsto \Phi_x^{(j)} \circ (\Phi_x^{(i)})^{-1}|_{\bbmF_{1}^{(i)}}
\end{align*}
is smooth. Then there exists a unique topology and smooth manifold structure on $\Ecal^1 := \bigcup_{x \in \Bcal} \Ecal_x^1$ so that $\mathcal{A}^1 := \{ \Phi^{(i)} : \Ecal^1|_{\Ucal_i} \rightarrow \Ucal_i \times \bbmF^{(i)}_1\}_{i \in I} $ is a smooth bundle atlas for $\pi^1 := \pi|_{\Ecal^1} : \Ecal^1 \rightarrow \Bcal \,$. With the vector bundle structure on $\Ecal^1$ induced by $\mathcal{A}^1 \,$, the pair $(\Ecal,\Ecal^1)$ is a B-bundle pair.
\end{lemma}

\begin{proof}
    Existence and uniqueness of the smooth structure on $\Ecal^1$ are immediate consequences of the vector bundle construction lemma, as stated in \cite[Ch.~III, Prop.~1.2]{Lang_Differential_Geometry} for example. By definition of a B-bundle pair, with the vector bundle structure on $\Ecal^1$ induced by $\mathcal{A}^1 \,$, the pair $(\Ecal,\Ecal^1)$ is a B-bundle pair.
\end{proof}

\subsection{B-bundle Pair Construction in Subsection \ref{subsec: Schwarz thm - perturbation thm application}}
\label{app subsec: B-bundle pair Schwarz}

\noindent In this appendix we verify that the pair $(\Ecal,\Ecal^1) \,$, introduced in Subsection \ref{subsec: Schwarz thm - perturbation thm application}, is indeed a B-bundle pair. To this end we want to invoke Lemma \ref{app lem: B-Bundle Construction}.
In fact, we will show that the natural extension of $(\Ecal,\Ecal^1)$ over $W^{1,p}(\bbS^2,M)$, which we also denote by $(\Ecal,\Ecal^1)$, is a B-bundle pair. This extended B-bundle pair then restricts to a B-bundle pair over $\Bcal \,$.

In the following discussion, we presume the reader is familiar with spaces of Sobolev sections of finite-rank vector bundles, as described in  \cite[App.~B.1]{J_holomorphic_curves} and \cite[App.~A.3]{SFT_Wendl} for example.

Let us first introduce some notation.
By a \textit{vector bundle chart} we mean a trivialization composed with a chart for the base. A smooth $u_1 \in \Cinfty(\bbS^2,M)$ induces the inverse $\Psi_{u_1}$ of a vector bundle chart for $\Ecal$, as in the following commutative diagram
\begin{equation*}
\begin{tikzcd}
   \Ecal|_{\Ucal} \ar[d, "\pi"'] &&& \Vcal \times \mathrlap{ L^p(\bbS^2,\Lambda^{0,1} \otimes_{\C} u_1^*TM) } \ar[lll, "\Psi_{u_1}", "\cong"'] \ar[d, "\mathrm{pr}_1"] \\
  \Ucal &&& \Vcal   \ar[lll, "\exp_{u_1}", "\cong"']
\end{tikzcd}
\end{equation*}
where $\Ucal \subseteq \Bcal$ is an open neighborhood of $u_1$ and the diffeomorphic image under $\exp_{u_1}$ of an open neighborhood
\[
\Vcal \subseteq W^{1,p}(\bbS^2,u_1^*TM)
\]
of the origin. The map $\Psi_{u_1}$ is fiberwise defined as follows: For $\xi \in \mathcal{V} $ and $u := \exp_{u_1}(\xi) \,$, the isomorphism over the fiber of $\xi$ is
\[
(\Psi_{u_1})_{\xi} = \big (\mathrm{id}_{\Lambda^{0,1}} \otimes_{\C} \Phi_{u_1}(\xi) \big)_* \, : \,\, L^p(\bbS^2,\Lambda^{0,1} \otimes_{\C} u_1^*TM) \overset{\cong}{\longrightarrow} L^p(\bbS^2,\Lambda^{0,1} \otimes_{\C} u^*TM) = \Ecal_u \,\, ,
\]
that is, $(\Psi_{u_1})_\xi$ is the pushforward of $L^p$-sections along the bundle isomorphism $\mathrm{id}_{\Lambda^{0,1}} \otimes_{\C} \Phi_{u_1}(\xi) \,$, where the complex bundle isomorphism $\Phi_{u_1}(\xi) : u_1^*TM \rightarrow u^*TM$ has been defined in \eqref{eq: parallel transport along nabla^tilde} via $\widetilde{\nabla}$-parallel transport.

We point out that the parallel transport map just depends on a point $(z,v) \in u_1^*TM \,$. For this reason, for $(z,v) \in u_1^*TM$ sufficiently close to the zero section we also define
\[
\Phi_{u_1} (z,v) : T_{u_1(z)} M \overset{\cong}{\longrightarrow} T_{\exp_{u_1(z)}(v)} M
\]
to be the $\widetilde{\nabla}$-parallel transport along the path $[0,1] \rightarrow M \comma \tau \mapsto \exp_{u_1(z)}(\tau v) \,$. With this notation, we can now express $\Phi_{u_1}(\xi)$ fiberwise as $\Phi_{u_1}(\xi)_z = \Phi_{u_1}(z,\xi(z)) \,$.

\begin{lemma}
\label{app lem: Ecal^1_u in trivialization}
For every $u = \exp_{u_1}(\xi) \in \Ucal$ we have
\[
\Ecal^1_u = (\Psi_{u_1})_{\xi} \Big( W^{1,p}(\bbS^2,\Lambda^{0,1} \otimes_{\C} u_1^*TM) \Big) \,\, .
\]
\end{lemma}

\begin{proof}
  Fix a finite holomorphic atlas of $\{(\Omega_j, \psi_j)\}_j$ of $\bbS^2$ and a collection $\{\phi_j : TM|_{U_j} \rightarrow U_j \times \R^{2m}\}_j$ of complex trivializations for $(TM, J)$, \ie $\phi_j$ intertwines $J$ with the standard almost complex structure, so that $u_1(\overline{\Omega_j}) \subseteq U_j$ for every $j$ and so that each chart $(\Omega_j,\psi_j)$ and each trivialization $\phi_j$ can be extended over the closures of $\Omega_j$ respectively $U_j \,$. We denote the coordinate vector fields on $\Omega_j \subseteq \bbS^2$ of the chart $\psi_j$ by $\partial_s^j , \, \partial_t^j \,$. 

  Let $\eta \in \Ecal_u = L^p(\bbS^2,\Lambda^{0,1} \otimes_{\C} u^*TM)$ be an arbitrary element in the fiber over $u = \exp_{u_1}(\xi)$. Our goal is to show
  \begin{equation}
    \label{app eq: Ecal^1_u in trivialization goal}
    \eta \in  \Ecal^1_u \,\, \Longleftrightarrow \,\, (\Psi_{u_1})_{\xi}^{-1}(\eta) \in W^{1,p}(\bbS^2,\Lambda^{0,1} \otimes_{\C} u_1^*TM) \,\, .
  \end{equation}
  We may assume that $\eta$ is a continuous section of $\Lambda^{0,1} \otimes_{\C} u^*TM$ since otherwise both sides in \eqref{app eq: Ecal^1_u in trivialization goal} are false statements. By definition of $\Ecal^1_u$ (see Remark \ref{rmk: Ecal^1_u def Schwarz theorem}) we have
  \begin{align*}
      \eta \in \Ecal^1_u \,\, &\Longleftrightarrow \,\, (u,\eta(X)) \in W^{1,p}(\bbS^2,TM) \quad \Forall X \in \mathfrak{X}(\bbS^2) \\
      & \Longleftrightarrow \,\,  (\phi_j)_u \cdot \eta(\partial_s^j) \in W^{1,p}(\Omega_j,\R^{2m}) \,\,  \quad \Forall j \,\, ,
  \end{align*}
  where for the second equivalence we used that $u \in W^{1,p}(\bbS^2,M)$ and that $\eta$ is complex antilinear.
Since the holomorphic chart $(\Omega_j,\psi_j)$ and the complex trivialization $\phi_j$ over $U_j$ induce a trivialization for $\Lambda^{0,1} \otimes_{\C} u_1^*TM $ over $\Omega_j \,$, we similarly have
\begin{align*}
    &(\Psi_{u_1})_{\xi}^{-1} (\eta) \in  W^{1,p}(\bbS^2,\Lambda^{0,1} \otimes_{\C} u_1^*TM)  \\
       \Longleftrightarrow \,\,  &(\phi_j)_{u_1} \cdot (\Phi_{u_1}(\xi))^{-1} \cdot \eta(\partial_s^j) \in W^{1,p}(\Omega_j,\R^{2m}) \,\,  \quad \Forall j \,\, .
  \end{align*}
  \begin{claim}
    Let $j$ and $\zeta : \Omega_j \rightarrow \R^{2m}$ be arbitrary. Then
    \begin{align*}
        \zeta \in W^{1,p}(\Omega_j, \R^{2m}) \,\, \Longleftrightarrow \,\, (\phi_j)_{u_1} \cdot \Phi_{u_1}(\xi)^{-1} \cdot (\phi_j)_{u}^{-1} \cdot \zeta \in W^{1,p}(\Omega_j,\R^{2m}) \,\, . 
    \end{align*}
  \end{claim}
   \noindent Then \eqref{app eq: Ecal^1_u in trivialization goal} follows from the previously established equivalences and the claim with
  \[
  \zeta := (\phi_j)_u \cdot \eta (\partial_s^j) \,\, .
  \]
  \begin{proofClaim}
    Let us drop the index $j$ for readability.
    For $z \in \Omega$ define the automorphism
    \[
    A(z) := \phi_{u_1(z)} \cdot \Phi_{u_1}(z,\xi(z))^{-1} \cdot \phi_{\exp_{u_1(z)}(\xi(z))}^{-1} \in \mathrm{GL}(2m,\R) \subseteq \R^{2m \times 2m} \,\, .
    \]
    Then $A$ and $A^{-1}$ are of regularity $W^{1,p}$ since $\xi$ is.
    The claim follows because we have a well-defined product pairing
    \[
    W^{1,p}(\Omega,\R^{2m\times 2m}) \times W^{1,p}(\Omega,\R^{2m}) \rightarrow W^{1,p}(\Omega,\R^{2m}) \comma \quad (\widetilde{A}, \widetilde{\zeta}) \mapsto \widetilde{A} \cdot \widetilde{\zeta}
    \]
    since $p > 2 = \dim(\Omega)$.
  \end{proofClaim}

\end{proof}

\begin{lemma}
\label{app lem: (L^p,W^1,p) B-pair}
For every smooth finite-rank vector bundle $E$ over $\bbS^2$ the pair
\[
\big( L^p(\bbS^2,E) , \, W^{1,p}(\bbS^2,E) \big)
\]
is a B-pair.
\end{lemma}

\begin{proof}
    The inclusion $W^{1,p}(\bbS^2,E) \hookrightarrow L^p(\bbS^2,E)$ is compact since $\bbS^2$ is compact. Moreover $\Cinfty(\bbS^2,M)$ is dense in $L^p(\bbS^2,E)$, so $W^{1,p}(\bbS^2,E)$ is dense as well.
\end{proof}

\noindent Now, for smooth $u_1,u_2 \in \Cinfty(\bbS^2,M)$, we investigate the transition map $\Psi_{u_2}^{-1} \circ \Psi_{u_1}$ of the associated vector bundle charts when restricted to the fibers of $\Ecal^1 \,$. Consider the commutative diagram
\begin{equation*}
\begin{tikzcd}
    \mathcal{V}_{12} \times W^{1,p}(\bbS^2,\Lambda^{0,1} \otimes_{\C} u_1^*TM) \ar[r, "\Psi_{u_1}"] \ar[d, "\mathrm{pr}_1"'] & \Ecal^1|_{\Ucal_1 \cap \Ucal_2} \ar[d, "\pi" ] & \mathcal{V}_{21} \times W^{1,p}(\bbS^2,\Lambda^{0,1} \otimes_{\C} u_2^*TM) \ar[l, "\Psi_{u_2}"'] \ar[d, "\mathrm{pr}_1"]  \\
    \mathcal{V}_{12} \ar[r, "\cong", "\exp_{u_1}"'] & \Ucal_1 \cap \Ucal_2 & \mathcal{V}_{21} \ar[l, "\cong"', "\exp_{u_2}"]
\end{tikzcd}
\end{equation*}
The next lemma, which is a slight adaption of \cite[Thm.~6.1]{Eliasson}, asserts that the above transition map is fiberwise bounded linear and smoothly depending on the base point.

\begin{lemma}
\label{app lem: transition maps for Ecal^1 smooth}
Given $u_1,u_2 \in \Cinfty(\bbS^1,M) \,$, the vector bundle charts $\Psi_{u_1}$ and $\Psi_{u_2}$ restricted to $\Ecal^1$ have a smooth transition map, \ie there is a well-defined smooth map
\begin{align}
\Vcal_{12} &\rightarrow \mathcal{L}\Big( W^{1,p}(\bbS^2,\Lambda^{0,1} \otimes_{\C} u_1^*TM) \, , \,  W^{1,p}(\bbS^2,\Lambda^{0,1} \otimes_{\C} u_2^*TM)\Big) \label{app eq: transition map for Ecal^1}\\
\xi &\mapsto (\Psi_{u_2})_{\exp_{u_2}^{-1} \circ \exp_{u_1}(\xi)}^{-1} \circ (\Psi_{u_1})_{\xi} \,\, . \notag
\end{align}
\end{lemma}

\begin{proof}
    In the terminology of Eliasson \cite{Eliasson}, the \textit{section functor} $\mathfrak{S} := W^{1,p}(\bbS^2, \,\cdot\,)$ over the base manifold $\bbS^2$ is a \textit{manifold model}, that is $\mathfrak{S}$ is a covariant functor from the smooth finite-rank vector bundles over $\bbS^2$ to the category of Banachable spaces so that the following hold:
    \begin{enumerate}[label=\arabic*)]
        \item\label{app it: manifold model - C^0} $\mathfrak{S}(E) \subseteq C^0(\bbS^2,E)$ is a bounded linear inclusion for every smooth bundle $E$.
        \item\label{app it: manifold model - sec(L) < L(sec,sec)} For every smooth vector bundles $E,F$ the map
        \[
        \mathfrak{S}(\mathrm{Hom}(E,F)) \rightarrow \mathcal{L}(\mathfrak{S}(E), \mathfrak{S}(F)) \comma \quad A \mapsto A_* \,\, ,
        \]
        where $A_*$ is defined by $A_*(\eta) (z) := A(z) \cdot \eta(z) \,$, is well-defined and bounded linear.
        \item\label{app it: manifold model - fiber-preserving map}
        Let $E,F$ be vector bundles, $\mathcal{O} \subseteq E$ an open subset and $f : \mathcal{O} \rightarrow F$ be a smooth fiber-preserving map. Then the map
        \[
        \mathfrak{S}(f) : \,\, \mathfrak{S}(\mathcal{O}) = \set{\eta \in \mathfrak{S}(E)}{\eta(\bbS^2) \subseteq \mathcal{O}} \rightarrow \mathfrak{S}(F) \comma \quad \eta \mapsto f \circ \eta
        \]
        is well-defined and continuous.
    \end{enumerate}
    All of the above points essentially follow from $p > 2 = \dim(\bbS^2) $ and are explicitly asserted in \cite[Thm.~A.23]{SFT_Wendl}. Since $\mathfrak{S}$ is a manifold model, by \cite[Lemma 4.1]{Eliasson} the induced map $\mathfrak{S}(f)$ in point \ref{app it: manifold model - fiber-preserving map} is in fact smooth.

    Let $\mathcal{O}_j \subseteq u_j^*TM \comma j=1,2 \,$, be fiberwise convex zero section neighborhoods with $\mathfrak{S}(\mathcal{O}_j) = \mathcal{V}_j \,$. Define
    \[
    \mathcal{O}_{12} := \set{(z,v) \in \mathcal{O}_1}{\exp_{u_1(z)}(v) = \exp_{u_2(z)}(w) \text{ for some } (z,w) \in \mathcal{O}_2} \subseteq u_1^*TM \,\, ,
    \]
    hence $\mathfrak{S}(\mathcal{O}_{12}) = \mathcal{V}_{12} \,$.
    Consider the fiber-preserving map     
    \[
    f: \,\mathcal{O}_{12} \rightarrow \mathrm{Hom}(\Lambda^{0,1} \otimes_{\C} u_1^*TM , \, \Lambda^{0,1} \otimes_{\C} u_2^*TM ) \,\, ,
    \]
    defined by
     \[
    f(z,v) := \mathrm{id}_{\Lambda^{0,1}} \otimes_{\C} \Big( \Phi_{u_2}(z, \exp_{u_2(z)}^{-1} \circ \exp_{u_1(z)} (v))^{-1} \circ \Phi_{u_1}(z,v) \Big) \,\, .
    \]
    Since $f$ is smooth, the induced map
    \[
    \mathfrak{S}(f) : \, \mathcal{V}_{12} \rightarrow \mathfrak{S}\Big( \mathrm{Hom}(\Lambda^{0,1} \otimes_{\C} u_1^*TM , \, \Lambda^{0,1} \otimes_{\C} u_2^*TM ) \Big) 
    \]
    is smooth. Its composition with the bounded linear map from point \ref{app it: manifold model - sec(L) < L(sec,sec)} in the definition of a manifold model
    \[
    \mathfrak{S}\Big( \mathrm{Hom}(\Lambda^{0,1} \otimes_{\C} u_1^*TM , \, \Lambda^{0,1} \otimes_{\C} u_2^*TM ) \Big)  \longrightarrow \mathcal{L}\Big( \mathfrak{S}(\Lambda^{0,1} \otimes_{\C} u_1^*TM) , \, \mathfrak{S} (\Lambda^{0,1} \otimes_{\C} u_2^*TM)\Big)
    \]
    agrees with the transition map in \eqref{app eq: transition map for Ecal^1}, which is therefore smooth.
\end{proof}

\begin{corollary}
\label{app cor: (Ecal,Ecal^1) B-bundle pair}
The pair $(\Ecal,\Ecal^1)$ is a B-bundle pair over $W^{1,p}(\bbS^2,M) \,$.
\end{corollary}

\begin{proof}
    This follows from the B-bundle pair construction lemma (Lemma \ref{app lem: B-Bundle Construction}), whose assumptions have been checked in Lemmata \ref{app lem: Ecal^1_u in trivialization}, \ref{app lem: (L^p,W^1,p) B-pair} and \ref{app lem: transition maps for Ecal^1 smooth}.
\end{proof}

\printbibliography

\end{document}